\documentclass[notitlepage,reqno, 11pt]{article}
\usepackage[T1]{fontenc}
\usepackage[utf8]{inputenc}
\usepackage{lmodern}
\usepackage[affil-it]{authblk}
\usepackage{amsmath, amsthm, amssymb, amscd, mathrsfs, amsfonts, stmaryrd}
\usepackage{hyperref}
\usepackage{tikz}
\usetikzlibrary{matrix,arrows,decorations.pathmorphing}

\usepackage{mathbbol}
\usepackage{colonequals}
\usepackage{enumitem}

\usepackage{geometry}
\usepackage{graphicx}
\usepackage{color}

\usepackage{xcolor}
\usepackage{todonotes}

\usepackage[all]{xy}

\usepackage[lite]{amsrefs}
\newcommand*{\MRref}[2]{ \href{http://www.ams.org/mathscinet-getitem?mr=#1}{MR \textbf{#1}}}
\newcommand*{\arxiv}[1]{\href{http://www.arxiv.org/abs/#1}{arXiv: #1}}
\renewcommand{\PrintDOI}[1]{\href{http://dx.doi.org/\detokenize{#1}}{doi: \detokenize{#1}}%
  \IfEmptyBibField{pages}{, (to appear in print)}{}}

\def\commutatif{\ar@{}[rd]|{\circlearrowleft}}

\newcommand{\eq}[1][r]
   {\ar@<-3pt>@{-}[#1]
    \ar@<-1pt>@{}[#1]|<{}="gauche"
    \ar@<+0pt>@{}[#1]|-{}="milieu"
    \ar@<+1pt>@{}[#1]|>{}="droite"
    \ar@/^2pt/@{-}"gauche";"milieu"
    \ar@/_2pt/@{-}"milieu";"droite"}

\def\dar[#1]{\ar@<2pt>[#1]\ar@<-2pt>[#1]}
  \entrymodifiers={!!<0pt,0.7ex>+} 
  
  \newcommand{\bigon}[4][r]{
    \ar@/^1pc/[#1]^{#2}_*=<0.3pt>{}="HAUT"
    \ar@/_1pc/[#1]_{#3}^*=<0.3pt>{}="BAS"
    \ar@{<=} "HAUT";"BAS" _{#4}
  }

\newcommand{\bigons}[6][r]{  
    \ar@/^2pc/[#1]^{#2}_*=<0.3pt>{}="HAUT"
    \ar@{}    [#1]     ^*=<0.3pt>{}="MILIEUHAUT"
                       _*=<0.3pt>{}="MILIEUBAS"
    \ar[#1]_(0.3){#3}                  
    \ar@/_2pc/[#1]_{#4}^*=<0.3pt>{}="BAS"
    \ar@{<=} "HAUT";"MILIEUHAUT" _{#5}
    \ar@{<=} "MILIEUBAS";"BAS" _{#6}
  }

\newtheorem{thm}{Theorem}[section]
\newtheorem{pro}[thm]{Proposition}
\newtheorem{lem}[thm]{Lemma}
\newtheorem{dfpro}[thm]{Definition and Proposition}

\newtheorem{cor}[thm]{Corollary}

\theoremstyle{definition}
\newtheorem{df}[thm]{Definition}

\newtheorem{nota}[thm]{Notations}

\theoremstyle{remark}
\newtheorem{rmk}[thm]{Remark}

\newtheorem{ex}[thm]{Example}

\allowdisplaybreaks

\def\id{\operatorname{1}}

\def\Id{\operatorname{id}}

\def\<{\langle}
\def\>{\rangle}

\def\rTo{\longrightarrow}
\def\lTo{\longleftarrow}
\def\rrTo{\rightrightarrows}

\def\mto{\longmapsto}
\def\RTo{\Longrightarrow}
\def\mono{\hookrightarrow}
\def\epi{\twoheadrightarrow}

\def\xTo{\xymatrix{\ar@{-x}[r]&}}
\def\hoTo{\xymatrix{\ar@<-0.03mm>[r] \ar@<-0.02mm>[r] \ar@<0.02mm>[r]\ar@<0.03mm>[r] &}} 

\newcommand{\ou}[3]{\overset{#2}{\underset{#3}{#1}}}

\let\frak\mathfrak 
\let\cal\mathcal 

\def\a{\alpha}
\def\b{\beta}
\def\g{\gamma}
\def\G{\Gamma}
\def\d{\delta}

\def\lam{\lambda}
\def\vp{\varphi}
\def\s{\sigma}

\def\t{\theta}
\def\p{\partial}
\def\k{\kappa}

\newcommand{\XExt}{\operatorname{\bf XExt}}
\newcommand{\fXExt}{\operatorname{\frak{XExt}}}
\newcommand{\XMod}{\operatorname{\bf XMod}}

\newcommand{\Ggrd}{\cal G\rrTo \cal G^0} 
\newcommand{\Gamgrd}{\Gamma \rrTo \Gamma^0} 
\newcommand{\Hgrd}{\cal H\rrTo \cal H^0} 
\newcommand{\AutH}{\xymatrix{\Aut(\cal H) \dar[r] & \cal H^0}}

\newcommand{\HG}{\xymatrix{\cal H \ar@{-x}[r]^{\p} & \cal G}}
\newcommand{\HGOne}{\xymatrix{\cal H_1 \ar@{-x}[r]^{\p_1} & \cal G_1}}
\newcommand{\HGTwo}{\xymatrix{\cal H_2 \ar@{-x}[r]^{\p_2} & \cal G_2}}
\newcommand{\HGThree}{\xymatrix{\cal H_3 \ar@{-x}[r]^{\p_3} & \cal G_3}}

\newcommand{\exch}{\xymatrix{\ar@3{->}[r]&}}
\newcommand{\iexch}{\xymatrix{\ar@3{<->}[r]&}}

\newcommand{\Aut}{\operatorname{\bf Aut}}

\newcommand{\Iso}{\operatorname{\bf Iso}}

\newcommand{\Dprod}{ 
\begin{tikzpicture}
\draw (-0.13,0) -- (0,0.13) -- (0.13,0) -- (0,-0.13) -- cycle ;
\draw[black!30] (0,-0.13) -- (0,0.13);
\end{tikzpicture}
}

\newcommand{\tr}{ 
\begin{tikzpicture}
\draw[->, line width=2pt, >=latex] (0,0) -- (1,0);
\end{tikzpicture}
}

\newcommand{\cross}{
\begin{tikzpicture}
\draw[->, dashed, >=latex] (-0.5,-0.1) -- (0.5,0.1);
\draw[<-,>=latex] (-0.5,0.1) -- (0.5,-0.1);
\end{tikzpicture}
}

\newcommand{\xext}{
\begin{tikzpicture}
\draw[->, >=latex] (-0.5,-0.1) -- (0.5,0.1);
\draw[<-,>=latex] (-0.5,0.1) -- (0.5,-0.1);
\end{tikzpicture}
}

\newcommand{\sxc}{ 
\begin{tikzpicture}
\draw[->, double distance=2pt, thick, >=stealth] (0,0) -- (1,0);
\end{tikzpicture}
}

\newcommand{\xc}{ 
\begin{tikzpicture}
\draw[<->, double distance=2pt, thick, >=stealth] (0,0) -- (1,0);
\end{tikzpicture}
}

\tikzset{2ar/.style={double distance=2pt,thick,->,>=stealth}}
\tikzset{3ar/.style={double distance =3pt, thick,->, >=stealth}}
\tikzset{1ar/.style={->,>=stealth,thick}}
\tikzset{2li/.style={-, double distance=2pt, thick}}
\tikzset{crossar/.style={preaction={draw=white, -, line width=6pt}}}
\tikzset{descr/.style={fill=white},text height=1.5ex, text depth=0.25ex}

\hypersetup{
  colorlinks = true,
  urlcolor = blue,
  linkcolor = blue,
  citecolor = red,
  pdfauthor = {Moutuou, E.},
  pdfkeywords = {crossed modules, crossed extensions, 2-extensions, groupoids, groupoid cohomology, Morita equivalence, higher categories},
  pdftitle = {Moutuou, E. - Crossed extensions},
  pdfsubject = {groupoid cohomology},
  pdfpagemode = UseNone
}

\begin{document}

\title{Crossed extensions and equivalences of topological $2$--groupoids}

\author{El-ka\"ioum M. Moutuou%
    \thanks{E-mail: \texttt{elkaioum@moutuou.net}}}

\maketitle 

\begin{abstract}
We provide concrete models for generalized morphisms and Morita equivalences of topological $2$--groupoids by introducing the notions of crossings and crossed extensions of groupoid crossed modules. A systematic study of these objects is elaborated and an explicit description of how they do yield a groupoid and geometric picture of weak $2$--groupoid morphisms is presented. Specifically, we construct a weak 3-category whose objects are crossed modules of topological groupoids and in which weak $1$--isomorphisms correspond to Morita equivalences in the "category" of topological $2$--groupoids. 
\end{abstract}

\tableofcontents

\section{Introduction}

A {\em $2$--groupoid} is a strict $2$--category (\cite{Baez:intro-nCat}) ${\bf G}$ in which all $1$--arrows and $2$--arrows are invertible. It is the data of three sets ${\bf G}^0, {\bf G}^1, {\bf G}^2$ of {\em units, $1$--morphisms} and {$2$--morphisms}, respectively, {\em structure maps} $${\bf G}^2\ou{\rrTo}{s_2}{t_2}{\bf G}^1\ou{\rrTo}{s}{t}{\bf G}^0,$$ a partial product and an inversion map on $1$--morphisms making ${\bf G}^1\ou{\rrTo}{s}{t}{\bf G}^0$ a groupoid, and two partial products $\star_h$ and $\star_v$ respectively called {\em horizontal} and {\em vertical} compositions on $2$--morphisms, all verifying some coherence relations. Moreover, ${\bf G}^2\ou{\rrTo}{s_2}{t_2}{\bf G}^1$ is a groupoid with respect to the operation $\star_v$. ${\bf G}$ is a {\em topological $2$--groupoid} if ${\bf G}^0, {\bf G}^1, {\bf G}^0$ are topological spaces and all the structure maps as well as all the operations involved are continuous maps. For instance, every topological groupoid $\cal G\rrTo X$ can be seen as a topological $2$--groupoid with all identity $2$--morphisms. Homomorphisms (resp. isomorphisms) of topological $2$--groupoids are defined as strict $2$--functors whose components on the sets of units, $1$--morphisms, and $2$--morphisms, are all continuous maps (resp. homeomorphisms). Also, there is a well known notion of {\em equivalences} of higher categories (\cite{Lurie:HigherTopos}) which can easily be modified to adapt to topological $2$--groupoids (see Definition~\ref{def:weak_equiv_2-grpd}). For instance, suppose $\cal G\rrTo X$ is a topological groupoid and $\cal U=\{U_i\}_I$ is an open covering of $X$. We form the {\em cover $2$--groupoid} $\cal G[\cal U,2]\rrTo \cal G[\cal U]\rrTo \coprod_iU_i$ (see Example~\ref{ex:cover_2-gpd}) by taking 

\[
\cal G[\cal U,2]\colonequals \left\{(i_1,i_2,g,h,j_1,j_2)\in I^2\times \cal G^2\times I^2\mid t(g)=t(h)\in U_{i_1i_2}, s(g)=s(h)\in U_{j_1j_2}\right\}
\]
where $U_{ij}\colonequals U_i\cap U_j$, and $
\cal G[\cal U]\colonequals \{(i,g,j)\in I\times \cal G\times I\mid t(g)\in U_i, s(g)\in U_j\}$. Then, the canonical map $$\cal G[\cal U,2]\ni (i_1,i_2,g,h,j_1,j_2)\mto gh^{-1}g\in \cal G$$ defines an equivalence of topological $2$--groupoids.  

Unfortunately, such maps appear to be highly insufficient as a way to map $2$--groupoids to one another, for they rarely occur in practice. For instance, since (topological) groupoids are particular cases of (topological) $2$--groupoids, it is natural to expect generalized morphisms of groupoids (\cites{Hilsum-Skandalis:Morphismes, Moutuou:Real.Cohomology, Tu-Xu:Ring_structure}) to fit in some notion of {\em generalized morphisms} of topological $2$--groupoids. This problem was addressed by C. Zhu who proposed a formal definition of {\em generalized morphisms} of $n$--groupoids using a simplicial point of view (\cite[\S 2.2]{Zhu:Lie_n_stacky}). Precisely, she defined a generalized morphism from ${\bf G}$ to ${\bf H}$ to be a zig-zag of homomorphisms ${\bf G}\stackrel{\sim}{\lTo}{\bf G}'\rTo {\bf H}$ where the leftward one is a {\em hypercover} (Definition 2.3 in {\em ibid}.). We shall note that in the case of $2$--groupoids, Zhu's generalized morphisms correspond to what we have called {\em weak homomorphisms} in this work (Definition~\ref{df:weak_hom_2-grpd}). \\

One of the original goals of this work was an attempt to investigate a more concrete picture of what should be "generalized morphisms" and Morita equivalences of topological $2$--groupoids by laying out the whole problem on the framework of topological ($1$--)groupoids. We have reached this goal by using the well known one-to-one correspondence between $2$--groupoids and {\em crossed modules of topological groupoids}, which led us to develop a systematic study of these objects. Crossed modules were first defined by Whitehead (\cite{Whitehead:Combinatorial_homotopy_I}) for groups in order to study $3$--type spaces. Recently these gadgets have aroused much interest among mathematicians from various areas and have been used, for instance, in twisted $K$-theory (\cites{Tu-Xu:Ring_structure}), differential geometry (\cite{Ginot-Stienon:G-gerbes}), groupoid symmetries and actions on $C^*$--algebras (\cite{Buss-Meyer-Zhu:Non-Hausdorff_symmetries}), gr-stacks (\cite{Aldrovandi-Noohi:Butterflies}), etc.    \\

A {\em groupoid crossed module} (\cites{Ginot-Stienon:G-gerbes,Tu-Xu:Ring_structure}) ${\frak G}$ with unit space $X$ consists of a topological groupoid $\cal G\rrTo X$, a bundle of topological groups $\cal H\rTo X$, a strict homomorphism of topological groupoids $\cal H\stackrel{\p}{\rTo}\cal G$, and a groupoid action by automorphisms of $\cal G$ on $\cal H$; that is, for each arrow $x\stackrel{g}{\rTo}y$ in  $\cal G$, an isomorphism of topological groups $\cal H^{y}\ni h\mto h^g\in \cal H^x$, with the property that $\p(h^g)=g^{-1}\p(h)g$, and  $k^{\p(h)}=h^{-1}kh$ for  $h\in \cal H^{s(k)}$. In particular, a crossed module of topological groups (\cites{Whitehead:Combinatorial_homotopy_I, Holt:Cohomology, Noohi:two-groupoids}) is a groupoid crossed module with one-object unit space. Any topological groupoid $\cal G\rrTo X$ can be seen as a groupoid crossed module with $\cal H=X$, $\p$ being the inclusion map $X\rTo \cal G$, together with the trivial action of $\cal G$ on $X$. \\

One can "unfold" a topological $2$--groupoid ${\bf G}$ to get a groupoid crossed module ${\frak G}$ by taking $\cal G$ to be the topological groupoid ${\bf G}^1\rrTo {\bf G}^0$ and $\cal H$ to be the subspace of ${\bf G}^2$ consisting of all $2$--arrows $a\in {\bf G}^2$ whose source $s_2(a)$ are identity $1$--morphisms of the form $\id_x$ with $x\in {\bf G}^0$. The strict morphism $\p\colon \cal H\rTo {\bf G}^1$ is then given by the target map $t_2$, while the groupoid action by automorphisms of ${\bf G}^1$ on $\cal H$ is given by $a^g\colonequals \id_{g^{-1}}\star_ha\star_h\id_g$ for $g\in {\bf G}^1$ and $a\in \cal H^{t(g)}$. Conversely, given a groupoid crossed module, one forms the topological $2$--groupoid $\cal H\rtimes \cal G\ou{\rrTo}{s_2}{t_2}\cal G\rrTo X$ where $s_2(h,g)=g, t_2(h,h)=\p(h)g$, $(h_1,g_1)\star_h(h_2,g_2)=(h_1h_2^{g_1^{-1}},g_1g_1)$, and $(h,g)\star_v(k,\p(h)g)=(hk,g)$ (cf. \S~\ref{sec:cross_vs_2-grpd}). \\

Now, given two groupoid crossed modules $\frak G_1$ and $\frak G_2$, a {\em (strict) morphism} $\chi\colon \frak G_1\rTo \frak G_2$ is a commutative diagram of groupoid strict morphisms 
\[
\xymatrix{
\cal H_1\ar[r]^{\p_1}\ar[d]_{\chi} & \cal G_1\ar[d]^{\chi} \\
\cal H_2\ar[r]^{\p_2} & \cal G_2
}
\]
satisfying a certain compatibility condition related to the groupoid actions. Such morphisms clearly induce homomorphisms of the $2$--groupoids obtained from the above construction and vice-versa. Moreover, we have weakened this definition by introducing various notions including {\em transformations} between strict morphisms, {\em strong equivalences}, {\em Morita equivalences} of groupoid crossed modules, and {\em crosssings}, and shown that they encode morphisms and weak morphisms of topological $2$--groupoids. \\

More precisely, let $\frak G_1$ and $\frak G_2$ be groupoid crossed modules with unit spaces $X_1$ and $X_2$, respectively. We define a {\em crossing} $\frak G\underset{\cal M}{\cross}\frak G_2$ to be a topological groupoid $\cal M\rrTo \cal M^0$, together with two continuous maps 
\[
X_1\stackrel{\tau}{\lTo} \cal M^0 \stackrel{\s}{\rTo} X_2
\]
and a commutative diagram 
\[
\xymatrix{
\cal H_1[\cal M^0] \ar[rr]^{\tau^*\p_1} \ar[rd]_{\a_1} && \cal G_1[\cal M^0] \\
& \cal M \ar[ru]_{\a_2}  \ar[rd]^{\b_2} & \\
\cal H_2[\cal M^0] \ar[ru]^{\b_1} \ar[rr]^{\s^*\p_2} && \cal G_2[\cal M^0]
}
\]
satisfying a number of axioms (Definition~\ref{def:crossing}). For example, let $\cal G\rrTo X$ and $\Gamma\rrTo Y$ be topological groupoids viewed as the groupoid crossed modules $X\rTo \cal G$ and $Y\rTo \Gamma$. Then, a crossing 
\[
\xymatrix{
\cal M^0 \ar[rr] \ar[rd] && \cal G[\cal M^0] \\
& \cal M \ar[ru]_{\a}  \ar[rd]^{\b} & \\
\cal M^0 \ar[ru] \ar[rr] && \Gamma[\cal M^0]
}
\]
is nothing but a generalized map from $\cal G$ to $\Gamma$ in the sense of~\cite{Moerdijk:Orbifolds_groupoids}. We have shown the following result (Theorem~\ref{thm:crossing_decomposition}, Proposition~\ref{pro2:crossing} ):\\

\noindent  {\em There is a one-to-one correspondence between crossings of groupoid crossed modules and weak morphisms of topological $2$--groupoids}. \\

Groupoid crossed modules form a weak $3$--category $\XMod$ in which $1$--morphisms are crossings, $2$--morphisms are what we have called {\em exchangers} (\S~\ref{subs:exchangers}), and $3$--morphisms are morphisms of those (Theorem~\ref{thm:3-cat}). Coherently invertible $1$--morphisms in $\XMod$ are called {\em crossed extensions} and, as the terminology suggests, they generalize groupoid extensions (Example~\ref{ex:extension_crossed}) and principal $2$--group(oid) bundles (see \S~\ref{subs:principal}). Moreover (as a consequence of Theorem~\ref{thm:Morita_xext}), we have the following characterization of Morita equivalences of topological $2$--groupoids:\\

\noindent  {\em Let \ ${\bf G}, {\bf H}$ be topological $2$--groupoids, and let ${\frak G}_{\bf G}$ and ${\frak G}_{\bf H}$ be their corresponding groupoid crossed modules. Then, ${\bf G}$ is Morita equivalent to ${\bf H}$ is and only if ${\frak G}_{\bf G}$ and ${\frak G}_{\bf H}$ are coherently $1$--isomorphic in $\XMod$.} \\

Several concepts and results about crossed modules, topological $2$--groupoids, and their morphisms are presented in this work. Also, we shall note that it has come to our attention that our notion of crossings covers {\em butterflies} which were defined only for discrete groups by Noohi in~\cite{Noohi:Butterflies}. \\

\noindent {\bf General plan.} The paper is organized as follows: \S\ref{sec:2-grpd} is devoted to basic definitions and elementary properties about various kinds of morphisms of topological $2$--groupoids. In \S\ref{sec:crossed_modules}, fundamental elements of groupoid crossed modules are presented, transformations of their strict morphisms and {\em hypercovers} are introduced. \S\ref{sec:cross} is an introduction to {\em crossings} and {\em crossed extensions}; also their composition called {\em diamond product} is defined. We go further to construct in \S\ref{sec:2-cat} a $2$--category structure on crossed extensions  by introducing a kind of weak morphisms we called {\em exchangers} which generalize Morita equivalences of groupoid extensions. In \S\ref{sec:3-cat} we give a complete description of all the various compositions of crossed extensions, exchangers, and their morphisms, and then construct explicitly the weak $3$--category $\XMod$ of groupoid crossed modules.

\section{Equivalences of topological $2$--groupoids}\label{sec:2-grpd}

In this section, we revisit topological $2$--groupoids. Specifically, we give explicit definitions and derive basic results on weak morphisms and weak equivalences between topological $2$--groupoids.

\subsection{Strict $2$--groupoids revisited}

Recall 
that a topological groupoid consists of a topological space $\cal G^0$ of {\em units}  (or {\em objects}) $x, y, z$, etc., and a topological space $\cal G^1$ of invertible arrows $g:x\rTo y$ between units, together with
\begin{itemize}
\item two continuous maps $s,t:\cal G^1\rTo \cal G^0$, the {\em source} and {\em target} maps, defined by $s(g)=x$ and $t(g)=y$, for $g:x\rTo y$;
\item a continuous map $\cal G^1\ni (g:x\rTo y)\mto (g^{-1}:y\rTo x)\in \cal G^1$, called the {\em inverse} map;
\item a continuous embedding $\cal G^0\ni x\mto 1_x\in \cal G^1$, where $1_x$ is the identity arrow $x\rTo x$; and
\item a continuous {\em composition} of arrows $\cal G^{(2)}\ni (g_1,g_2)\mto g_1g_2\in \cal G^1$, where $\cal G^{(2)}:=\cal G^1\times_{s,\cal G^0,t}\cal G^1$ is the set of all composable pairs $(g_1, g_2)$; that is, $s(g_1)=t(g_2)$.
\end{itemize} 

Moreover, the inverse map satisfies $gg^{-1}=t(g)$ and $g^{-1}g=s(g)$ for all $g\in \cal G$. We will use the notation $\Ggrd$ for such a groupoid, and when there is no risk of confusion we will write $\cal G$ for $\cal G^1$. As usual, for $x,y\in \cal G^0$, we denote by $\cal G_x, \cal G^y$, and $\cal G_x^y$, as the subspaces of $\cal G$ consisting of all arrows with source $x$, with target $y$, and with source $x$ and target $y$, respectively.\\

In this paper, by a $2$--groupoid, we mean a {\em strict $2$--groupoid} (see for instance~\cites{Buss-Meyer-Zhu:Higher_twisted, Noohi:two-groupoids}); that is, a strict $2$--category (~\cite{Baez:intro-nCat}) in which all $1$--arrows and $2$--arrows are invertible. To unwrap this up, a {\em $2$--groupoid} $\bf G$ consists of 
\begin{itemize}
\item three sets ${\bf G}^0, {\bf G}^1$, and ${\bf G}^2$, whose elements are respectively called objects, {\em $1$--morphisms}, (or $1$--arrows), and {\em $2$--morphisms} (or $2$--arrows);
\item four maps ${\bf G}^2 \ou{\rrTo}{s_2}{t_2} {\bf G}^1\ou{\rrTo}{s}{t} {\bf G}^0$ such that $t\circ t_2=t\circ s_2$ and $s\circ s_2=s\circ t_2$;
\item a partial product ${\bf G}^1\times_{s,{\bf G}^0,t}{\bf G}^1\ni (g,h)\mto gh\in {\bf G}^1$ with respect to which ${\bf G}^1 \ou{\rrTo}{s}{t} {\bf G}^0$ is a groupoid;
\item a partial product ${\bf G}^2\times_{s_2,{\bf G}^1,t_2}{\bf G}^2 \ni (a,b)\mto a\star_vb \in {\bf G}^2$, called {\em vertical composition}, with respect to which ${\bf G}^2 \ou{\rrTo}{s_2}{t_2} {\bf G}^1$ is a groupoid with source and targets maps $s_2$ and $t_2$, respectively;
\item a partial product ${\bf G}^2\times_{s\circ s_2,{\bf G}^0,t\circ s_2}{\bf G}^2\ni (a,b)\mto a\star_hb\in {\bf G}^2$, called {\em horizontal composition}; 
\end{itemize}
with the requirement that the coherence law 
\[
(a_1\star_ha_2)\star_v(b_1\star_hb_2)=(a_1\star_vb_1)\star_h(a_2\star_vb_2)
\]
when the products make sense. Objects are represented by the letters $x,y,z$, etc., while $1$--arrows and $2$--arrows in $\bf G$ are respectively represented by arrows and {\em bigons} as below:

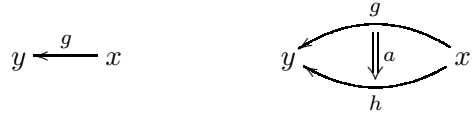
\[
\xymatrix{\relax
y & x \ar[l]_{g} && y && x \bigon[ll]{h}{g}{a}
}
\] 
Horizontal composition can then be visualized through horizontal concatenation of bigons as below
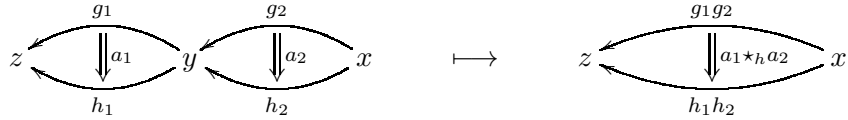
\[
\xymatrix{\relax
z && y \bigon[ll]{h_1}{g_1}{a_1} && x\bigon[ll]{h_2}{g_2}{a_2} & \mto & z &&& x \bigon[lll]{h_1h_2}{g_1g_2}{a_1\star_ha_2}
}
\]
and vertical composition is pictured through vertical concatenation of bigons as follow 
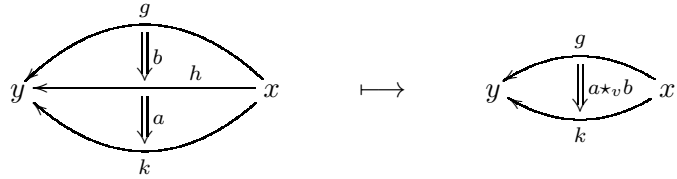
\[
\xymatrix{\relax
y &&& x \bigons[lll]{k}{h}{g}{a}{b} & \mto & y && x \bigon[ll]{k}{g}{a\star_vb} 
}
\]

\begin{df}
By a {\em topological $2$--groupoid} it is meant a $2$--groupoid $\bf G$ where $\bf G^0, \bf G^1,\bf G^2$ are topological spaces, and all the structure maps are continuous; that is, $s,t,s_2,t_2$, all the composition maps on $\bf G^1$ and $\bf G^2$, as well as the inversion maps of the groupoids ${\bf G}^1\rrTo {\bf G}^0$ and ${\bf G}^2\rrTo {\bf G}^1$ are all continuous.
\end{df}

\begin{ex}
Every topological groupoid $\cal G\rrTo \cal G^0$ is a topological $2$--groupoid with $2$--morphisms being all trivial (the identity map). 
\end{ex}

\begin{ex}[Cover $2$--groupoid]\label{ex:cover_2-gpd}
Let $\cal G\rrTo \cal G^0$ be a topological groupoid and $\cal U=\{U_i\}_{i\in I}$ be an open cover of the unit space. We form the {\em cover $2$--groupoid} $\cal G[\cal U,2]$ as follows: 
\begin{itemize}
\item the object space is the disjoint union $\coprod_iU_i=\{(i,x)\in I\times \cal G^0\mid x\in U_i\}$;
\item $1$--morphisms are elements of the space 
\[
\cal G[\cal U]\colonequals \{(i,g,j)\in I\times \cal G\times I\mid t(g)\in U_i, s(g)\in U_j\}
\]
with inverse, source, and target maps respectively defined by \[
(i,g,j)^{-1}\colonequals (j,g^{-1},i), \ s(i,g,j)\colonequals (j,s(g)), \ {\rm and\ } t(i,g,j)\colonequals (i,t(g));
\] 
\item $2$--morphisms space $\cal G[\cal U]^2$ is defined as
\[
\cal G[\cal U,2]\colonequals \left\{(i_1,i_2,g,h,j_1,j_2)\in I^2\times \cal G^2\times I^2\mid t(g)=t(h)\in U_{i_1i_2}, s(g)=s(h)\in U_{j_1j_2}\right\}
\]
with inverse, source, and target respectively given by 
\[
\begin{array}{lcl}
(i_1,i_2,g,h,j_1,j_2)^{-1} & = & (i_1,i_2,h,g,j_1,j_2); \\ 
s_2(i_1,i_2,g,h,j_1,j_2) & = & (i_1,g,j_1) ; \ {\rm and\ }\\
t_2(i_1,i_2,g,h,j_1,j_2) & = & (i_2,h,j_2);
\end{array}
\]
\item composition of $1$--morphisms is $(i,g,j)(j,h,k)\colonequals(i,gh,k)$, when $s(g)=t(h)$;
\item horizontal composition of $2$--morphisms is given by 
\[
(i_1,i_2,g_1,h_1,j_1,j_2)\star_h(j_1,j_2,g_2,h_2,k_1,k_2)\colonequals (i_1,i_2, g_1g_2, h_1h_2, k_1,k_2)
\]
and is defined for $s(g_1)=t(g_2)$;
\item vertical composition of $2$--morphisms is 
\[
(i_1,i_2,g,h,j_1,j_2)\star_v(i_2,i_3,h,k,j_2,j_3)\colonequals (i_1,i_3,g,k,j_1,j_3).
\]
\end{itemize}
\end{ex}

\noindent Two $1$--arrows $g$ and $h$ with same source and target in $\bf G$ are said to be $2$--{\em isomorphic} if there is a $2$--arrow $a\colon g\RTo h$. 

\begin{df}
A ({\em strict}) {\em homomorphism} $\phi\colon {\bf G}\rTo {\bf H}$ of topological $2$--groupoids consists of a commutative diagram of continuous maps 
\[
\xymatrix{
\bf G^2 \dar[r]^{^{s_2}}_{t_2} \ar[d]_{\phi^2} & \bf G^2\dar[r]^{^s}_t \ar[d]_{\phi^1} & \bf G^0\ar[d]^{\phi^0} \\ 
\bf H^2 \dar[r]^{^{s_2}}_{t_2} & \bf H^1\dar[r]^{^s}_t & \bf H^0
}
\]
which respects all the obvious relations defining the $2$--groupoid structures of $\bf G$ and $\bf H$. We say $\phi$ is a {\em (strict) isomorphism} of topological $2$--groupoids if all the vertical arrows of the above diagram are homeomorphisms.
\end{df} 

It is immediate that the composite $\psi\circ \phi=(\psi^2\circ \phi^2,\psi^1\circ \phi^1,\psi^0\circ\phi^0)$ of two homomorphisms of topological $2$--groupoids $\bf G \stackrel{\phi}{\rTo}\bf H\stackrel{\psi}{\rTo} \bf K$ is a homomorphism. We say the topological $2$--groupoids $\bf G$ and $\bf H$ are isomorphic if there are homomorphisms $\bf G\stackrel{\phi}{\rTo}\bf H \stackrel{\psi}{\rTo} \bf G$ such that $\phi\circ \psi =\Id_{\bf H}$ and $\psi \circ \phi=\Id_{\bf G}$. For the sake of simplicity, we will remove the superscripts from the components $\phi^2, \phi^1, \phi^0$ and denote all of them simply by $\phi$.

\begin{ex}\label{ex2:cover_2-gpd}
Let $\cal G\rrTo \cal G^0$ be a topological groupoid and $\cal U=\{U_i\}_{i\in I}$ an open cover of $\cal G^0$. Then we get a homomorphism of topological $2$--groupoids $\phi\colon \cal G[\cal U,2]\rTo \cal G$, where $\cal G$ is seen as a $2$--groupoid, by setting 
\[
\begin{array}{rcl}
\phi^2(i_1,i_2,g,h,j_1,j_2) & \colonequals & gh^{-1}g, \ (i_1,i_2,g,h,j_1,j_2)\in \cal G[\cal U,2]; \\
\phi^1(i,g,h) & \colonequals & \phi^2(i,i,g,g,j,j)=g, \ (i,g,j)\in \cal G[\cal U]; \\ 
\phi^0(i,x) & \colonequals & \phi^1(i,\id_x,i)=x, \ (i,x)\in \coprod_iU_i.
\end{array}
\]
\end{ex}

\subsection{Morita equivalences}

\begin{df}[Transformation]\label{df:transformation}
Let ${\bf G}, {\bf H}$ be topological $2$--groupoids, and $\phi, \psi\colon \bf G\rTo \bf H$ be two homomorphisms. A {\em transformation} $\phi \tr \psi$ is a continuous map $v\colon {\bf G}^1\rTo {\bf H}^2$ such that 
\begin{itemize}
\item[(i)] for every $x\in {\bf G}^0$, $v_x$ is an $1$-arrow from $\phi(x)$ to $\psi(x)$ in ${\bf H}$;
\item[(ii)] for every $1$--arrow $x\stackrel{g}{\rTo}y$ in $\bf G$, $v_g$ is a $2$--arrow from $\psi(g)v_x$ to $v_y\phi(g)$ in $\bf H$; that is, there is a bigon
\[
\xymatrix{\relax
\psi(y) && \phi(x) \bigon[ll]{v_y\phi(g)}{\psi(g)v_x}{v_g}
}
\]
\item[(iii)] for all chain of $1$--arrows $x\stackrel{h}{\rTo}y\stackrel{g}{\rTo}z$ in $\bf G$, the following equality of $2$--arrows is satisfied in $\bf H$
\[
v_{gh}=v_g\star_h\id_{v_y^{-1}}\star_h v_h
\]
\item[(iv)] for all bigon $\xymatrix{\relax  y && x \bigon[ll]{h}{g}{a}}$ the diagram of $2$--arrows below commutes in ${\bf H}^2\ou{\rrTo}{s_2}{t_2} {\bf H}^1$
\[
\xymatrix{
\psi(g)v_x \ar@{=>}[rr]^-{v_g}  \ar@{=>}[d]_-{\psi(a)\star_h\id_{v_x}} && v_y\phi(g) \ar@{=>}[d]^{\id_{v_y}\star_h\phi(a)} \\
\phi(h)v_x \ar@{=>}[rr]^{v_h} && v_y\phi(h) 
}
\] 
\end{itemize}
\end{df}

\begin{df}[Strong equivalences]\label{df:strong_equiv}
We say a strict homomorphism $\phi\colon {\bf G}\rTo {\bf H}$ of topological $2$--groupoids is a {\em strong equivalence} provided there exist a strict homomorphism $\psi\colon {\bf H}\rTo {\bf G}$ and transformations $u\colon \phi\circ \psi \tr \Id_{\bf H}$ and $v\colon \psi\circ \phi \tr \Id_{\bf G}$.
\end{df}

\begin{lem}\label{lem:strong_equiv}
Assume $\phi\colon {\bf G}\rTo {\bf H}$ is a strong equivalence of $2$--groupoids. Then the following hold.
\begin{itemize}
\item[(i)] $\phi$ is injective on $1$--arrows up to $2$--isomorphisms; that is, if $g,h\colon x\rTo y$ are $1$--arrows in $\bf G$ such that $\phi(g)$ and $\phi(h)$ are $2$--isomorphic in ${\bf H}$, then $g$ and $h$ are $2$--isomorphic in ${\bf G}$. 
\item[(ii)] $\phi$ is a homeomorphism on $2$--arrows; that is, the map 
\begin{eqnarray}\label{eq:strong_map1}
{\bf G}^2\ni a\mto (t_2(a),\phi(a),s_2(a))\in {\bf G}^1\times_{\phi, t_2}{\bf H}^2\times_{s_2,\phi}{\bf G}^1, 
\end{eqnarray}
is a homeomorphism.
\end{itemize} 
\end{lem}

\begin{proof}
Let $\psi\colon {\bf H}\rTo {\bf G}, \  u\colon \bf G^1\rTo {\bf G}^2$, and $v\colon {\bf H}^1\rTo {\bf H}^2$ be as in the definition. (i) Suppose we are given a $2$--arrow $$\xymatrix{\relax  \phi(y) && \phi(x) \bigon[ll]{\phi(h)}{\phi(g)}{b}}$$ in $\bf H$. Then $\xymatrix{\relax  y && x \bigon[ll]{h}{g}{a}}$ is a $2$--arrow in ${\bf G}$, where 
\begin{eqnarray}\label{eq1:strong_a}
a\colonequals (u_g\star_v(\id_{u_y}\star_h\psi(b))\star_vu_h)\star_h\id_{u_x^{-1}}.
\end{eqnarray}
\noindent (ii) First, note that from axiom (iv) in Definition~\ref{df:transformation} applied to the transformation $\psi\circ \phi \stackrel{u}{\RTo}\id_{\bf G}$, if $\phi(a)=\phi(a')$ for two $2$--arrows in $\bf G$, then we get 
\[
(a\star_h\id_{u_x})\star_vu_h=u_g\star_v(\id_{u_y}\star_h\psi(\phi(a')))=(a'\star_h\id_{u_x})\star_vu_h.
\]
Therefore, $a=a'$ in ${\bf G}$, and the continuous map~\eqref{eq:strong_map1} is injective. Clearly, the same axiom applied to the natural transformation $\phi\circ \psi \stackrel{v}{\RTo}\id_{\bf H}$ shows that when $b$ and $b'$ are $2$--arrows in ${\bf H}$ such that $\psi(b)=\psi(b')$, then $b=b'$. Now to see that the map~\eqref{eq:strong_map1} is surjective, it suffices to show that for a $2$--arows of the form
\[
\xymatrix{\relax  \phi(y) && \phi(x) \bigon[ll]{\phi(h)}{\phi(g)}{b}}
\] 
in ${\bf H}$, we have $b=\phi(a)$ where $a\colon g\RTo h\in {\bf G}^2$ is given by the formula~\eqref{eq1:strong_a}. In virtue of Definition~\ref{df:transformation} (iv), we have the following commutative diagram of $2$--arrows 
\[
\xymatrix{
gu_x\ar@{=>}[rr]^{u_g} \ar@{=>}[d]_{a\star_h\id_{u_x}} && u_y\psi\phi(g) \ar@{=>}[d]^{\id_{u_y}\star_h\psi\phi(a)} \\ 
hu_x\ar@{=>}[rr]^{u_h} && u_y\psi\phi(h) 
}
\]
But, by definition of the element $a$, we get 
\[
u_g\star_v(\id_{u_y}\star_h\psi(b))=(a\star_h\id_{u_x})\star_vu_h.
\]
It follows that $\id_{u_y}\star_h\psi(b)=\id_{u_y}\star_h\psi(\phi(a))$, hence $\phi(a)=b$; which implies surjectivity. Moreover, it is straightforward that the map
\[
{\bf G}^1\times_{\phi, t_2}{\bf H}^2\times_{s_2,\phi}{\bf G}^1\ni (h,b,g)\mto a\in {\bf G}^2
\]
is continuous.
\end{proof}

Let us now recall a few notions about $2$--categories (see~\cite[\S 1.13]{Lurie:HigherTopos}). Two objects $x,y$ in a $2$--category $\cal C$ are said {\em equivalent} provided there exists two arrows $x\stackrel{f}{\rTo}y\stackrel{g}{\rTo}x$ and two invertible $2$--arrows $fg\RTo \id_y$ and $gf\RTo \id_x$. Let $\cal C$ and $\cal C'$ be $2$--categories. A $2$--functor $\phi\colon \cal C\rTo \cal C'$ is an {\em equivalence} if 
\begin{itemize}
\item it is {\em essentially surjective}; that is, every object $y$ in $\cal C'$ is equivalent to an object of the form $\phi(x)$ for some object $x$ in $\cal C$;
\item it is fully faithful; that is, for all pair of objects $x,y$ in $\cal C$, the functor 
\[
\cal C(x,y)\rTo \cal C'(\phi(x),\phi(y))
\]
is an equivalence of categories. 
\end{itemize} 

For topological $2$--groupoids, we give the definition below.

\begin{df}[Weak equivalences]\label{def:weak_equiv_2-grpd}
A strict homomorphism $\phi\colon {\bf G}\rTo {\bf H}$ of topological $2$--groupoids is said to be a {\em weak equivalence} if the following are satisfied
\begin{enumerate}[label=(\textbf{WE\arabic*}), align=left]
\item\label{WE1} the continuous map $\xymatrix{{\bf G}^0\times_{\phi,{\bf H}^0,t}{\bf H}^1 \ar[rr]^-{s\circ pr_2} && {\bf H}^0}$ is surjective;
\item\label{WE2} the continuous map ${\bf G}^1\times_{\phi,{\bf H}^1,t_2}{\bf H}^2\ni (g,b)\mto (t(g),s_2(b),s(g))\in {\bf H}^1[\bf G^0]$ is surjective;
\item\label{WE3} the diagram 
\[
\xymatrix{
{\bf G}^2 \ar[rr]^-\phi \ar[d]_{(s_2,t_2)} && {\bf H}^2\ar[d]^{(s_2,t_2)} \\ {\bf G}^1\times {\bf G}^1 \ar[rr]^{\phi\times \phi} && {\bf H}^1\times {\bf H}^1
}
\]
is a pullback of topological spaces.
\end{enumerate}
We will write ${\bf G}\stackrel{\sim}{\rTo} {\bf H}$ for a weak equivalence of $2$--groupoids.
\end{df}

\begin{rmk}\label{rmk:we_2-grpd_Morita_grpd}
Note that, in particular, axioms~\ref{WE2} and~\ref{WE3} imply that $\phi$ induces a weak equivalence of topological groupoids between ${\bf G}^2\ou{\rrTo}{s_2}{t_2}{\bf G}^1$ and ${\bf H}^2\ou{\rrTo}{s_2}{t_2} {\bf H}^1$ (see~\cite{Moerdijk:Orbifolds_groupoids}).
\end{rmk}

\begin{ex}
It is easy to check that for any topological groupoid $\cal G\rrTo \cal G^0$ and an open cover $\cal U$ of $\cal G^0$, the homomorphism $\phi\colon \cal G[\cal U,2]\rTo \cal G$ defined in Example~\ref{ex2:cover_2-gpd} is a weak equivalence.
\end{ex}

\begin{pro}\label{pro:strong_weak}
Strong equivalences of topological $2$--groupoids are weak equivalences.
\end{pro}

\begin{proof}
Suppose $\phi\colon {\bf G}\rTo {\bf G}$ is a strong equivalence of topological $2$--groupoids. Let $\psi\colon {\bf H}\rTo {\bf G}$, $\psi\phi \stackrel{u}{\RTo} \id_{\bf G}$, and $\phi\psi\stackrel{v}{\RTo} \id_{\bf H}$ be as in Definition~\ref{df:strong_equiv}. For $y\in {\bf H}^0$, we have the $1$--morphism $\phi\psi(y) \stackrel{v_y}\rTo y$, hence $s\circ pr_2(\psi(y),v_y^{-1})=y$, which implies that $\phi$ satisfies axiom~\ref{WE1}. Assume now that $(y,\g,x)\in {\bf H}^1[{\bf G}^0]$; that is $\phi(x)\stackrel{\g}{\rTo}{\phi(y)}$ is a $1$--morphism in ${\bf H}$. Let $g\in {\bf G}^1$ be the $1$--morphism from $x$ to $y$ given by 
\[
g=u_y\psi(\g)u_x^{-1}.
\]
Then, $\g=s_2(b)$ with $\phi(g)=t_2(b)$, where the $2$--arrow $\g\stackrel{b}{\RTo}\phi(g)$ in ${\bf H}$ is given by
\[
b\colonequals \left(v_{\g}\star_v(\id_{v_{\phi(y)}}\star_h\phi(\id_{u_y^{-1}}\star_hu_g))\star_vv_{\phi(g)}^{-1}\right)\star_h\id_{v_{\phi(x)}^{-1}}
\]
which shows that axiom~\ref{WE2} is satisfied. Finally, that axiom~\ref{WE3} holds for $\phi$ follows from Lemma~\ref{lem:strong_equiv}.
\end{proof}

\begin{df}\label{df:weak_hom_2-grpd}
Let ${\bf G}$ and ${\bf H}$ be topological $2$--groupoids.
\begin{enumerate}
\item A {\em weak homomorphism} from ${\bf G}$ to ${\bf H}$ is a zig-zag chain of homomorphisms of topological $2$--groupoids ${\bf G}\stackrel{\sim}{\lTo}{\bf G}'\rTo {\bf H}$ where ${\bf K}$ is a third topological $2$--groupoid, the leftward arrow is a weak equivalence. 
\item We say that ${\bf G}$ and ${\bf H}$ are {\em Morita equivalent} if there is such a chain such that the rightward homomorphism is also a weak equivalence. In this case we say that $${\bf G}\stackrel{\sim}{\lTo}{\bf G}'\stackrel{\sim}{\rTo}{\bf H}$$ is a {\em Morita equivalence} of topological $2$--groupoids.
\end{enumerate}
\end{df}

\section{Groupoid crossed modules}\label{sec:crossed_modules}

\subsection{Automorphism groupoids}

Given a topological space $Z$, a right {\em groupoid action} of $\Ggrd$ on $Z$ consists of 
\begin{itemize}
\item a continuous map $\s: Z\rTo \cal G^0$, called the {\em moment map} or {\em generalized source map};
\item a continuous map $Z\times_{\s, t}\cal G\ni (z,g)\mto zg\in Z$, where $$Z\times_{\s, t}\cal G=\{z,g)\in Z\times \cal G\mid \s(z)=t(g)\};$$
\end{itemize}
such that $z(gh)=(zg)h$ for all $(z,g)\in Z\times_{\s,t}\cal G, (g,h)\in \cal G^{(2)}$, $\s(zg)=s(g)$, and $z\s(z)=z$, where we identify $\s(z)$ with the identity arrow $1_{\s(z)}$ in $\cal G$. Similarly, one defines left groupoid actions with moment maps $\tau:Z\rTo \cal G^{(0)}$.

If a topological groupoid $\Hgrd$ satisfies $\cal H_x=\cal H^x=\cal H^x_x$ for all $x\in \cal H^0$, then $\cal H$ is a {\em bundle of topological groups} over the object space $\cal H^0$. In such a case, we define the topological groupoid $\AutH$ in which arrows from $x$ to $y$ are all isomorphisms of topological groups from $\cal H_x$ to $\cal H_y$; that is $\Aut(\cal H)^y_x=\Iso(\cal H_y,\cal H_x)$, and $\Aut(\cal H)=\coprod_{x,y\in \cal H^0}\Iso(\cal H_y,\cal H_x)$, where given two topological groups $A$ and $B$, $\Iso(B,A)$ is the space of all isomorphisms of topological groups $\xymatrix{A\ar[r]^\cong & B}$ equipped with the compact open topology. 

An {\em action by automorphisms} of the topological groupoid $\Ggrd$ on the bundle of topological groups $\cal H\rTo \cal G^0$ is a morphism of topological groupoids $c:\cal G\rTo \Aut(\cal H)$. We will often write $h^g$ for $c_g(h)$ when the action makes sense. 

\begin{df}
If $\cal G$ acts by automorphisms on $\cal H$, the {\em semi-direct product} (\cite{Tu-Xu:Ring_structure}) $$\xymatrix{\cal H\rtimes \cal G \dar[r] & \cal G^0}$$ is the topological groupoid whose arrows are elements of the topological space 
\[
\cal H\rtimes \cal G:=\left\{(h,g)\in \cal H\times \cal G\mid \ h\in {\cal H}^{t(g)}\right\},
\]
source and target maps $s(h,g)=s(h^g)$ and $t(h,g)=t(h)$, respectively, and whose partial product and inverse are $(h,g)(k,f)=(hk^{g^{-1}},gf)$ and $(h,g)^{-1}=((h^g)^{-1},g^{-1})$.
\end{df}

\subsection{Crossed modules of topological groupoids}

We recall the basics of crossed modules of topological groupoids and define some preliminary notions we are going to use in the next sections.

\begin{df}[~\cites{Tu-Xu:Ring_structure, Buss-Meyer-Zhu:Non-Hausdorff_symmetries}]\label{def:crossed}
A {\em crossed module of topological groupoids} $\frak G$ consists of a quadruple $(\cal G,\cal H, \partial, c)$ where $\Ggrd$ is a topological groupoid, $\cal H$ is a bundle of topological groups over the unit space $\cal G^0$, $\partial:\cal H\rTo \cal G$ is a morphism of topological groupoids, and $c: \cal G\rTo \Aut(\cal H)$ is an action by automorphisms of $\cal G$ on $\cal H$ such that 
\begin{enumerate}
\item\label{item1:crossed} $\partial(c_g(h))=g^{-1}\partial(h)g$, for all $(g,h)\in \cal G\times_{t,\cal G^0}\cal H$,
\item\label{item2:crossed} $c_{\partial(h)}(k)=h^{-1}kh$, for all $(h,k)\in \cal H^{(2)}$.
\end{enumerate}
\end{df}

We will often omit the map $c$ and write $\xymatrix{\cal H \ar@{-x}[r]^\partial & \cal G}$ for the crossed module $\frak G$. 

\begin{ex}\label{ex:G-module_crossed}
Let $\cal G\rrTo X$ be a topological groupoid and $\cal A \overset{p}{\rTo} X$ a $\cal G$--module; that is, $\xymatrix{\cal A\ar[r]^p & X}$ is a bundle of topological abelian groups equipped with an action of $\cal G$ by automorphisms with moment map $p$. Then $\cal A \overset{p}{\xTo}\cal G$ is a groupoid crossed module with $c$ being given by the $\cal G$--action on $\cal A$.
\end{ex}

\begin{ex}\label{ex:crossed_bundle}
Let $\xymatrix{\cal H \ar[r]^p & X}$ be a bundle of topological groups. Then we form the crossed module of topological groupoids $\xymatrix{\cal H \ar@{-x}[r]^-{Ad} & \Aut(\cal H)}$ in the obvious way: for $h\in \cal H_x$, $Ad_h\in \Iso(\cal H_x,\cal H_x)$ is the inner automorphism of the topological group $\cal H_x$; that is $Ad_h(k)=hkh^{-1}$ for $k\in \cal H_x$. It is clear that $Ad$ is a morphism of topological groupoids. Moreover, we define the groupoid action by automorphisms $c:\Aut(\cal H)\rTo \Aut(\cal H)$ by setting $c_f(h):=f\circ Ad_h=Ad_{f(h)}\circ f$, for $f\in \Iso(\cal H_y,\cal H_x)$ and $h\in \cal H_y$.
\end{ex}

\begin{ex}\label{ex:group_bundle}
A particular case of Example~\ref{ex:crossed_bundle} is when $\cal H$ is a $G$-bundle; that is, all the fibers $\cal H_x=p^{-1}(x)$ are isomorphic to a fixed topological group $G$, and $\cal H$ is locally isomorphic to $U\times G$, $U$ an open subset of $X$. Then $\Aut(\cal H)\cong X\times \Aut(G)$, hence the groupoid morphism $Ad$ is given by $Ad_h=(x,ad_h)\in X\times \Aut(G)$, for $h\in \cal H_x$, where $ad_h$ is the usual conjugation morphim in the group $G$.
\end{ex}

\begin{ex}\label{ex:extension_crossed}
Let $\Ggrd$ be a topological groupoid and $\cal A$ be a $\cal G$--module. Let $$\xymatrix{\cal A\ar@{^(->}[r]^i & \tilde{\cal G} \ar@{->>}[r]^\pi & \cal G}$$ be an {\em extension} of $\cal G$ (see~\cite{Renault:Groupoid_Cstar}). We may identify the groupoid $\xymatrix{\cal A\dar[r]&\cal G^0}$ with its image $i(\cal A)$ in $\xymatrix{\tilde{\cal G} \dar[r]&\cal G^0}$. Then $\tilde{G}$ induces a groupoid action of $\cal G$ by automorphisms on $\cal A$ by $c_g(a):=\tilde{g}a\tilde{g}^{-1}$, for $s(\tilde{g})=p(a)$, $a\in \cal A$, where $\tilde{g}$ is any lift of $g$ in $\tilde{G}$, that is $\pi(\tilde{g})=g$. Furthermore, it is easy to check that $(\cal G,\cal A, \pi\circ i, c)$ is a crossed module of topological groupoids.
\end{ex}

\begin{ex}\label{ex:inertia_crossed}
Let $\Ggrd$ be a topological groupoid. The {\em inertia groupoid} (~\cite{Tu-Xu:Ring_structure}) $S\cal G\rrTo \cal G^0$ of $\cal G$ is defined by $$S\cal G\colonequals \{g\in \cal G\mid s(g)=t(g)\}$$ with the inherited structure maps from $\cal G$. Then the source and target maps make $S\cal G$ into a bundle of topological groups over $\cal G^0$ upon which $\cal G$ acts by automorphisms through the conjugation map $Ad\colon \cal G\ni g\mto Ad_g\in \Aut(S\cal G)$. Moreover, together with the canonical inclusion $\iota\colon S\cal G\mono \cal G$, $(\cal G,S\cal G,\iota, Ad)$ is a crossed module of topological groupoids which we will call the {\em inertia crosssed module} of $\cal G$ and we will denote by $\cal S\cal G$. 
\end{ex}

\begin{df}
A {\em strict morphism} of groupoid crossed modules $\chi\colon \frak G_1 \rTo \frak G_2$ is a pair $(\chi^l, \chi^r)$ of groupoid strict morphisms $\chi^l\colon \cal H_1\rTo \cal H_2$, $\chi^r\colon \cal G_1\rTo \cal G_2$ together with a commutative diagram of topological groupoids 
\[
\xymatrix{\cal H_1 \ar[d]_{\chi^l} \ar[r]^{\p_1} & \cal G_1 \ar[d]^{\chi^r} \\ 
\cal H_2\ar[r]^{\p_2} & \cal G_2}
\]
such that for all arrow $g:x\rTo y$ in $\cal G_1$, the following is a commutative diagram of topological groups
\[
\xymatrix{\cal H_1^y \ar[d]_{\chi^l} \ar[r]^{(c_1)_g} & \cal H_1^x \ar[d]^{\chi^r} \\
\cal H_2^{\chi^l(y)} \ar[r]^{(c_2)_{\chi^r(g)}} & \cal H_2^{\chi^l(x)}
}
\]
Here the superscripts $l$ and $r$ refer to left and right, respectively. Such a strict morphism is called a {\em strict isomorphism} if the vertical arrows of the above diagram are homeomorphisms.
\end{df}

In the sequel, we will omit the superscripts $l$ and $r$ and indistinguishably denote both strict morphisms by $\chi$. 

We compose strict morphisms of crossed modules $\frak G_1\stackrel{\chi}{\rTo} \frak G_2\stackrel{\k}{\rTo}\frak G_3$ in the obvious way.  

\begin{df}\label{df:semidirect_strict-morphism}
Associated to every strict morphism of groupoid crossed modules 
\[
\xymatrix{
\cal H_1\ar[r]^{\p_1} \ar[d]_{\chi} &\cal G_1\ar[d]^{\chi}\\
\cal H_2\ar[r]^{\p_2} & \cal G_2
}
\]
with common unit space $X_1=X_2=X$, there is a topological groupoid $\cal H_2\rtimes_\chi \cal G_1\rrTo X$ called the {\em semidirect product groupoid associated to} $\chi$, and where the groupoid action by automorphisms of $\cal G_1$ on $\cal H_2$ is with respect to the composition map $\cal G_1\stackrel{\chi}{\rTo} \cal G_2\rTo \Aut(\cal H_2)$ (and the source and target maps are $s(h_2,g_1)=s(g_1)$ and $t(h_2,g_1)=t(g_1)$). When there is no risk of confusion, we will just write $\cal H_2\rtimes \cal G_1$ 
\end{df}

\subsection{Pullbacks and projective product groupoids}\label{subsec:proj}

Recall that if $\Ggrd$ is a topological groupoid and $\s:Z\rTo \cal G^0$ a continuous function, the {\em pullback groupoid} $\s^*\cal G\rrTo Z$, also denoted by $\cal G[Z] \rrTo Z$ when the map $\s$ is understood, is defined as follows: 
\begin{itemize}
\item the space of arrows is the fibered product 
\[
Z\times_{\s,t}\cal G\times_{s,\s}Z:=\{(z_1,g,z_2)\in Z\times \cal G\times Z \mid \s(z_1)=t(g), s(g)=\s(z_2)\};
\]
\item the target and source maps are $t(z_1,g,z_2)=z_1$ and $s(z_1,g,z_2)=z_2$;
\item the inverse map is given by $(z_1,g,z_2)^{-1}=(z_2,g^{-1},z_1)$, and
\item the partial product is $(z_1,g,z_2)(z_2,h,z_3):=(z_1,gh,z_3)$, when $s(g)=t(h)$.
\end{itemize}

\begin{pro}\label{pro:strict_pullback}
Suppose $\frak G$ is a groupoid crossed module and $\s:Z\rTo \cal G^0$ a continuous function. Form the bundle of topological groups $\cal H[\s]\rTo Z$ by setting 
\[
\cal H[\s]:=Z\times_{\s,s}\cal H:=\{(z,h)\in Z\times \cal H\mid \s(z)=s(h)=t(h)\}
\]
with projection $\cal H[\s]\ni (z,h)\mto z\in Z$. Then, the operations 

\begin{itemize}
\item $\s^*\p(z,h)=(z,\p(h),z), \ {\rm for\ } (z,h)\in \cal H[\s]$, and
\item $\s^*c_{(z_1,g,z_2)}(z_2,h)=(z_1,c_g(h))$, for  $(z_1,g,z_2)\in \s^*\cal G, \ h\in \cal H^{\s(z_2)}$,
\end{itemize}
make $\xymatrix{\cal H[\s] \ar@{-x}[r]^-{\s^*\p} & \s^*\cal G}$ a crossed module of topological groupoids with unit space $Z$. Furthermore, the canonical projections $\cal H[\s]\ni (z,h)\mto h\in \cal H$ and $\s^*\cal G\ni (z_1,g,z_2)\mto g\in \cal G$ define a strict morphism of crossed modules
\[
\xymatrix{
\cal H[\s] \ar[r]^-{\s^*\p} \ar[d] & \s^*\cal G \ar[d] \\
\cal H\ar[r]^-\p & \cal G
}
\]
\end{pro}

\begin{proof}
That the groupoid morphisms $\s^*\p:\cal H[\s]\rTo \s^*\cal G$ and $\s^*c:\s^*\cal G\rTo \Aut(\cal H[\s])$ satisfy~\ref{item1:crossed} and~\ref{item2:crossed} of Definition~\ref{def:crossed} follows directly from the fact that $(\cal G,\cal H,\p,c)$ is a crossed module.
\end{proof}

\begin{df}
Given $\frak G = \HG$ and $\s:Z\rTo \cal G^0$ as above, the crossed module $$\xymatrix{\cal H[\s] \ar@{-x}[r]^-{\s^*\p} & \s^*\cal G}$$ is called the {\em pullback} of $\frak G$ over $Z$ through the map $\s$ and denoted by $\s^*\frak G$. When there is no risk of confusion with the map $\s$, we will also write $\cal H[Z]$ for the bundle of topological groups $\cal H[\s]$, and $\frak G[Z]$ for $\s^*\frak G$.
\end{df}

Now, we will see that Morita equivalences of groupoids induces chains of strict morphisms of crossed modules. More precisely, let $\frak G_1$ and $\frak G_2$ be two groupoid crossed modules whose base groupoids $\cal G_1$ and $\cal G_2$ are Morita equivalent; that is there exist a topological space $P$ and two continuous maps 
\[
\xymatrix{
\cal G_1^0 & P \ar[l]_{\tau} \ar[r]^{\s} & \cal G_2^0
}
\]
where $\tau:P\rTo \cal G_1^0$ is a locally trivial principal $\cal G_2$--bundle and $\s:P\rTo \cal G_2^0$ is a locally trivial $\cal G_1$--principal bundle. Define the {\em projective product groupoid of $\cal G_1$ and $\cal G_2$ over $P$} as the groupoid $\cal G_1\underset{P}{\ast}\cal G_2 \rrTo P$ whose morphisms are quadruples $(g_1,p_1,p_2,g_2)\in \cal G_1\times P\times P\times \cal G_2$ such that 
\[
(g_1,p_2)\in \cal G_1\times_{s,\tau}P, (p_1,g_2)\in P\times_{\s,t}\cal G_2, \ {\rm and\ } \ p_1g_2=g_1p_2.
\] 
The source and target maps of $\cal G_1\underset{P}{\ast}\cal G_2$ are $s(g_1,p_1,p_2,g_2)=p_2, t(g_1,p_1,p_2,g_2)=p_1$, and the composition and inverse are 
\[
(g_1,p_1,p_2,g_2)(g_1',p_2,p_3,g'_2)=(g_1g_1',p_1,p_3,g_2g_2'), \ (g_1,p_1,p_2,g_2)^{-1}=(g_1^{-1},p_2,p_1,g_2^{-1}).
\]

Similarly, we form $\cal H_1\underset{P}{\ast}\cal H_2\rrTo P$ to be the groupoid with unit space $P$ whose morphisms are triples $(h_1,p,h_2)\in \cal H_1\times_{s,\tau}P\times_{\s,s}\cal H_2$ such that $p\p_2(h_2)=\p_1(h_1)p$, and whose source and target maps are both given by $s(h_1,p,h_2)=t(h_1,p,h_2)=p$. Then $\cal H_1\underset{P}{\ast}\cal H_2$ is clearly a bundle of topological groups over $P$. In fact, we have the following

\begin{pro}
Let $\frak G_1$, $\frak G_2$, and $P$ be as above. Then, together with the map 
\[
\p_1\underset{P}{\ast}\p_2\colon \cal H_1\underset{P}{\ast}\cal H_2 \ni (h_1,p,h_2)\mto (\p_1(h_1),p,p,\p_2(h_2))\in \cal G_1\underset{P}{\ast}\cal G_2
\]
and the actions by of $\cal G_1\underset{P}{\ast}\cal G_2$ by automorphisms on $\cal H_1\underset{P}{\ast}\cal H_2\rTo P$ given by 
\[
(c_1\underset{P}{\ast}c_2)_{(g_1,p_1,p_2,g_2)}(h_1,p_1,h_2)\colonequals (h_1^{g_1},p_2,h_2^{g_2})
\]
$\cal H_1\underset{P}{\ast}\cal H_2\overset{\p_1\underset{P}{\ast}\p_2}{\xTo}\cal G_1\underset{P}{\ast}\cal G_2$ is a groupoid crossed module which will be called the {\em projective product of $\frak G_1$ and $\frak G_2$ over $P$}, and denoted by $\frak G_1\underset{P}{\ast}\frak G_2$. Furthermore, this construction is natural with respect to morphisms of Morita equivalence and strict morphisms of crossed modules; that is, any morphism of Morita equivalences $P\rTo Q$ induces a strict morphism of crossed modules $\frak G_1\underset{P}{\ast}\frak G_2\rTo \frak G_1\underset{Q}{\ast}\frak G_2$.
\end{pro}

\begin{proof}
That $\frak G_1\underset{P}{\ast}\frak G_2$ is a crossed module is straightforward. Now if $Q$ is another Morita equivalence between $\cal G_1$ and $\cal G_2$, and $f:P\rTo Q$ a morphism of Morita equivalence; that is, $f$ is a continuous map that commutes with the groupoid actions of $\cal G_1$ and $\cal G_2$ on $P$ and $Q$, then the strict morphism of crossed modules $\frak G_1\underset{P}{\ast}\frak G_2\rTo \frak G_1\underset{Q}{\ast}\frak G_2$ is obtained through the strict morphisms of groupoids 
\begin{eqnarray*}
\cal G_1\underset{P}{\ast}\cal G_2\ni (g_1,p_1,p_2,g_2)\mto (g_1,f(p_1),f(p_2),g_2)\in \cal G_1\underset{Q}{\ast}\cal G_2\\
\cal H_1\underset{P}{\ast}\cal H_2\ni (h_1,p,h_2)\mto (h_1,f(p),h_2)\in \cal H_1\underset{Q}{\ast}\cal H_2
\end{eqnarray*}
\end{proof}

\begin{pro}\label{pro:projective_product}
Let $\frak G_1$ and $\frak G_2$ be two groupoid crossed modules and $\cal G_1\overset{P}{\rTo} \cal G_2$ be a Morita equivalence of the base groupoids. Then there exist a crossed module $\frak G$ and two strict morphism 
\[
\xymatrix{
\frak G_1 & \frak G \ar[l] \ar[r] & \frak G_2
}
\]
which are natural with respect to morphisms of Morita equivalences.
\end{pro}

\begin{proof}
Take $\frak G\colonequals \frak G_1\underset{P}{\ast}\frak G_2$. Then, the canonical projections give the commutative diagrams 

\[
\xymatrix{
\cal H_1 \ar[r]^{\p_1} & \cal G_1 \\
\cal H_1\underset{P}{\ast} \cal H_2\ar[u]_{pr_1} \ar[r]^{\p_1\underset{P}{\ast}\p_2} \ar[d]^{pr_3} & \cal G_1\underset{P}{\ast} \cal G_2 \ar[u]^{pr_1}  \ar[d]^{pr_4}\\
\cal H_2 \ar[r]^{\p_2} & \cal G_2 
}
\]
which are easily checked to satisfy the the axioms of strict morphisms of crossed modules. These data being natural with respect to morphisms of Morita equivalence follows from the projective product over a Morita equivalence being natural.
\end{proof}

\subsection{Crossed modules vs $2$--groupoids}\label{sec:cross_vs_2-grpd}

Let $\frak G=\cal H\stackrel{\p}{\xTo} \cal G$ be a crossed module of topological groupoids. Then we form the topological $2$--groupoid $\bf G$ as the $2$--groupoid  
\[
\xymatrix{
\cal H\rtimes \cal G\dar[r]^-{s_2}_-{t_2} & \cal G\dar[r]^{s}_t & \cal G^0
} 
\]  
where the semidirect product $\cal H\rtimes \cal G\rrTo \cal G$ is equipped with the structure of a tolopological groupoid with unit space $\cal G$ with source, target, and inverse maps
\[
s_2(h,g)\colonequals g, \ t_2(h,g)\colonequals \p(h)g, \ (h,g)^{-1}\colonequals (h^{-1},\p(h)g),  \ (h,g)\in \cal H\rtimes \cal G; 
\]
and where compositions, when defined, are given by
\[
(h_1,g_1)\star_h(h_2,g_2)\colonequals (h_1h_2^{g_1^{-1}}, g_1g_2), \ (h,g)\star_v(k,\p(h)g)\colonequals (hk, g). 
\]

\noindent We call $\bf G$ the (topological) $2$--groupoid associated to the groupoid crossed module $\frak G$, and will be sometimes denoted as ${\bf G}_{\frak G}$. The topological groupoid $\cal H\rtimes \cal G\rrTo\cal G$ will be called the {\em vertical semidirect product groupoid}.  Moreover, it is easy to see that this construction is natural with respect to strict morphisms of crossed modules and strict homomorphisms of $2$--groupoids; that is, if $\chi\colon \frak G_1\rTo \frak G_2$ is a strict morphism of crossed modules, then the diagram of homomorphisms of topological groupoids
\[
\xymatrix{
\cal H_2\rtimes\cal G_2 \dar[r]^-{s_2}_-{t_2} \ar[d]_{\chi\times \chi} & \cal G_2 \dar[r]^s_t \ar[d]_{\chi} & \cal G_2^0 \ar[d]^{\chi} \\ 
\cal H_1\rtimes\cal G_1\dar[r]^-{s_2}_-{t_2} & \cal G_1\dar[r]^s_t & \cal G_1^0
}
\]
commutes, hence $\chi$ induces a strict homomorphism $\chi\colon \bf G_1\rTo \bf G_2$ between the associated topological $2$--groupoids.

Conversely, suppose $\bf G$ is a topological $2$--groupoid ${\bf G}^2\rrTo {\bf G}^1\rrTo {\bf G}^0$. Then, associated to $\bf G$, there is a crossed module of topological groupoids $\frak G$ defined as follows:
\begin{itemize}
\item we set $\cal G$ to be the groupoid ${\bf G}^1\rrTo {\bf G}^0$;
\item the bundle of topological groups $\cal H\rTo {\bf G}^0$ is defined as the subspace of ${\bf G}^2$ consisting of all the $2$--morphisms of the form 
\[
\xymatrix{\relax
x & x \bigon[l]{g}{\id_x}{a}
}
\]
with the obvious projection onto ${\bf G}^0$: $a\mto s(t_2(a))=t(t_2(a)), a\in \cal H$;
\item the groupoid morphism $\p\colon \cal H\rTo \cal G$ is $\p(a)=t_2(a)$; and finally,
\item the groupoid action by automorphisms $c\colon \cal G\rTo \Aut(\cal H)$ is given by 
\[
c_g(a)=a^g \colonequals \id_{g^{-1}}\star_ha\star_h\id_g, \ g\in \cal G, a\in {\cal H}^{t(g)}.
\]
\end{itemize}

\begin{lem}
Let $\bf G$ be a topological $2$--groupoid and $\frak G$ its associated crossed modules of topological groupoids. Denote by ${\bf G}_{\frak G}$ the $2$--groupoid associated to $\frak G$. Then there is a strict isomorphism of topological $2$--groupoids 
\[
{\bf G} \stackrel{\cong}{\rTo} {\bf G}_{\frak G}
\] 
which is natural with respect to strict morphisms; that is, if ${\frak G}={\cal H}\xTo \cal G$ as above, then there is a strict isomorphism of topological $2$--groupoids
\[
\xymatrix{
{\bf G}^2 \dar[r]^-{s_2}_-{t_2} \ar[d]_{\phi} & {\bf G}^1 \dar[r]^-{s}_-t \ar[d]_{\phi} & {\bf G}^0\ar[d]^{\phi}\\
\cal H\rtimes \cal G \dar[r]^-{s_2}_-{t_2} & \cal G \dar[r]^s_t & {\bf G}^0 
}
\]  
\end{lem}

\begin{proof}
We get the isomorphism of topological $2$--groupoids by setting for a bigon $\xymatrix{\relax y & x \bigon[l]{\g'}{\g}{b}}\in {\bf G}^2$:
\[
\phi(b) \colonequals (b\star_h\id_{\g^{-1}}, \g) \in \cal H\rtimes \cal G, 
\]
and for $(a,g)\in \cal H\rtimes \cal G$:
\[
\psi(a,g)\colonequals a\star_h \id_g\in {\bf G}^2
\]
\end{proof}

\begin{cor}
The functor ${\bf G}\mto {\bf G}_{\frak G}$ is an equivalence of categories between the categories of topological $2$--groupoids and strict homomorphisms and the category of crossed modules of topological groupoids and strict morphisms. 
\end{cor}

\subsection{Transformations and strong equivalences}

We define a strong notion of equivalence of groupoid crossed modules. We will need the following which generalizes the notion of {\em crossed homomorphisms} (\cite{MacLane:Homology}) for topological groupoids.

\begin{df}
Let $\cal H\stackrel{\p}{\xTo} \cal G$ be a groupoid crossed module, $\Gamgrd$ a topological groupoid. A {\em $\cal G$--crossed homomorphism} $\G\stackrel{(\vp,\lam)}{\rTo} \cal H$ consists of a pair $(\vp, \lam)$ where $\vp\colon \G\rTo \cal G$ is a strict morphism of topological groupoids and $\lam\colon \G\rTo \cal H$ is a continuous map such that
\begin{itemize}
\item the diagram below commutes
\[
\xymatrix{
\G \ar[r]^{\lam} \ar[d]_{\vp} & \cal H\ar[d]^t \\ 
\cal G \ar[r]^s & \cal G^0
}
\]
that is, $\lam(\g)\in \cal H^{\vp(t(\g))}$, for every $\g\in \G$;
\item $\lam(\g\g^\prime)=\lam(\g)\lam(\g^\prime)^{\vp(\g)^{-1}}$, for every composable pair $(\g,\g^\prime)\in \G^{(2)}$.
\end{itemize}
\end{df}

In particular, when $\G=\cal G$, a continuous map $\lam\colon\cal G\rTo \cal H$ is called a crossed homomorphism if $(\Id_{\cal G},\lam)$ is a $\cal G$--crossed homomorphism; that is, $\lam(g)\in \cal H^{t(g)}$ for every $g\in \cal G$ and $$\lam(gg^\prime)=\lam(g)\lam(g^\prime)^{g^{-1}}$$ when $g$ and $g^\prime$ are composable. 

\begin{df}\label{def:transformation}
Let $\frak G_1$ and $\frak G_2$ be groupoid crossed modules with unit spaces $X_1$ and $X_2$, respectively, and $\chi, \k\colon \frak G_1\rTo \frak G_2$ be two strict morphisms of crossed modules. A {\em transformation} $\chi \tr \k$ is a pair $(T,\lam)$ where $T\colon X_1\rTo \cal G_2$ and $\lam\colon \cal G_1\rTo \cal H_2$ are continuous maps satisfying the following properties:
\begin{enumerate}[label=(\textbf{T\arabic*}), align=left]
\item\label{T1} For every $x_1\in X_1$, $T(x_1)$ is an arrow from $\chi(x_1)$ to $\k(x_1)$ in $\cal G_2$.
\item\label{T2} $\cal G_1 \stackrel{(\chi, \lam)}{\rTo}\cal H_2$ is a $\cal G_2$--crossed homomorphism.
\item\label{T3} For every arrow $x_1\stackrel{g_1}{\rTo}y_1 \in \cal G_1$, the relation $\k(g_1)T(x_1)=T(y_1)\p_2(\lam(g_1))\chi(g_1)$ holds.
\item\label{T4} $\k(h_1)^{T(x_1)}=\lam(\p_1(h_1))\chi(h_1)$ for all $x_1\in X_1$ and $h_1\in \cal H_1^{x_1}$.
\end{enumerate}
\end{df}

Notice that in the particular case where $\lam$ is trivial, that is $\lam(\cal G_1)\subset X_2$, then it follows from~\ref{T3} that every arrow $x_1\stackrel{g_1}{\rTo}y_1 \in \cal G_1$ induces a commutative diagram 
\[
\xymatrix{
\chi(x_1) \ar[r]^{T(x_1)} \ar[d]_{\chi(g_1)} & \k(x_1)\ar[d]^{\k(g_1)} \\
\chi(y_1) \ar[r]^{T(y_1)} & \k(y_1)
}
\]
of arrows in $\cal G_2$, so that $T$ is a natural transformation from the groupoid strict morphism $\chi^r\colon \cal G_1\rTo \cal G_2$ to $\k^r$ (\cite{Moerdijk:Orbifolds_groupoids}).  

As an example of transformations of strict morphisms between crossed modules, we have the lemma below.

\begin{lem}\label{lem1:transform_crosssed_module_vs_2-grpd}
Transformations of strict homomorphisms betweem topological $2$--groupoids are transformations of strict morphisms between the associated groupoid crossed modules.
\end{lem} 

\begin{proof}
Let $\bf G$ and $\bf H$ be topological $2$--groupoids, $\phi, \psi\colon {\bf G}\rTo {\bf H}$ two strict homomorphisms. Assume $\phi \stackrel{v}{\RTo}\psi$ is a transformation. Denote by ${\frak G}_{\bf G}$ and ${\frak G}_{\bf H}$ the groupoid crossed modules associated to $\bf G$ and $\bf H$, respectively, and $\phi, \psi\colon {\frak G}_{\bf G}\rTo {\frak G}_{\bf H}$ the induced strict morphisms. We write ${\frak G}_{\bf G}=\cal H_1\xTo {\bf G}^1$, and ${\frak G}_{\bf H}=\cal H_2\xTo {\bf H}^1$. Then we obtain the desired transformation $\phi \stackrel{(T,\lambda)}{\tr}\psi$ by letting $T\colon {\bf G}^0\rTo {\bf H}^1$ to be defined by $T(x)\colonequals v_x$, and defining the function $\lambda\colon {\bf G}^1\rTo {\bf H}^2$ by the horizontal composition 
\[
\xymatrix{\relax
\chi(y) && \k(y) \bigon[ll]{v_y^{-1}}{v_y^{-1}}{\id_{v_y^{-1}}} && \chi(x)  \bigon[ll]{v_y\chi(g)}{\k(g)v_x}{v_g} && \k(x)  \bigon[ll]{v_x^{-1}}{v_x^{-1}}{\id_{v_x^{-1}}} && \k(y) \bigon[ll]{\k(g)^{-1}}{\k(g)^{-1}}{\id_{\k(g)^{-1}}} && \chi(y) \bigon[ll]{v_y}{v_y}{\id_{v_y}} 
}
\]
That is, for $x\stackrel{g}{\rTo}y\in {\bf G}^1$, we set
\[
\lambda(g)\colonequals \id_{v_y^{-1}}\star_hv_g\star_h\id_{v_x^{-1}}\star_h\id_{\k(g)^{-1}}\star_h\id_{v_y}\colon \id_{\chi(y)}\RTo \chi(g)v_x^{-1}\k(g)^{-1}v_y \in {\cal H}_2
\]
It is left to the reader to check that $(T,\lambda)$ indeed satisfies all the axioms~\ref{T1}--~\ref{T4}. 
\end{proof}

\begin{lem}\label{lem2:transform_crosssed_module_vs_2-grpd}
Transformations between strict morphisms of groupoid crossed modules induce transformations between the induced morphisms between the associated topological $2$--groupoids.
\end{lem}

\begin{proof}
Indeed, let $(T,\lam)\colon \chi \tr \k$ be as in Definition~\ref{def:transformation}, then for every arrow $x_1\stackrel{g_1}{\rTo}y_1$ in $\cal G_1$, we have the $2$--arrow 
\[
\xymatrix{
\chi(x_1) \ar[r]^{T(x_1)} \ar[d]_{\chi(g_1)} & \k(x_1) \ar[d]^{\k(g_1)} \ar@{=>}[ld]_{v_{g_1}} \\
\chi(y_1) \ar[r]_{T(y_1)} & \k(y_1)
}
\]
in the $2$--groupoid $\cal H_2\rtimes \cal G_2 \rrTo \cal G_2\rrTo X_2$, where 
\begin{eqnarray}\label{eq:transform_transform}
v_{g_1}\colonequals ((\lam(g_1)^{T(y_1)^{-1}})^{-1},\ \k^r(g_1)T(x_1))\in \cal H_2\rtimes \cal G_2.
\end{eqnarray}

Moreover, given a $2$--arrow $(h_1,g_1)\colon g_1\RTo \p_1(h_1)g_1$ in $\cal H_1\rtimes \cal G_1\rrTo \cal G_1\rrTo X_1$, axioms~\ref{T3} and~\ref{T4} gives the commutative diagram of $2$--arrows in $\cal H_2\rtimes \cal G_2\rrTo \cal G_2\rrTo X_2$ 
\[
\xymatrix{
\k(g_1)T(x_1) \ar@{=>}[rr]^-{v_{g_1}} \ar@{=>}[d]_{(\k(h_1),\ \k(g_1)T(x_1))} && T(y_1)\chi(g_1) \ar@{=>}[d]^{(\chi(h)^{T(y_1)^{-1}},\ T(y_1)\chi(g_1))} \\
\p_2(\k(h_1))\k(g_1)T(x_1) \ar@{=>}[rr]^{v_{\p_1(h_1)g_1}} && T(y_1)\p_2(\chi(h_1))\chi(g_1)
}
\]
Hence, we get the transformation $v\colon {\cal G}_1\rTo \cal H_2\rtimes \cal G_2$ from $\chi$ to $\kappa$ as strict homomorphisms of topological $2$--groupoids ${\bf G}_1\rTo {\bf G}_2$, where $v_{x_1}=T(x_1)$ for $x_1\in X_1$, and for $g_1\in \cal G_1$, the $2$--arrow $v_{g_1}\in \cal H_2\rtimes \cal G_2$ is given by~\eqref{eq:transform_transform}. 
\end{proof}

\begin{pro}
Let $\chi$ and $\k$ be strict morphisms from the groupoid crossed module $\frak G_1$ to $\frak G_2$ with same unit space $X$. Assume $\chi \stackrel{(T,\lambda)}{\tr}\k$ is a transformation of strict morphisms. Then, the map 
\[
\begin{array}{ccc}
 \cal H_2\rtimes_{\chi}\cal G_1 & \rTo & \cal H_2\rtimes_{\k}\cal G_1 \\
 (h_2,g_1) & \mto & \left(h_2^{T(t(g_1))^{-1}}, g_1\right)
\end{array}
\]
is an isomorphism of topological groupoids.
\end{pro}

\begin{proof}
Straightforward.
\end{proof}

\begin{df}
A strict morphism $\chi\colon \frak G_1\rTo\frak G_2$ is called a {\em strong equivalence} provided that there exist a strict morphism $\k\colon \frak G_2\rTo \frak G_1$ and transformations $\chi\circ \k\overset{(T,\lam)}{\tr} \Id_{\frak G_2}$ and $\k\circ \chi \overset{(T^\prime,\lam^\prime)}{\tr}\Id_{\frak G_1}$. Two crossed modules are said to be {\em strongly equivalent} if there is a strong equivalence between them.
\end{df}

By virtue of Lemma~\ref{lem1:transform_crosssed_module_vs_2-grpd} and Lemma~\ref{lem2:transform_crosssed_module_vs_2-grpd}, we immediately have the following.

\begin{pro}
Two groupoid crossed modules are strongly equivalent if and only if their associated topological $2$--groupoids are.
\end{pro}

\subsection{Hypercovers}

\begin{df}
A strict morphism of groupoid crossed modules 
\[
\xymatrix{\cal H_1 \ar[d]_{\chi} \ar[r]^{\p_1} & \cal G_1 \ar[d]^{\chi} \\ 
\cal H_2\ar[r]^{\p_2} & \cal G_2}
\]
is called a {\em hypercover} if the induced homomorphism of $2$--groupoids between the associated $2$--groupoids 
\[
\xymatrix{
\cal H_1\rtimes \cal G_1\dar[r]^-{s_2}_-{t_2}\ar[d]_{\chi\times \chi} & \cal G_1\dar[r]^{s}_t \ar[d]_\chi & X_1 \ar[d]^\chi \\ 
\cal H_2\rtimes\cal G_2\dar[r]^-{s_2}_-{t_2}& \cal G_2\dar[r]^s_t &  X_2
}
\] 
is a weak equivalence. 
\end{df}

It is an immediate observation that the composition of hypercovers is a hypercover.

\begin{rmk}\label{rmk1:hypercover}
In particular, a hypercover from $\frak G_1$ to $\frak G_2$ gives a Morita equivalence of topological groupoids between the semidirect product groupoids $\cal H_1\rtimes\cal G_1\rrTo \cal G_1$ and $\cal H_2\rtimes \cal G_2\rrTo \cal G_2$, thanks to Remark~\ref{rmk:we_2-grpd_Morita_grpd}. 
\end{rmk}

\begin{ex}
Let $\cal H\stackrel{\p}{\xTo} \cal G$ be a crossed module and $\s\colon Z\rTo \cal G^0$ a continuous surjection. Then the strict morphism 
\[
\xymatrix{\cal H[Z] \ar[d]_{pr_2} \ar[r]^{\s^\ast\p} & \cal G[Z] \ar[d]^{pr_2} \\ 
\cal H\ar[r]^{\p} & \cal G}
\]
is a hypercover. 
\end{ex}

As a consequence of Proposition~\ref{pro:strong_weak}, Lemma~\ref{lem1:transform_crosssed_module_vs_2-grpd} and Lemma~\ref{lem2:transform_crosssed_module_vs_2-grpd}, we have the result below. 

\begin{pro}
Every strong equivalence of crossed modules of topological groupoids is a hypercover.
\end{pro}

\begin{df}\label{df:Morita_crossed}
A {\em Morita equivalence} between two crossed modules of topological groupoids ${\frak G}_1$ and ${\frak G}_2$ is a sequence $(\tilde{\frak G}_1, \chi_1), \cdots, (\tilde{\frak G}_n, \chi_n)$ where for $i=1,\cdots, n$, $\tilde{\frak G}_i$ is a crossed module of topological groupoids with $\tilde{\frak G}_1={\frak G}_1$ and $\tilde{\frak G}_n={\frak G}_2$, and for $1\leq i\leq n-1$, $\chi_i$ is either a hypercover $\tilde{\frak G}_i\rTo \tilde{\frak G}_{i+1}$ either a hypercover $\tilde{\frak G}_{i+1}\rTo \tilde{\frak G}_i$. In such a case, we say that ${\frak G}_1$ and ${\frak G}_2$ are {\em Morita equivalent}.
\end{df}

In particular, 

\begin{pro}
Two groupoid crossed modules are Morita equivalent if and only if their associated topological $2$--groupoids are.
\end{pro}

\section{Crossed extensions}\label{sec:cross}

Our goal is to give more concrete notions of weak homomorphisms and Morita equivalences of groupoid crossed modules, hence of topological $2$--groupoids.

\subsection{Crossings}

In this section we introduce the notion of {\em crossings} and {\em crossed extensions} between groupoid crossed modules. In the sequel we write {\bf TL} for Top left, {\bf TR} for top right, {\bf BL} for bottom left, and {\bf BR} for bottom right.

\begin{pro}\label{pro:crossing}
Let 
\[
\xymatrix{\cal H_1 \ar[d]_{\chi^l} \ar[r]^{\p_1} & \cal G_1 \ar[d]^{\chi^r} \\ 
\cal H_2\ar[r]^{\p_2} & \cal G_2}
\]
be a strict morphism of groupoid crossed modules with same unit space $X=\cal H_1^0=\cal G_1^0=\cal H_2^0=\cal G_2^0$. Then there exists a topological groupoid $\xymatrix{\cal M\dar[r] & X}$ and commutative diagrams

\begin{eqnarray}\label{eq1:crossing}
\xymatrix{
\cal H_1 \ar[rr]^{\p_1} \ar[rd]_{\a_1} && \cal G_1 \\
& \cal M \ar[ru]_{\a_2}  \ar[rd]^{\b_2} & \\
\cal H_2 \ar[ru]^{\b_1} \ar[rr]^{\p_2} && \cal G_2
}
\end{eqnarray}
such that 
\begin{enumerate}[label=(\textbf{CR\arabic*}), align=left]
\item\label{CR1} $\a_i$ and $\b_i, i=1,2$ are the identity map on the unit space $X$;
\item\label{CR2} the diagonals are complexes of topological groupoids over $X$; 
\item\label{CR3} the diagonal {\bf BL-TR} is an extension of topological groupoids over $X$;
\item\label{CR4} the following equations hold

\begin{align}
\a_1(h_1^{\a_2(m)})=m^{-1}\a_1(h_1)m, \ \forall m\in \cal M, \ h_1\in \cal H_1^{s(\a_2(m))}, t(m)=\a_1(s(h_1)), \label{eq2:crossing} \\
\b_1(h_2^{\b_2(m)})=m^{-1}\b_1(h_2)m, \ \forall m\in \cal M, h_2\in \cal H_2^{s(\b_2(m))}, \ t(m)=\b_1(s(h_2)) \label{eq3:crossing}
\end{align} 
\end{enumerate}
\end{pro}

\begin{proof}
Let $\cal M$ be the semidirect product groupoid $\xymatrix{\cal H_2\rtimes \cal G_1\dar[r] & X}$ associated to $\chi$ (see Definition~\ref{df:semidirect_strict-morphism}). Next, define the groupoid morphisms 
\[
\a_1\colon \cal H_1 \ni h_1\mto (\chi (h_1^{-1}), \p_1(h_1))\in \cal M, \ \ \a_2\colon \cal M \ni (h_2,g_1)\mto g_1\in \cal G_1, 
\]
and 
\[
\b_1\colon \cal H_2 \ni h_2\mto (h_2,s(h_2))\in \cal M, \ \ \b_2\colon \cal M\ni (h_2,g_1) \mto \p_2(h_2)\chi(g_1)\in \cal G_2.
\]
Then, it is easy to check that the quintuple $(\cal H_2\rtimes \cal G_1,\a_1,\a_2,\b_1,\b_2)$ satisfies the commutative diagram~\eqref{eq1:crossing} as well as relations~\eqref{eq2:crossing} and~\eqref{eq3:crossing}. 
\end{proof}

\begin{df}\label{def:crossing}
Let $\frak G_1$ and $\frak G_2$ be two groupoid crossed modules with unit spaces $\cal G_1^0=X_1$ and $\cal G_2^0=X_2$. A {\em crossing} from $\frak G_1$ to $\frak G_2$ consists of a quintuple $(\cal M, \a_1,\a_2,\b_1,\b_2)$, where $\cal M\rrTo \cal M^0$ is a topological groupoid together with two continuous maps 
\[
X_1\stackrel{\tau}{\lTo} \cal M^0\stackrel{\s}{\rTo}X_2
\]
and $\a_1\colon\cal H_1[\cal M^0]\rTo \cal M, \ \a_2\colon \cal M\rTo \cal G_1[\cal M^0], \b_1\colon\cal H_2[\cal M^0]\rTo \cal M$, and $\b_2\colon \cal M\rTo \cal G_2[\cal M^0]$ are groupoid strict morphisms making the following
\begin{eqnarray}\label{eq4:crossing}
\xymatrix{
\cal H_1[\cal M^0] \ar[rr]^{\tau^*\p_1} \ar[rd]_{\a_1} && \cal G_1[\cal M^0] \\
& \cal M \ar[ru]_{\a_2}  \ar[rd]^{\b_2} & \\
\cal H_2[\cal M^0] \ar[ru]^{\b_1} \ar[rr]^{\s^*\p_2} && \cal G_2[\cal M^0]
}
\end{eqnarray}
a commutative diagram satisfying to axioms~\ref{CR1},~\ref{CR2},~\ref{CR3}, and~\ref{CR4}.
We will usually omit the morphisms $\a_i,\b_i, i=1,2$. \\
\end{df}

An {\em endocrossing} of a groupoid crossed module $\frak G$ is a crossing from $\frak G$ to itself.

\begin{df}
A {\em crossed extension} is a triple $(\frak G_1,\cal M,\frak G_2)$ where $\frak G_1, \frak G_2$ are groupoid crossed modules and $\cal M$ is a crossing from $\frak G_1$ to $\frak G_2$ and from $\frak G_2$ to $\frak G_1$; that is, there are commutative diagrams such as~\eqref{eq4:crossing} satisfying~\ref{CR1} -~\ref{CR4} and the following axiom
\begin{enumerate}[label=(\textbf{CR\arabic*}'), align=left] 
 \setcounter{enumi}{2}
\item\label{CR3_prime} the diagonal {\bf TL - BR} is an extension of topological groupoids.
\end{enumerate} 
We also say that $\cal M$ is a crossed extension of $\frak G_1$ and $\frak G_2$.      
\end{df}

We will use the following notations.

\begin{rmk}
Notice that if $\cal M=(\cal M,\a_1,\a_2,\b_1,\b_2)$ is a crossed extension of $\frak G_1$ and $\frak G_2$, then $$\bar{\cal M}=(\cal M,\b_1,\b_2,\a_1,\a_2)$$ is a crossing from ${\frak G}_1$ to ${\frak G}_2$, and is actually a crossed extension of ${\frak G}_2$ and ${\frak G}_1$. 
\end{rmk}

\begin{rmk}
Note our notion of {\em crossed extension} should not be confused with Holt's {\em crossed sequences}~\cite{Holt:Cohomology}.
\end{rmk}

\begin{ex}
Let $\cal G\rrTo X$ be a topological groupoid and $p\colon \cal A\rTo X$ a $\cal G$--module. Any groupoid $\cal A$--extension $\cal A\stackrel{\iota}{\mono}\tilde{\cal G}\stackrel{\pi}{\epi} \cal G$ defines a crossed extension via the commutative diagram 
\[
\xymatrix{
\cal A \ar[rr]^{p} \ar[rd]_{\iota} && \cal G \\
& \tilde{\cal G} \ar[ru]_{\pi}  \ar[rd]^{\pi} & \\
\cal A \ar[ru]^{\iota} \ar[rr]^{p} && \cal G
}
\]
Conversely, a crossed extension of $\cal A\xTo \cal G$ and itself is a groupoid $\cal A$--extension of $\cal G$.
\end{ex}

\begin{ex}
Given two crossed modules of groups $\frak G_1= H_1\xTo G_1$ and $\frak G_2= H_2\xTo G_2$, a crossing from $\frak G_1$ to $\frak G_2$ is a {\em butterfly}, and a crossed extension $(G_1, M,G_2)$ is a {\em flippable butterfly} as defined by Noohi~\cite{Noohi:Butterflies} and Aldrovandi-Noohi~\cite{Aldrovandi-Noohi:Butterflies}.
\end{ex}

\begin{ex}\label{ex1:groupoid_crossing}
Let $\Gamgrd$ be a topological groupoid viewed as a groupoid crossed module via the inclusion map $\G^0 \overset{i}{\mono} \G$ and the trivial group bundle $\G^0\rTo \G^0$. Let $\frak G=\cal H\xTo \cal G$ be a groupoid crossed module with unit space $\G^0$. Then if $\cal M\rrTo \G^0$ is a crossing between $\G$ and $\frak G$, the morphism $\b_2$ is a groupoid isomorphism $\cal M \cong \cal G$ and $\cal G$ is a groupoid $\cal H$--extension of $\G$. In fact, we will show by Theorem~\ref{thm:crossings_vs_extensions} that all groupoid extensions are obtained this way.
\end{ex}

\begin{ex}
Let $\Gamgrd$ and $\Ggrd$ be topological groupoids regarded as crossed modules as in Example~\ref{ex1:groupoid_crossing}. Then a crossed extension $\cal M\rrTo \cal M^0$ of $\Gamma$ and $\cal G$ is the same thing a Morita equivalence between $\Gamma$ and $\cal G$. Indeed, if the crossed extension is given by the maps $\Gamma^0\stackrel{\tau}{\lTo}\cal M^0\stackrel{\s}{\rTo} \cal G^0$, then we get a commutative diagram of the form 
\[
\xymatrix{
\cal M^0 \ar@{^(->}[rr] \ar@{^(->}[rd] && \Gamma[\cal M^0] \\
& \cal M \ar[ru]_{\a_2} \ar[rd]^{\b_2} & \\
\cal M^0 \ar@{^(->}[ru] \ar@{^(->}[rr] && \cal G[\cal M^0]
}
\]
where $\a_2$ and $\b_2$ are groupoid isomorphisms. The converse is obvious by the characterizations of groupoid Morita equivalence (see for instance~\cites{Moerdijk:Orbifolds_groupoids, Moutuou:Real.Cohomology}) and the very definition of crossed extensions. 
\end{ex}

\begin{ex}
From Proposition~\ref{pro:crossing}, we see that any strict morphism of crossed modules from $\frak G_1$ to $\frak G_2$ with the same unit space $X$ induces the crossing $\cal H_2\rtimes\cal G_1$ from ${\frak G}_1$ to ${\frak G}_2$, with the maps $\tau$ and $\s$ taken to be the identity map $\Id_X$.
\end{ex}

More generally, we have the following.

\begin{pro}\label{pro2:crossing}
Every strict morphism of groupoid crossed modules induces a crossing.
\end{pro}

\begin{proof}
Let $\frak G_1, \frak G_2$ be groupoid crossed modules with unit spaces $X_1$ and $X_2$, respectively, and suppose $\chi\colon \frak G_1\rTo \frak G_2$ is a strict morphism. 
Put $Z_{\chi}\colonequals X_1\times_{\chi,t}\cal G_2$ and define the continuous maps
\[
\xymatrix{X_1 & Z_{\chi} \ar[l]_\tau \ar[r]^\s & X_2}
\]
by $\tau(x,g_2)\colonequals x$ and $\s(x,g_2)\colonequals s(g_2)$. Then we obtain the strict morphism $(\tilde{\chi}^l,\tilde{\chi}^r)\colon \tau^*\frak G_1\rTo \s^*\frak G_2$ of crossed modules with the same unit space $Z_{\chi}$, where the groupoid morphisms $\tilde{\chi}^l\colon \cal H_1[\tau]\rTo \cal H_2[\s]$ and $\tilde{\chi}^r\colon \tau^*\cal G_1\rTo \s^*\cal G_2$ are given by
\begin{eqnarray*}
\tilde{\chi}^l\colon \cal H_1[Z_\chi] \ni ((x,g_2),h_1) \mto ((x,g_2), \chi^l(h_1)^{g_2}) \in \cal H_2[Z_\chi], \\
\tilde{\chi}^r\colon \cal G_1[Z_\chi] \ni ((y,f_2),g_1,(x,g_2)) \mto ((y,f_2),f_2^{-1}\chi^r(g_1)g_2,(x,g_2))\in \cal G_2[Z_\chi].
\end{eqnarray*}
Hence, by Proposition~\ref{pro:crossing}, the groupoid semidirect product $\cal H_2[Z_{\chi}]\rtimes \cal G_1[Z_{\chi}]\rrTo Z_{\chi}$ the desired crossing from $\frak G_1$ to $\frak G_2$ over $Z_{\chi}$.
\end{proof}

\begin{pro}\label{pro:images_crossing}
Let $(\cal M,\a_1,\a_2,\b_1,\b_2)$ be a crossing between $\frak G_1$ and $\frak G_2$. Then the images of $\a_1$ and $\b_1$ commute in $\cal M$. 
\end{pro}

\begin{proof}
For the sake of simplicity, we may assume both crossed modules and the groupoid $\cal M$ have the same unit space, and that $X$, and that $\tau=\s=\Id_X$ (since otherwise we would just consider the pullback groupoids). Let $h_1\in \cal H_1$ and $h_2\in \cal H_2$ such that $s(h_1)=s(h_2)$. Then, by~\ref{CR2} and~\ref{CR3}, we have 
\[
\a_1(h_1)\b_1(h_2) = \b_1(h_2)(\b_1(h_2)^{-1}\a_1(h_1)\b_1(h_2))=\b_1(h_2)\a_1(h_1^{\a_2(b_1(h_2))}) =\b_1(h_2)\a_1(h_1).
\]
\end{proof}

In particular, the above proposition implies that all the fibers $\a_1(\cal H_1[\cal M^0])\cap \b_1(\cal H_2[\cal M^0])^x$ of the bundle $\a_1(\cal H_1[\cal M^0])\cap \b_1(\cal H_2[\cal M^0])\rTo \cal M^0$ over $x\in \cal M^0$ are abelian groups; therefore, $\a_1(\cal H_1[\cal M^0])\cap \b_1(\cal H_2[\cal M^0])$ is an $\cal M$--module.

\begin{pro}\label{pro:hypercover_xext}
Let $\frak G_1$ and $\frak G_2$ be groupoid crossed modules with unit spaces $X_1$ and $X_2$. If $\chi\colon \frak G_1\rTo \frak G_2$ is a hypercover, then the groupoid $\cal H_2[Z_{\chi}]\rtimes \cal G_1[Z_\chi]\rrTo Z_{\chi}$, with $Z_{\chi}\colonequals X_1\times_{\chi,t}\cal G_2$, is a crossed extension of $\frak G_1$ and $\frak G_2$.
\end{pro}

\begin{proof}
Thanks to Proposition~\ref{pro2:crossing}, we may assume that all of the groupoids involved have the same unit space $X$. We then show that the semidirect product groupoid associated to the hypercover $\chi$ defines a crossed extension. But for this, it only remains to show that the quadruple $(\cal H_2\rtimes\cal G_1,\a_1,\a_2,\b_1,\b_2)$ defined in the proof of Proposition~\ref{pro:crossing} satisfies axiom~\ref{CR3_prime}; that is, we need to show that the sequence 
\[
\cal H_1\stackrel{\a_1}{\rTo} \cal H_2\rtimes\cal G_1\stackrel{\b_2}{\rTo}\cal G_2
\]
is exact, where $\a_1(h_1)=(\chi(h_1^{-1}),\p_1(h_1))$ and $\b_2(h_2,g_1)=\p_2(h_2)\chi(g_1)$. But since the induced strict homomorphism of topological $2$--groupoids $\chi\colon {\bf G}_{\frak G_1}\rTo {\bf G}_{\frak G_2}$ is a weak equivalence, the map 
\[
\cal H_1\rtimes \cal G_1 \ni (h_1,g_1)\mto (\p_1(h_1)g_1, (\chi(h_1),\chi(g_1)), g_1)\in \cal G_1\times_{\chi,t_2}(\cal H_2\rtimes\cal G_2)\times_{s_2,\chi}\cal G_1
\]
is a homeomorphism. Therefore, if $(h_2,g_1), (h_2',g_1')\in \cal H_2\rtimes\cal G_1$ are such that $\b_2(h_2,g_1)=\b_2(h_2',g_1')$, so that 
\[
(g_1,(h_2(h_2')^{-1},\chi(g_1')),g_1')\in \cal G_1\times_{\chi,t_2}(\cal H_2\rtimes\cal G_2)\times_{s_2,\chi}\cal G_1 \ ,
\]
then there exists a unique $h_1\in {\cal H}_1^{t(g_1')}$ such that $g_1'=\p_1(h_1^{-1})g_1'$ and $h_2(h_2')^{-1}=\chi(h_1^{-1})$, hence
\[
(h_2',g_1')=(h_2\chi(h_1^{-1}),\p_1(h_1)g_1)=\a_1(h_1)\cdot (h_2,g_1)
\]
which implies exactness of the sequence.
\end{proof}

\begin{df}[Trivial crossed extension]\label{def:trivial_crossing}
Let $\frak G$ be a groupoid crossed module. The {\em trivial crossed extension} of $\frak G$ is the endocrossing obtained by applying the construction in Proposition~\ref{pro:crossing} to the identity strict morphism of crossed modules
\[
\xymatrix{
\cal H\ar[r]^\p \ar[d]_{\Id} & \cal G\ar[d]^{\Id} \\
\cal H\ar[r]^{\p} & \cal G
}
\]
That is, the groupoid semidirect product $\cal H\rtimes \cal G\rrTo \cal G^0$, together with the homomorphisms 
\begin{eqnarray*}
\a_1\colon \cal H\ni h\mto (h^{-1},\p(h))\in \cal H\rtimes \cal G \ , \\
\a_2\colon \cal H\rtimes \cal G \ni (h,g)\mto g\in \cal G\ , \\
\b_1\colon \cal H\ni h\mto (h,s(h))\in \cal H\rtimes \cal G \ , \\
\b_2\colon \cal H\rtimes \cal G\ni (h,g)\mto \p(h)g\in \cal G.
\end{eqnarray*} 
This crossed extension will be denoted by $\cal O_{\frak G}$.
\end{df}

\begin{nota}
If $\frak G_1$ and $\frak G_2$ are groupoid crossed modules, we will write 
\[
\frak G_1\underset{\cal M}{\cross} \frak G_2
\]
to say that the groupoid $\cal M\rrTo \cal M^0$ is a crossing from $\frak G_1$ to $\frak G_2$. Moreover, if $\cal M$ is a crossed extension of $\frak G_1$ and $\frak G_2$, we write 
\[
\frak G_1\underset{\cal M}{\xext}\frak G_2.
\]
Also, we will use the notation $\frak G_1\xext \frak G_2$ without subscript to say that both crossed modules admit a crossed extension. 
\end{nota}

\subsection{Decomposition}

In this section we examine how crossings decompose into strict morphisms of groupoid crossed modules.

\begin{thm}\label{thm:crossing_decomposition}
Let $\frak G_1, \frak G_2$ be groupoid crossed modules, and $\frak G_1\underset{\cal M}{\cross}\frak G_2$ be a crossing. Then there exists a third  groupoid crossed module ${\frak G}'$ and a zig-zag chain of strict morphisms 
\begin{eqnarray}\label{eq:zig_zag-chain}
\frak G_1\lTo \frak G'\rTo \frak G_2
\end{eqnarray}
such that the leftward arrow is a hypercover. In other words, $\cal M$ induces a weak homomorphism of the associated topological $2$--groupoids from ${\bf G}_{\frak G_1}$ to ${\bf G}_{\frak G_2}$.
\end{thm}

\begin{proof}
We first assume that both crossed modules have the same unit space $\cal G_1^0=\cal G_2^0=X$ and that the crossing $\cal M$ has unit space $\cal M^0=X$ together with trivial maps $\tau=\s=\Id_X$, with respect to the commutative diagram 
\[
\xymatrix{
\cal H_1\ar[rr]^{\p_1} \ar[rd]_{\a_1} &&\cal G_1 \\
& \cal M\ar[ru]_{\a_2} \ar[rd]^{\b_2} & \\
\cal H_2\ar[ru]^{\b_1} \ar[rr]^{\p_2} && \cal G_2 
}
\]
Then $\cal M$ acts by automorphisms on $\cal H_1$ and $\cal H_2$ through the compositions $\cal M\stackrel{\a_2}{\rTo}\cal G_1\rTo \Aut(\cal H_1)$ and $\cal M\stackrel{\b_2}{\rTo}\cal G_2\rTo \Aut(\cal H_2)$, respectively. Now, notice that the fibered product $$\cal H_1\times_X\cal H_2=\{(h_1,h_2)\in \cal H_1\times\cal H_2\mid s(h_1)=s(h_2)\}$$ has the structure of topological groupoid with unit space $X$, where composition and inverse maps are components-wise; in fact it is a bundle of topological groups over $X$. Furthermore, $\cal M$ acts by automorphisms on $\cal H_1\times_X\cal H_2$ by 
\[
(h_1,h_2)^m\colonequals (h_1^{\a_2(m)},h_2^{\b_2(m)}), \ {\rm for\ } (h_1,h_2)\in \cal H_1^{t(\a_2(m))}\times \cal H_2^{t(\b_2(m))}.
\] 
Moreover, together with this groupoid action by automorphisms, the groupoid morphism 
\[
\a_1\cdot \b_1\colon \cal H_1\times_X\cal H_2\ni (h_1,h_2)\mto \a_1(h_1)\b_1(h_2)=\b_1(h_2)\a_1(h_1)\in \cal M
\]
we get the crossed module of topological groupoids $\frak G'\colonequals\cal H_1\times_X\cal H_2\stackrel{\a_1\cdot\b_1}{\xTo}\cal M$ with unit space $X$. What's more, we have the following two strict morphisms 
\[
\xymatrix{
\cal H_1\ar[r]^{\p_1} & \cal G_1 \\
\cal H_1\times_X\cal H_2 \ar[u]^{pr_1} \ar[d]_{pr_2} \ar[r]^-{\a_1\cdot\b_1} & \cal M\ar[u]_{\a_2} \ar[d]^{\b_2} \\
\cal H_2\ar[r]^{\p_2} & \cal G_2
}
\] 
We claim that the upper strict morphism $\chi\colonequals (pr_1,\a_2)\colon \frak G'\rTo \frak G_1$ is a hypercover. Indeed, it is clear that, since $\a_2$ is surjective,  the maps 
\[
X\times_{\a_2,t}\cal M \ni (x,m)\mto s(m)\in X
\]
and 
\[
\cal M\times_{\a_2,t_2}(\cal H_1\rtimes\cal M)\ni (m,(h_1,g_1))\mto (t(m),g_1,s(m))\in X\times_{\a_2,t}\cal G_1\times_{s,\a_2}X
\]
are continuous surjections. Also, since $\b_1$ is injective, the continuous map 
\[
(\cal H_1\times_X\cal H_2)\rtimes\cal M\ni ((h_1,h_2),m)\mto (\a_1(h_1)\b_1(h_2)m, (h_1,\a_2(m)), m)\in (\cal H_1\rtimes\cal G_1)[\cal M]
\]
is injective. To see that this last map is also surjective, take any element $(m_2,(h_1,g_1),m_1)\in (\cal H_1\rtimes\cal G_1)[\cal M]$. Then $\a_2(m_1)=g_1$ and $\a_2(m_2)=\p_1(h_1)\a_2(m_1)=\a_2(\a_1(h_1)m_1)$. Hence, since the sequence $$\cal H_2\stackrel{\b_1}{\rTo}\cal M\stackrel{\a_2}{\rTo}\cal M$$ is exact, there is a unique $h_2\in \cal H_2$ such that $m_2=\b_2(h_2)\a_1(h_1)m_1=\a_1(h_1)\b_2(h_2)m_1$. This means that $((h_1,h_2),m_1)\in (\cal H_1\times_X\cal H_2)\rtimes \cal M$ is sent to $(m_2,(h_1,g_1),m_1)$ through the above map, which implies surjectivity; therefore the map is a homeomorphism. Finally, in the general case, the same constructions apply on pullbacks to get the strict morphisms 
\[
\frak G_1\lTo \frak G_1[\cal M^0] \stackrel{(pr_1,\a_2)}{\lTo} \frak G' \stackrel{(pr_2,\b_2)}{\rTo} \frak G_2[\cal M^0]\rTo \frak G_2
\] 
where $\frak G'$ is the crossed module $\cal H_1[\cal M^0]\times_{\cal M^0}\cal H_2[\cal M^0]\xTo \cal M$. We then achieve the proof by noticing that the composite of the leftward arrows is a hypercover. 
\end{proof}

\begin{cor}\label{cor:crossing_decomposition}
Assume two given groupoid crossed modules admit a crossed extension. Then they are Morita equivalent.  
\end{cor}

\begin{proof}
The constructions and arguments in the proof of Theorem~\ref{thm:crossing_decomposition} can be used to get the zig-zag chain of hypercovers 
\[
\frak G_1\stackrel{(pr_1,\a_2)}{\lTo} \frak G' \stackrel{(pr_2,\b_2)}{\rTo} \frak G_2
\]
\end{proof}

\begin{thm}
Let $\frak G_1\underset{\cal M}{\cross}\frak G_2$ be a crossing of groupoid crossed modules. Then the topological groupoids $\cal H_1\rtimes \cal G_1\rrTo \cal G_1$ and $(\cal H_1[\cal M^0]\times_{\cal M^0}\cal H_2[\cal M^0])\rtimes \cal M\rrTo \cal M$ are Morita equivalent.
\end{thm}

\begin{proof}
By the constructions in the proof of Theorem~\ref{thm:crossing_decomposition}, the crossing $\cal M\rrTo \cal M^0$ induces a hypercover  
\[
\frak G_1\stackrel{\chi}{\lTo} \frak G'
\]
where $\frak G'$ is the crossed module $\cal H_1[\cal M^0]\times_{\cal M^0}\cal H_2[\cal M^0]\xTo \cal M$. We then conclude using Remark~\ref{rmk1:hypercover}.
\end{proof}

\begin{cor}\label{cor:Morita_semidirect}
Assume the groupoid crossed modules $\frak G_1$ and $\frak G_2$ have a crossed extension. Then the vertical semidirect product groupoids $\cal H_1\rtimes\cal G_1\rrTo \cal G_1$ and $\cal H_2\rtimes\cal G_2\rrTo \cal G_2$ are Morita equivalent.
\end{cor}


\subsection{Diamond product}

Our goal in this section is to show how to concatenate crossings of groupoid crossed modules; that is, given two crossings $\frak G_1 \underset{\cal M}{\cross} \frak G_2 \underset{\cal N}{\cross} \frak G_3$, we want to be able to get a crossing $${\frak G}_1\underset{\cal P}{\cross}{\frak G}_3.$$ 

\begin{df}
Let $\frak G$ be a groupoid crossed module with unit space $X$, $\cal M\rrTo X$ and $\cal N\rrTo X$ two topological groupoids together with a commutative diagram 
\[
\xymatrix{
& \cal M \ar[rd]^{\b_2} & \\
\cal H \ar[ru]^{\b_1} \ar[rr]^{\p} \ar[rd]_{\d_1} && \cal G \\
& \cal N \ar[ru]_{\d_2} &
}
\]
of strict morphisms of topological groupoids such that 
\begin{equation}\label{CR3-prime}
\begin{array}{lcl}
\b_1(h^{\b_2(m)}) &= & m^{-1}\b_1(h)m, \forall h\in \cal H, m\in \cal M_{s(h)}, \quad {\rm and} \\
 \d_1(h^{\d_2(n)})& =& n^{-1}\d_1(h)n, \ \forall h\in \cal H, n\in \cal N_{s(h)}.
\end{array}
\end{equation}
We define the {\em diamond product} $\cal M\ou{\Dprod}{\cal H}{\cal G}\cal N$ by 
\[
\cal M\ou{\Dprod}{\cal H}{\cal G}\cal N \colonequals \cal M\times_{\cal G}\cal N/_{\sim} \colonequals \{(m,n)\in \cal M\times \cal N\mid \b_2(m)=\d_2(n)\}/_{\sim}
\]
where the equivalence relation '$\sim$' in $\cal M\times_{\cal G_2}\cal N$ is generated by 
\[
(\b_1(h)m,\d_1(h)n)\sim (m, n), \ {\rm for \ } h\in \cal H, \  s(h_2)=s(m)=s(n)
\] 
We will write $[m,n]$ for the class of $(m,n)$ in $\cal M\ou{\Dprod}{\cal H}{\cal G}\cal N$.
\end{df}

Notice that the diamond product $\cal M\ou{\Dprod}{\cal H}{\cal G}\cal N$, equipped with the quotient topology, is a topological groupoid with unit space $X$. More precisely, the source and target maps are 
\[
s([m,n])\colonequals s(m)=s(n), \ t([m,n])\colonequals t(m)=t(n)\ , 
\] 
(these are well defined since the strict groupoid morphisms $\b_2$ and $\d_2$ are the identities on $X$), 
composition is $[m,n][m',n']\colonequals [mm',nn']$ when both products are defined and the inverse is $[m,n]^{-1}\colonequals [m^{-1},n^{-1}]$. To see that the composition is well defined, notice that thanks to~\eqref{CR3-prime}, given $h\in \cal H$ and $(m,n), (m',n')\in \cal M\times_{\cal G}\cal N$ such that $t(m)=s(h)=s(m')$, we get 
\[
\left(m\b_1(h)m',n\d_1(h)n'\right)=\left(\b_1(h^{\b_2(m^{-1})})mm', \d_1(h^{\d_2(n^{-1})})nn'\right)\sim (mm',nn')
\]

\begin{lem}\label{lem:diamond_crossing}
Let $\frak G_i, i=1,2,3$ be groupoid crossed modules with the same unit space $X$. Assume two topological groupoids $\cal M\rrTo X$ and $\cal N\rrTo X$ are such that $\frak G_1 \underset{\cal M}{\cross}\frak G_2$ and $\frak G_2\underset{\cal N}{\cross}\frak G_3$ are crossings with respect to the identity maps and the commutative diagrams

\[
\xymatrix{
\cal H_1\ar[rr]^{\p_1} \ar[rd]_{\a_1} && \cal G_1  & & \cal H_2 \ar[rd]_{\d_1} \ar[rr]^{\p_2} && \cal G_2 \\
& \cal M \ar[ru]_{\a_2} \ar[rd]^{\b_2} & && & \cal N\ar[ru]_{\d_2} \ar[rd]^{\t_2} & && \\
\cal H_2 \ar[rr]^{\p_2} \ar[ru]^{\b_1} && \cal G_2 && \cal H_3 \ar[rr]^{\p_3} \ar[ru]^{\t_1} && \cal G_3  
}
\]

Then the topological groupoid $\cal M\ou{\Dprod}{\cal H_2}{\cal G_2}\cal N\rrTo X$ is a crossing from  $\frak G_1$ to $ \frak G_3$ with respect to the strict morphisms 
\begin{eqnarray*}
\a_1\Dprod\d_1\colon \cal H_1\ni h_1\mto [\a_1(h_1), s(h_1)] \in M\ou{\Dprod}{\cal H_2}{\cal G_2}\cal N\\
\a_2\Dprod\d_2\colon M\ou{\Dprod}{\cal H_2}{\cal G_2}\cal N\ni [m,n]\mto \a_2(m) \in \cal G_1\\
\b_1\Dprod\t_1\colon \cal H_3\ni h_3\mto [s(h_3),\t_1(h_3)]\\
\b_2\Dprod\t_2\colon M\ou{\Dprod}{\cal H_2}{\cal G_2}\cal N\ni [m,n]\mto \t_2(n) \in \cal G_3
\end{eqnarray*}
Furthermore, $\cal M\ou{\Dprod}{\cal H_2}{\cal G_2}\cal N$ is a crossed extension if $\cal M$ and $\cal N$ are.
\end{lem}

\begin{proof}
Straightforward.
\end{proof}

More generally,

\begin{dfpro}\label{pro:diamond_product}
Suppose $\cal M\rrTo \cal M^0$ and $\cal N\rrTo \cal N^0$ are crossings from $\frak G_1$ to $\frak G_2$ and from $\frak G_2$ to $\frak G_3$, respectively, with respect to the continuous maps
\[
\xymatrix{\cal G^0_1 & \cal M^0 \ar[l]_{\tau_1} \ar[r]^{\s_1} & \cal G_2^0 & \cal N^0 \ar[l]_{\tau_2} \ar[r]^{\s_2} & \cal G_3^0}
\]
and the commutative diagrams  
\[
\xymatrix{
\cal H_1[\cal M^0] \ar[rr]^{\tau_1^*\p_1} \ar[rd]_{\a_1} && \cal G_1[\cal M^0] \\
& \cal M \ar[ru]_{\a_2}  \ar[rd]^{\b_2} & \\
\cal H_2[\cal M^0] \ar[ru]^{\b_1} \ar[rr]^{\s_1^*\p_2} && \cal G_2[\cal M^0]
\ \quad \quad }
\xymatrix{
\cal H_2[\cal N^0] \ar[rr]^{\tau_2^*\p_2} \ar[rd]_{\d_1} && \cal G_2[\cal N^0] \\
& \cal N \ar[ru]_{\d_2}  \ar[rd]^{\theta_2} & \\
\cal H_3[\cal N^0] \ar[ru]^{\theta_1} \ar[rr]^{\s_2^*\p_3} && \cal G_3[\cal N^0]
}
\]
Consider the canonical projections $pr_1\colon\cal M^0\times_{\cal G_2^0}\cal N^0\epi \cal M^0$ and $pr_2\colon\cal M^0\times_{\cal G_2^0} \cal N^0\epi \cal N^0$. Then
\begin{itemize}
\item[(i)] $\cal H_1[\cal M^0\times_{\cal G_2^0}\cal N^0]\cong \cal H_1[\cal M^0]\times_{\s_1\circ s, \tau_2}\cal N^0$ and $\cal G_1[\cal M^0\times_{\cal G_2^0}\cal N^0]\cong \cal N^0\times_{\tau_2,\s_1\circ t}\cal G_1[\cal M^0]\times_{\s_1\circ s,\tau_2}\cal N^0$ as groupoids over $\cal M^0\times_{\cal G_2^0}\cal N^0$. 
\item[(ii)] $\cal H_3[\cal M^0\times_{\cal G_2^0}\cal N^0]\cong \cal M^0\times_{\s_1,\tau_2\circ s}\cal H_3[\cal N^0]$ and $\cal G_3[\cal M^0\times_{\cal G_2^0}\cal N^0]\cong \cal M^0\times_{\s_1,\tau_2\circ t}\cal G_3[\cal N^0]\times_{\tau_2\circ s, \s_1}\cal M^0$ as groupoids over $\cal M^0\times_{\cal G_2^0}\cal N^0$.
\item[(iii)] The topological groupoid $\cal M\Dprod_{\frak G_2}\cal N\rrTo \cal M^0\times_{\cal G_2^0}\cal N^0$, where 
\[
\cal M\Dprod_{\frak G_2}\cal N\colonequals \cal M[\cal M^0\times_{\cal G_2^0}\cal N^0]\ou{\Dprod}{\cal H_2[\cal M^0\times_{\cal G_2^0}\cal N^0]}{\cal G_2[\cal M^0\times_{\cal G_2^0}\cal N^0]}\cal N[\cal M^0\times_{\cal G_2^0}\cal N^0]\ , 
\]
 is a crossing from $\frak G_1$ to $\frak G_3$ with respect to the maps $$\xymatrix{\cal G_1^0  && \cal M^0\times_{\cal G_2^0}\cal N^0 \ar[ll]_-{pr_1\circ \tau_1}  \ar[rr]^-{pr_2\circ \s_2} && \cal G_3^0}$$
 and the associated strict morphisms corresponding to the construction of Lemma~\ref{lem:diamond_crossing} which will be denoted $\a_i\Dprod_{\frak G_2}\d_i$ and $\b_i\Dprod_{\frak G_2}\t_i, i=1,2$. 
\end{itemize}
We call the groupoid $\cal M\Dprod_{\frak G_2}\cal N\rrTo \cal M^0\times_{\cal G_2^0}\cal N^0$ the {\em diamond product} of the crossings $\cal M\rrTo\cal M^0$ and $\cal N\rrTo\cal N^0$. Moreover, $\frak G_1\underset{\cal M\Dprod_{\frak G_2}\cal N}{\xext}\frak G_3$ is a crossed extension if $\frak G_1\underset{\cal M}{\xext}\frak G_2$ and $\frak G_2\underset{\cal N}{\xext}\frak G_3$ are crossed extensions. 
\end{dfpro}

\begin{proof}
(i) and (ii) are immediate consequences of the definition of the pullback of a groupoid. (iii) The groupoid $\cal M\Dprod_{\frak G_2}\cal N\rrTo \cal M^0\times_{\cal G_2^0}\cal N^0$ is a crossing from $\frak G_1[\cal M^0\times_{\cal G_2^0}\cal N^0]$ to $\frak G_3[\cal M^0\times_{\cal G_2^0}\cal N^0]$ over the unit space $\cal M^0\times_{\cal G_2^0}\cal N^0$, thanks to Lemma~\ref{lem:diamond_crossing}. The last statement is obvious. 
\end{proof}


We have already seen that crossed extensions imply Morita equivalences of groupoid crossed modules (see Corollary~\ref{cor:crossing_decomposition}). Now, diamond product allows us to prove the converse. Specifically, we show that crossed extensions provide, in a sense, a more concrete interpretation of Morita equivalences of groupoid crossed modules. More precisely, we have the result below.

\begin{thm}\label{thm:Morita_xext}
Let ${\frak G}_1$ and ${\frak G}_2$ be two crossed modules of topological groupoids. Then,  $\frak G_1 \xext \frak G_2$ if and only if $\frak G_1$ and $\frak G_2$ are Morita equivalent.
\end{thm}

\begin{proof}
($\RTo$): See Corollary~\ref{cor:crossing_decomposition}. 
\noindent ($\Longleftarrow$) : Let $(\tilde{\frak G}_i, \chi_i), i=1,\cdots, n$ be as in Definition~\ref{df:Morita_crossed}. Using the diamond product of crossed extensions, it suffices to show that no matter the directions of the hypercovers $\chi_i=(\chi_i^l,\chi_i^r)$ are, the sequence induces crossed extensions
\[
\frak G_1=\tilde{\frak G}_1\underset{\cal M_1}{\xext} \tilde{\frak G}_2 \underset{\cal M_2}{\xext} \cdots \underset{\cal M_n}{\xext} \tilde{\frak G}_n=\frak G_2
\]
so that the groupoid $(\cdots (\cal M_1\Dprod_{\tilde{\frak G}_2}\cal M_2)\Dprod_{\tilde{\frak G}_3}\cdots )\Dprod_{\tilde{\frak G}_{n-1}}\cal M_{n-1}\rrTo \cal M^0_1\times_{\tilde{\cal G}_2^0}\cal M_2^0\times_{\tilde{\cal G}_3^0}\cdots \times_{\tilde{\cal G}_{n-1}^0}\cal M_{n-1}^0$ is a crossed extension of $\frak G_1$ and $\frak G_2$. But this, in fact, follows from Proposition~\ref{pro:hypercover_xext}; indeed, each hypercover $\chi_i$ induces a crossed extension $\cal M_i$ between $\tilde{\frak G}_i$ and $\tilde{\frak G}_{i+1}$, $i=1,\ldots, n-1$. .
\end{proof}

\begin{cor}\label{cor:xext_pullback}
Let $\frak G$ be a crossed module of topological groupoids and let $\vp\colon P\rTo \cal G^0$ and $\psi\colon Q\rTo \cal G^0$ be two continuous surjections. Then $\frak G[P]\xext \frak G[Q]$.
\end{cor}

\begin{proof}
By Proposition~\ref{pro:strict_pullback}, we have two strict morphisms of crossed modules $$\xymatrix{\frak G[P] \ar[r] & \frak G &\frak G[Q] \ar[l]}$$ which actually are hypercover. We then conclude by applying Theorem~\ref{thm:Morita_xext}.
\end{proof}



We end this section by the following observation.

\begin{cor}\label{cor:crossing_pullback}
If $\frak G_1\underset{\cal M}{\xext} \frak G_2$ and $\vp\colon Z\rTo \cal M^0$ is a continuous map, then $\frak G_1\underset{\cal M[Z]}{\xext} \frak G_2$.
\end{cor}

\begin{proof}
Suppose the crossed extension $\cal M$ is given through the maps $\xymatrix{\cal G_1^0 & \cal M^0 \ar[l]_{\tau} \ar[r]^{\s} & \cal G_2^0}$. 
Notice first that Corollary~\ref{cor:xext_pullback} guarantees that $\frak G_1[Z] \xext \frak G_1$ and $\frak G_2 \xext \frak G_2[Z]$, considering the maps $\tau\circ \vp\colon Z\rTo \cal G_1^0$ and $\s\circ \vp\colon Z\rTo \cal G_2^0$. Therefore, $\frak G_1[Z]\xext \frak G_2[Z]$. So, the main point here is that precisely $\cal M[Z]$ is a crossed extension of $\frak G_1[Z]$ and $\frak G_2[Z]$ over $Z$, which indeed follows from the commutative diagrams 
\[
\xymatrix{
\cal H_1[Z] \ar[rr]^{\vp^*\tau^*\p_1} \ar[rd]_{\vp^*\a_1} && \cal G_1[Z] \\
& \cal M[Z] \ar[ru]_{\vp^*\a_2}  \ar[rd]^{\vp^*\b_2} & \\
\cal H_2[Z] \ar[ru]^{\vp^*\b_1} \ar[rr]^{\vp^*\s^*\p_2} && \cal G_2[Z]
}
\]
where 
\begin{eqnarray*}
\vp^*\a_1\colon \cal H_1[Z] \ni (z,h_1)\mto (z,\a_1(\vp(z),h_1),z)\in \cal M[Z]\ , \\
\vp^*\a_2\colon \cal M[Z]\ni (z_1,m,z_2)\mto (z_1,pr_2(\a_2(m)),z_2)\in \cal G_1[Z]\ ,\\
\vp^*\b_1\colon \cal H_2[Z] \ni (z,h_2)\mto (z, \b_1(\vp(z),h_2),z)\in \cal M[Z], \ {\rm and\ } \\
\vp^*\b_2\colon \cal M[Z]\ni (z_1,m,z_2)\mto (z_1,pr_2(\b_2(m)),z_2)\in \cal G_2[Z]\ .
\end{eqnarray*}
\end{proof}



\subsection{Diamond product vs semidirect product}

Recall that if a crossed extension $\cal M$ of the crossed modules $\frak G_1$ and $\frak G_2$ is given by $(\cal M,\a_1,\a_2,\b_1,\b_2)$, we denote by $\bar{\cal M}$ the crossed extension $\frak G_2$ and $\frak G_1$ given by the the quintuple $(\cal M, \b_1,\b_2,\a_1,\a_2)$. We thus obtain the crossed extensions $\frak G_1\underset{\cal M\Dprod_{\frak G_2}\bar{\cal M}}{\xext}\frak G_1$ and $\frak G_2\underset{\bar{\cal M}\Dprod_{\frak G_1}\cal M}{\xext}\frak G_2$.

\begin{thm}\label{thm:M-M_bar}
Let $\frak G_1\underset{\cal M}{\xext}\frak G_2$ be a crossed extension. Then there is a Morita equivalence of topological groupoids 
\[
\cal M\Dprod_{\frak G_2}\bar{\cal M} \sim \bar{\cal M}\Dprod_{\frak G_1}\cal M
\]
\end{thm}

\begin{cor}\label{cor:extension_diamond}
Let $\cal A\overset{\iota}{\mono} \tilde{\cal G} \overset{\pi}{\epi} \cal G$ be a groupoid extension. Then $\cal A\rtimes \cal G\cong \tilde{\cal G}\ou{\Dprod}{\cal A}{\cal G}\tilde{\cal G}$.
\end{cor}

Before we prove the theorem, let us first assume the two crossed modules $\frak G_1$ and $\frak G_2$ have the same unit space $X$ and $\cal M\rrTo X$ is crossing with respect to the morphisms $\a_i,\b_i, i=1,2$. Then, the morphism $\b_1\colon \cal H_2\rTo \cal M$ induces a groupoid left action of $\cal H_2$ on the topological space $\cal H_1\rtimes \cal M$ as follows: the momentum map is the target map $t\colon \cal H_1\rtimes \cal M\rTo X$, and the action is given by $$h_2\cdot (h_1,m)\colonequals (h_1, \b_1(h_2)m)$$ for $h_2\in \cal H_2, (h_1,m)\in \cal H_1\rtimes \cal M$ with $s(h_2)=t(m)$. Similarly, the morphism $\a_1\colon \cal H_1\rTo \cal M$ define a groupoid left action of $\cal H_1$ on the topological space $\cal H_2\rtimes \cal M$ with respect to the target map of $\cal H_2\rtimes \cal M$ by $h_1\cdot (h_2,m)\colonequals (h_2,\a_1(h_1)m)$ when the product is defined.

\begin{df}
Let $\frak G_1, \frak G_2$, and $\cal M$ be as above. We define the {\em crossed semidirect product $\cal M$ over} $\frak G_2$ (resp. over $\frak G_1$) to be the quotient spaces $\frac{\cal H_1\rtimes\cal M}{\cal H_2}$ (resp. $\frac{\cal H_2\rtimes\cal M}{\cal H_1}$). We denote these spaces by $\cal H_1\rtimes_{\cal H_2}\cal M$ and $\cal H_2\rtimes_{\cal H_1}\cal M$, respectively.
\end{df}

\begin{lem}\label{lem:crossed_semidirect}
Let $\frak G_1$ and $\frak G_2$ be groupoid crossed modules with same unit space $X$ and let $\cal M\rrTo X$ be a crossing from $\frak G_1$ to $\frak G_2$. Then the groupoid structures of the semidirect products $\cal H_1\rtimes \cal M$ and $\cal H_2\rtimes \cal M$ descend onto the crossed semidirect products $\cal H_1\rtimes_{\cal H_2} \cal M$ and $\cal H_2\rtimes_{\cal H_1}\cal M$, respectively, and both with unit space $X$. Moreover, the groupoids $\cal H_1\rtimes \cal G_1$ and $\cal H_1\rtimes_{\cal H_2}\cal M$ are isomorphic. Similarly, if $\cal M$ is a crossed extension, then $\cal H_2\rtimes\cal G_2\cong \cal H_2\rtimes_{\cal H_1}\cal M$.
\end{lem}

\begin{proof}
We show the result only for $\cal H_1\rtimes_{\cal H_2} \cal M$ as analogous arguments apply to the other quotient. Let $h_2,k_2\in \cal H_2$, and $(h_1,m), (k_1,n)\in \cal H_1\rtimes \cal M$ such that the product $(h_1,\b_1(h_2)m)(k_1,\b_1(k_2)n)$ makes sense. Then, by using the axioms of a crossing, we get 
\[
\begin{array}{lcl}
(h_1,\b_1(h_2)m)(k_1,\b_1(k_2)n) & = & (h_1k_1^{(\a_2(\b_1(h_2)m))^{-1}},\b_1(h_2)m\b_1(k_2)n) \\
&=& (h_1k_1^{\a_2(m)^{-1}},\b_1(h_2)\b_1(k_2^{\b_2(m)^{-1}})mn)\\
&=& (h_1k_1^{\a_2(m)^{-1}},\b_1(h_2k_2^{\b_2(m)^{-1}})mn)
\end{array}
\]
This implies that the product $(h_1,\b_1(h_2)m)(k_1,\b_1(k_2)n)$ is equivalent to $(h_1,m)(k_1,n)$ in the quotient space. Therefore the partial product of the semidirect product groupoid descends to a partial product on the quotient space $[h_1,m]\cdot [k_1,n]\colonequals [h_1k_1^{\a_2(m)^{-1}},mn]$, where $[h_1,m]$ is the class of $(h_1,m)$ in the quotient. As for the inverse, we have 
\[
\begin{array}{lcl}
(h_1,\b_1(h_2)m)^{-1} & = & ((h_1^{\a_2(m)})^{-1},m^{-1}\b_1(h_2)^{-1})\\
& = & ((h_1^{\a_2(m)})^{-1}, \b_1((h_2^{-1})^{\b_2(m)})m^{-1})
\end{array}
\]
Thus, the formula $[h_1,m]^{-1}\colonequals[(h_1^{\a_2(m)})^{-1},m^{-1}]$ is a well defined inversion map $$\cal H_1\rtimes_{\cal H_2} \cal M\rTo \cal H_1\rtimes_{\cal H_2} \cal M.$$ Also, it is clear that the source and target maps of the semidirect product descent to the quotient. Finally, the last two statments follow from the exactness of the complexes 
\[
\cal H_1\stackrel{\a_1}{\rTo} \cal M\stackrel{\b_2}{\rTo}\cal G_2, \ {\rm and\ } \  \cal H_2\stackrel{\b_1}{\rTo} \cal M\stackrel{\a_2}{\rTo} \cal G_1
\]
\end{proof}

\begin{proof}[Proof of Theorem~\ref{thm:M-M_bar}]
Since the sequence of topological groupoid morphisms 
\[
\cal H_1[\cal M^0]\overset{\a_1}{\rTo}\cal M\overset{\b_2}{\rTo}\cal G_2[\cal M^0] \quad {\rm and } \quad  \cal H_2[\cal M^0]\overset{\b_1}{\rTo}\cal M \overset{\a_2}{\rTo}\cal G_1[\cal M^0]
\]
are exact, there are continuous maps $\tilde{\a}_1\colon \cal M\times_{\b_2,\cal G_2[\cal M^0],\b_2}\cal M\rTo \cal H_1[\cal M^0]$ and $\tilde{\b}_1\colon \cal M\times_{\a_2,\cal G_1[\cal M^0],\a_2}\cal M\rTo \cal H_2[\cal M^0]$ defined by the property that for $(m_1,m_2)\in \cal M\times_{\b_2,\cal G_2[\cal M^0],\b_2}\cal M$, $\tilde{\a}_1(m_1,m_2)$ is the unique element in $\cal H_1[\cal M^0]$ satisfying $m_2=\a_1(\tilde{a}_1(m_1,m_2))m_1$, and analogously, for each element $(m_1,m_2)\in \cal M\times_{\a_2,\cal G_1,\a_2}\cal M$, $\tilde{\b}_1(m_1,m_2)\in\cal H_2[\cal M^0]$ is uniquely determined by $m_2=\b_1(\tilde{\b}_1(m_1,m_2))m_1$. Let 
\[
\Phi_1\colon \cal M\Dprod_{\frak G_2}\bar{\cal M}\ni [m_1,m_2]\mto [\tilde{\a}_1(m_1,m_2), m_1]\in \cal H_1[\cal M^0]\rtimes_{\cal H_2[\cal M^0]}\cal M 
\]
and 
\[
\Phi_2\colon \bar{\cal M}\Dprod_{\frak G_1}\cal M\ni [m_1,m_2]\mto [\tilde{\b}_1(m_1,m_2),m_1]\in \cal H_2[\cal M^0]\rtimes_{\cal H_1[\cal M^0]}\cal M
\] 
Let us show these are the desired isomorphisms. Here again we will only deal with the former, as similar methods apply to the latter. $\Phi_1$ is well defined, for by uniqueness, $\tilde{\a}_1$ is invariant by the $\cal H_2[\cal M^0]$--action on the fibered product $\cal M\times_{\b_2,\cal G_2,\b_2}\cal M$; indeed, we get $$\tilde{\a}_1(\b_1(h_2)m_1,\b_1(h_2)m_2) = \tilde{\a}_1(m_1,m_2)^{\a_2(\b_1(h_2^{-1}))}=\tilde{\a}_1(m_1,m_2).$$
Moreover, $\Phi_1$ is a groupoid morphism; indeed, by simple calculations and by uniqueness of the elements $\tilde{\a}_1(m_1,m_2), \tilde{\a}_1(m_1',m_2'), \tilde{\a}_1(m_1m_1',m_2m_2')$, we have
\[
\tilde{\a}_1(m_1m_1',m_2m_2')=\tilde{\a}_1(m_1,m_2)\tilde{\a}_1(m_1',m_2')^{\a_2(m_1)^{-1}}
\]
which gives $\Phi_1([m_1,m_2][m_1'm_2'])=\Phi_1([m_1,m_2])\Phi_1([m_1',m_2'])$ whenever these products make sense. Thus, since $\Phi_1$ is clearly continuous, it is a strict morphism of topological groupoids. Furthermore, it is easy to verify that the map 
\[
\Psi_1\colon \cal H_1[\cal M^0]\rtimes_{\cal H_2[\cal M^0]}\cal M\ni [h_1,m]\mto [m,\a_1(h_1)m]\in \cal M\Dprod_{\frak G_2}\bar{\cal M}
\]
is a well defined strict morphism of topological groupoids and that $\Phi_1\circ \Psi_1=\Id$ and $\Psi_1\circ \Phi_1=\Id$. This isomorphisms, combined with Corollary~\ref{cor:Morita_semidirect} and Lemma~\ref{lem:crossed_semidirect} the desired Morita equivalence of topological groupoids. 
\end{proof}


\section{$2$--groupoids of crossed extensions}\label{sec:2-cat}

We define for each pair of crossed modules of topological groupoids $\frak G_1$ and $\frak G_2$, a weak $2$--groupoid $\XExt(\frak G_1,\frak G_2)$ whose objects are crossed extensions of $\frak G_1$ and $\frak G_2$. 

Throughout, we denote by $\fXExt$ the collection of all triples of the form $(\frak G_1,\cal M,\frak G_2)$ where $\frak G_1,\frak G_2$ are crossed modules of topological groupoids and $\frak G_1\underset{\cal M}{\xext}\frak G_2$ is a crossed extension. Such triples will be represented by bold capital letters $\bf A, \bf B$, etc. Moreover, by virtue of the results and observations of the preceding sections, we will avoid complicated notations and calculations by assuming that for ${\bf A}=(\frak G_1, \cal M, \frak G_2)$, all the groupoids involved share the same unit space $X$ and that $\tau=\s=\Id_X$; by abuse of language, the triple $\bf A$ will be referred to as a crossed extension with unit space $X$.
 
\subsection{Preliminaries on generalized morphisms of topological groupoids}

We revisit generalized morphisms of topological groupoids and give a few results we are going to use later. 

Let $\cal M\rrTo \cal M^0$ and $\cal N\rrTo\cal N^0$ be topological groupoids. Associated to any generalized morphism $\xymatrix{\cal M^0& Z\ar[l]_{\vp} \ar[r]^{\psi} &\cal N^0}$, there is a continuous map $g^Z\colon Z\times_{\cal M^0}Z\rTo \cal N$ completely determined by the property that for $(z,z')\in Z\times_{\cal M^0}Z$, $g^Z(z,z')$ is the unique element in $\cal N$ such that $z'=z\cdot g^Z(z,z')$. The existence and uniqueness of $g^Z(z,z')\in \cal N$ come from the fact that the right action of $\cal N$ on $Z$ is free and principal.

\begin{lem}
Let $\cal M, \cal N$, be topological groupoids and $\xymatrix{\cal M^0& Z\ar[l]_{\vp} \ar[r]^{\psi} &\cal N^0}$ a generalized morphism. Then, the correspondence 
\[
\cal M[Z]\ni (z,m,z')\mto (z,g^Z(z,mz'),z')\in \cal N[Z]
\]
defines a strict morphism $\Phi^Z$ of topological groupoids. Furthermore, $\Phi^Z$ is an isomorphism if and only if $Z$ is a Morita equivalence.
\end{lem}

\begin{proof}
For $(z,m,z')\in \cal M[Z]$, we have $\vp(mz')=t(m)=\vp(z)$, hence there exists a unique $g^Z(z,mz')\in \cal N$ with $mz'=z\cdot g^Z(z,mz')$, and the map $\Phi^Z$ is well defined and continuous. Moreover, given $(z',m',z'')\in \cal M[Z]$, then by uniqueness we get $g^Z(z,mm'z'')=g^Z(z,mz')g^Z(z',m'z'')$; therefore, $$\Phi^Z((z,m,z')(z',m',z''))=\Phi^Z(z,m,z')\Phi^Z(z',m',z''),$$
which means $\Phi^Z$ is a strict morphism. \\
Suppose $\Phi^Z$ is an isomorphism. Then for all $(w,w')\in Z\times_{\cal N^0}Z$, the triple $(w,\psi(w),w')$ is an element in $\cal N[Z]$, and there is a unique element $m\in \cal M$ such that $\psi(w)=g^Z(w,mw')$ and $w=m w'$. This implies that $\psi\colon Z\rTo \cal N^0$ is a $\cal M$--principal bundle, thus $Z$ is a Morita equivalence. Conversely, if $Z$ is a Morita equivalence, $\Phi^Z$ is obviously an isomorphism of topological groupoids.
\end{proof}

\begin{ex}
Let $f\colon \cal M\rTo \cal N$ be a strict morphism of topological groupoids and $Z_f\colonequals \cal M^0\times_{f,\cal N^0, t}\cal N$ its induced generalized morphism. Recall that the momentum maps of $Z_f$ are 
\[
\xymatrix{
\cal M^0 & Z_f \ar[l]_-{pr_1} \ar[r]^{t\circ pr_2} & \cal N^0
}
\]
 the left $\cal M$--action on $Z_f$ is given $m\cdot (x,n)\colonequals (t(m),f(m)n)$ when $s(m)=x$, and the right $\cal N$--action in $(x,n)\cdot n'\colonequals (x,nn')$ when $s(n)=t(n')$. Then the strict morphism $\Phi^{Z_f}\colon \cal M[Z_f]\rTo \cal N[Z_f]$ if given by 
 \[
 \Phi^{Z_f}((x_1,n_1),m,(x_2,n_2))=((x_1,n_1),n_1^{-1}f(m)n_2,(x_2,n_2))
 \]
In particular, we recover $f$ on elements of the form $((x_1,f(x_1)),m,(x_2,f(x_2)))$; that is, 
\[
\Phi^{Z_f}((x_1,f(x_1)),m,(x_2,f(x_2)))=((x_1,f(x_1)),f(m),(x_2,f(x_2))).
\] 
Moreover, $Z_f$ is a Morita equivalence if and only if $f$ is an isomorphism of topological groupoids.
\end{ex}

\begin{lem}
Let $\cal M\rrTo\cal M^0, \cal N\rrTo \cal N^0, \cal R\rrTo\cal R^0$ be topological groupoids and $\cal M\overset{Z_1}{\rTo}\cal N\overset{Z_2}{\rTo} \cal R$ generalized morphisms with respect to the maps 
\[
\xymatrix{
\cal M^0 &Z_1\ar[l]_{\vp_1} \ar[r]^{\psi_1} &\cal N^0 & Z_2\ar[l]_{\vp_2} \ar[r]^{\psi_2} & \cal R^0
}
\]
Consider the generalized morphism $\xymatrix{\cal M \ar[r]^-{Z_1\underset{\cal N}{\bullet}Z_2} &\cal R}$. Then, for all $[z_1,z_2], [z_1',z_2']\in Z_1\underset{\cal N}{\bullet}Z_2$ such that $\vp_1(z_1)=\vp_1(z_1')$, we have 
\[
g^{Z_1\underset{\cal N}{\bullet}Z_2}([z_1,z_2],[z_1',z_2'])=g^{Z_2}(z_2,g^{Z_1}(z_1,z_1')\cdot z_2')
\]
Therefore the morphism $\Phi^{Z_1\underset{\cal N}{\bullet}Z_2}$ is given by 
\[
\cal M[Z_1\underset{\cal N}{\bullet}Z_2]\ni ([z_1,z_2],m,[z_1',z_2'])\mto ([z_1,z_2],g^{Z_2}(z_2,g^{Z_1}(z_1,mz_1')z_2'),[z_1',z_2'])\in \cal R[Z_1\underset{\cal N}{\bullet}Z_2]
\]
\end{lem}

\begin{proof}
By the very definition of the actions of $\cal M$ and $\cal R$ on the space $Z_1\underset{\cal N}{\bullet}Z_2$, we have $[z_1',z_2']=[z_1,z_2]\cdot r$ if and only if there is $n\in \cal N$ such that $z_1'=z_1n$ and $nz_2'=z_2r$; which, from the definition of the functions $g^{Z_1}$ and $g^{Z_2}$, necessarily means $n=g^{Z_1}(z_1,z_1')$, hence $r=g^{Z_2}(z_2,g^{Z_1}(z_1,z_1')\cdot z_2')$.
\end{proof}


\subsection{Equivalences of crosssed extensions}

\begin{df}\label{def:SCM}
Let ${\bf A}=(\frak G_1,\cal M, \frak G_2)$ and ${\bf B}=(\frak G_3,\cal N,\frak G_4)$ be objects in $\fXExt$ with unit spaces $X$ and $Y$, respectively. A {\em Homomorphism of crossed extensions} from ${\bf A}$ to ${\bf B}$ is a commutative diagram as below
\begin{eqnarray}
\xymatrix{
&& \cal H_3 \ar[rr]^-{\p_3} \ar@/^/[rdd]^-{\d_1}|!{[d];[rr]}\hole && \cal G_3 \\
\cal H_1 \ar[rr]^-{\p_1} \ar[rru]^{\chi_1} \ar@/^/[rdd]_-{\a_1} && \cal G_1 \ar[rru]^{\chi_1} &&\\
 && & \cal N \ar[ruu]_{\d_2} \ar[rdd]^{\t_2} &\\
& \cal M \ar[ruu]_{\a_2} \ar[rdd]_{\b_2} \ar[rru]^{\Phi} &&& \\
  && \cal H_4 \ar[rr]^{\p_4} \ar@/_/[ruu]^{\t_1} && \cal G_4 \\
\cal H_2 \ar[rr]^{\p_2} \ar@/_/[ruu]^{\b_1} \ar[rru]^{\chi_2} && \cal G_2 \ar[rru]^{\chi_2} & &
}
\end{eqnarray} 
where the top and bottom squares are strict morphisms of crossed modules, and $\Phi$ is a groupoid strict morphism from $\cal M\rrTo X$ to $\cal N\rrTo Y$ such that for all $x\in X$, the induced maps 
\begin{eqnarray}\label{SCM1} 
\Phi\colon \a_1(\cal H_1)^x \rTo \d_1(\cal H_3)^{\chi_1(x)}
\end{eqnarray} 
and
\begin{eqnarray}\label{SCM2} 
\Phi\colon \b_1(\cal H_2)^x \rTo \t_1(\cal H_4)^{\chi_2(x)}
\end{eqnarray}
are isomorphisms of topological groups.

We write ${\bf A}\overset{(\chi_1,\Phi, \chi_2)}{\rTo}{\bf B}$ or just ${\bf A}\overset{\Phi}{\rTo}{\bf B}$ for a  homomorphism of crossed extensions. 
\end{df}

In particular, a homomorphism of crossed extensions induces strict morphisms of groupoid extensions (\cite{Laurent-Stienon-Xu:Non-Abelian_Gerbes}) 
\[
\xymatrix{
\cal H_1 \ar[r]^{\a_1} \ar[d]_{\chi_1} & \cal M\ar[r]^{\b_2} \ar[d]_{\Phi} & \cal G_2 \ar[d]_{\chi_2} \dar[r] &X\ar[d] \\
\cal H_1' \ar[r]^{\a_1'} & \cal M' \ar[r]^{\b_2'} & \cal G_2' \dar[r] & X'
}
\]
and 

\[
\xymatrix{
\cal H_2 \ar[r]^{\b_1} \ar[d]_{\chi_2} & \cal M\ar[r]^{\a_2} \ar[d]_{\Phi} & \cal G_1 \ar[d]_{\chi_1} \dar[r] & X \ar[d]\\
\cal H_2' \ar[r]^{\b_1'} & \cal M' \ar[r]^{\a_2'} & \cal G_1' \dar[r] & X' 
}
\]

Also we deduce the following from the definition 

\begin{lem}
Given a homomorphism of crossed extensions as above, the commutative diagrams below 
\[
\xymatrix{
\cal H_1 \ar[d]_{\chi_1} \ar[r]^{\a_1} & \cal M \ar[d]^{\Phi} & & \cal H_2 \ar[d]_{\chi_2} \ar[r]^{\b_1} & \cal M \ar[d]^{\Phi} \\
\cal H_3 \ar[r]^{\d_1} & \cal N & & \cal H_4 \ar[r]^{\t_1} & \cal N
}
\]
are strict morphisms of groupoid crossed modules.
\end{lem}

Composition of strict morphisms ${\bf A}\rTo {\bf B} \rTo {\bf C}$ is defined in the obvious way. 

\begin{df}
A homomorphism of crossed extensions ${\bf A}\stackrel{(\chi_1,\Phi,\chi_2)}{\rTo}{\bf B}$ is called an {\em equivalence} of crossed extensions if 
\begin{itemize}
\item[(i)] the strict morphisms of groupoids crossed modules $\chi_1$ and $\chi_2$ are hypercovers; and
\item[(ii)] $\Phi:\cal M\rTo \cal N$ is a weak equivalence of topological groupoids. 
\end{itemize}
\end{df}

It is clear that composition of equivalences is an equivalence of crossed extensions.

\begin{ex}[Pullback]\label{ex:pullback_strict}
Let ${\bf A}=(\frak G_1,\cal M,\frak G_2)$ be a crossed extension with unit space $X$. Given a continuous map $\vp\colon Z\rTo X$, the canonical projections $\frak G_i[Z]\rTo \frak G_i, i=1,2$, and $\cal M[Z]\rTo \cal M$ define a homomorphism ${\bf A}[Z] \rTo {\bf A}$, where ${\bf A}[Z]$ is the pullback $(\frak G_1[Z],\cal M[Z],\frak G_2[Z])$. If in addition $\vp$ is surjective, then this homomorphism is an equivalence of crossed extensions.
\end{ex}

\begin{pro}\label{pro:unit_w-morphism}
Let $\bf A=(\frak G_1, \cal M, \frak G_2)\in \fXExt$ where the crossed extension $\frak G_1\underset{\cal M}{\xext}\frak G_2$ is defined with respect to the maps 
\[
\cal G_1^0 \overset{\tau}{\lTo} \cal M^0 \overset{\s}{\rTo} \cal G_2^0
\]
Consider the canonical Morita equivalence $\cal M^0\overset{t}{\lTo}\cal M\overset{s}{\rTo}\cal M^0$ between the topological groupoids $\cal M$ and itself, and let $\Phi^{\cal M}:t^*\cal M\rTo s^*\cal M$ be the associated strict morphism defined in the previous section. Then, there is an equivalence of crossed extensions 
$
\xymatrix{
 t^*{\bf A }\ar[rr]^-{(\a,\Phi^{\cal M},\b)} && s^*{\bf A}
}
$
where $t^*{\bf A}=(t^\ast \frak G_1, t^*\cal M, t^\ast\frak G_2), s^*{\bf A} = (s^\ast\frak G_1, s^\ast\cal M, s^\ast\frak G_2)$,  $\Phi^{\cal M}\colon \cal M[\cal M]\rTo \cal M[\cal M]$ is the groupoid strict morphism associated to the generalized morphism $\cal M$ constructed in the previous section, and where, for $i=1,2$, we have identified $t^\ast\frak G_i$ and $s^\ast\frak G_i$ with $(\tau\circ t)^\ast\frak G_i$ and $(\s\circ s)^*\frak G_i$.
\end{pro}

\begin{proof}
We identify the element $(m_1,(t(m_1),g,t(m_2)),m_2)\in t^\ast \cal G_i[\cal M^0]$ with $(m_1,g,m_2)\in t^\ast\cal G_i$, and use similar identifications for the pullbacks through the source map $s$ and for $\cal H_i, i=1,2$.  
We define the strict morphisms $\a=(\a^l, \a^r)$ and $\b=(\b^l,\b^r)$ of crossed modules by setting 
\begin{eqnarray*}
\a^l\colon t^\ast\cal H_1 \ni (m,h_1)\mto (m, h_1^{\a_2(m)})\in s^\ast\cal H_1 \\
\a^r\colon t^\ast\cal G_1 \ni (m_1,g_1,m_2)\mto (m_1,\a_2(m_1^{-1})g\a_2(m_2),m_2)\in s^\ast\cal G_1 \\
\b^l\colon t^\ast\cal H_2 \ni (m,h_2)\mto (m, h_1^{\b_2(m)})\in s^\ast\cal H_2 \\
\b^r\colon t^\ast\cal G_2 \ni (m_1,g_2,m_2)\mto (m_1,\b_2(m_1^{-1})g_2\b_2(m_2),m_2)\in s^\ast\cal G_2
\end{eqnarray*}
Hence, by construction, the diagram  
\begin{eqnarray*}
\xymatrix{
&& s^\ast\cal H_1 \ar[rr]^-{s^*\p_1} \ar@/^/[rdd]^-{s^\ast\a_1}|!{[d];[rr]}\hole && s^\ast\cal G_1 \\
t^\ast\cal H_1 \ar[rr]^-{t^\ast\p_1} \ar[rru]^{\a^l} \ar@/^/[rdd]_-{t^\ast\a_1} && t^\ast\cal G_1 \ar[rru]^{\a^r} &&\\
 && & s^\ast\cal M \ar[ruu]_{s^\ast\a_2} \ar[rdd]^{s^\ast\b_2} &\\
& t^\ast\cal M \ar[ruu]_{t^\ast\a_2} \ar[rdd]_{t^\ast\b_2} \ar[rru]^{\Phi^{\cal M}} &&& \\
  && s^\ast\cal H_2 \ar[rr]^{s^\ast\p_2} \ar@/_/[ruu]^{s^\ast\b_1} && s^\ast\cal G_2 \\
t^\ast\cal H_2 \ar[rr]^{t^\ast\p_2} \ar@/_/[ruu]^{t^\ast\b_1} \ar[rru]^{\b^l} && t^\ast\cal G_2 \ar[rru]^{\b^r} & &
}
\end{eqnarray*}
commutes.
\end{proof}

In particular, if ${\bf A}=(\frak G_1,\cal M, \frak G_2)$ is a crossed extension with unit space $X$, then $s^\ast \cal M=t^\ast\cal M=\cal M$, $\Phi^{\cal M}=\Id_{\cal M}$, and $\id_{\bf A}\colonequals(\a, \Id_{\cal M}, \b)$ is an equivalence $t^\ast {\bf A}\rTo s^\ast{\bf A}$.

\subsection{Exchangers}\label{subs:exchangers}

Our goal is to investigate a more concrete interpretation of weak morphisms of crosssings and crossed extensions. 

\begin{df}
Let ${\bf A}=(\frak G_1,\cal M,\frak G_2)$ and ${\bf B}=(\frak G_3,\cal N,\frak G_4)$ be crossed extensions with unit spaces $X$ and $Y$, respectively, through the diagrams

\[
\xymatrix{
\cal H_1\ar[rr]^{\p_1} \ar[rd]_{\a_1} && \cal G_1  & & \cal H_3 \ar[rd]_{\d_1} \ar[rr]^{\p_3} && \cal G_3 \\
& \cal M \ar[ru]_{\a_2} \ar[rd]^{\b_2} & && & \cal N\ar[ru]_{\d_2} \ar[rd]^{\t_2} & && \\
\cal H_2 \ar[rr]^{\p_2} \ar[ru]^{\b_1} && \cal G_2 && \cal H_4 \ar[rr]^{\p_4} \ar[ru]^{\t_1} && \cal G_4 
}
\]

\noindent A {\em semi-exchanger} from ${\bf A}$ to ${\bf B}$ consists of a generalized morphism of topological groupoids 
\[
X \overset{\vp}{\lTo} P \overset{\psi}{\rTo} Y
\]
such that the following properties hold. 
\begin{enumerate}[label=(\textbf{E\arabic*}), align=left]
\item\label{E1} The left groupoid action of $\cal H_1$ and the right action of $\cal H_3$ on $P$ induced by the groupoid morphisms $\cal H_1\overset{\a_1}{\rTo}\cal M$ and $\cal H_3\overset{\d_1}{\rTo}\cal N$, respectively, are free and have the same orbit space $\frac{P}{\cal H_1}=\frac{P}{\cal H_3}$.
\item\label{E2} The left groupoid action of $\cal H_2$ and the right action of $\cal H_4$ on $P$ induced by the groupoid morphisms $\cal H_2\overset{\b_1}{\rTo} \cal M$ and $\cal H_4\overset{\t_1}{\rTo}\cal N$, respectively, are free and have the same orbit space $\frac{P}{\cal H_2}= \frac{P}{\cal H_4}$.
\end{enumerate}
\noindent An {\em exchanger} between ${\bf A}$ and ${\bf B}$ is a semi-exchanger $P$ which defines a Morita equivalence bewteen $\cal M$ and ${\cal N}$. We write ${\bf A} \overset{P}{\sxc} {\bf B}$ for a semi-exchanger and sometimes ${\bf A}\overset{P}{\xc} {\bf B}$ for an exchanger. 
\end{df}

\begin{ex}\label{ex:unit_exchanger}
If ${\bf A}=(\frak G_1,\cal M,\frak G_2)$ with unit space $X$ is a crossed extension, it is easy to check that the canonical generalised morphism 
\[
X \stackrel{t}{\lTo} \cal M \stackrel{s}{\rTo} X
\]
is an exchanger ${\bf A}\xc {\bf A}$. We call this exchanger the {\em trivial exchanger} and denote it by $\bf I_{\cal M}$.
\end{ex}

\begin{lem}
Assume ${\bf A}$ and ${\bf B}$ are crossed extensions as above and ${\bf A}\overset{P}{\RTo} {\bf B}$ is an exchanger. Then 
\begin{itemize}
\item[(i)] The left actions of $\cal H_1$ and $\cal H_2$ on $P$ commute; that is $\a_1(h_1)\cdot (\b_1(h_2)\cdot p)=\b_1(h_2)\cdot (\a_1(h_1)\cdot p)$ for all $p\in P, h_1\in \cal H_1$ and $h_2\in \cal H_2$ such that the actions are defined.
\item[(ii)] The right actions of $\cal H_3$ and $\cal H_4$ on $P$ commute; that is, $\d_1(h_3)\cdot (\t_1(h_4)\cdot p)=\t_1(h_4)\cdot (\d_1(h_3)\cdot p)$, for all $p\in P, h_3\in \cal H_3$, and $h_4\in \cal H_4$ such that the actions are defined.
\item[(iii)] $P$ is a generalized morphism of groupoid extensions from $\cal H_1 \stackrel{\a_1}{\rTo} \cal M\stackrel{\b_2}{\rTo} \cal G_2$ to $\cal H_3\stackrel{\d_1}{\rTo}\cal N\stackrel{\t_2}{\rTo} \cal G_4$;
\item[(iv)] $P$ is a generalized morphism of groupoid extensions from $\cal H_2 \stackrel{\b_1}{\rTo} \cal M\stackrel{\a_2}{\rTo} \cal G_1$ to $\cal H_4\stackrel{\t_1}{\rTo}\cal N\stackrel{\d_2}{\rTo} \cal G_3$.
\end{itemize}
\end{lem}

\begin{proof}
(i) and (ii) follow from Proposition~\ref{pro:images_crossing} and from the associativity of the groupoid actions of $\cal M$ and $\cal N$ on $P$. (iii) and (iv) are immediate consequences of axioms~\ref{E1} and~\ref{E2} and the very definition of generalized morphisms of groupoid extensions (see for instance~\cites{Moutuou:Real.Cohomology, Laurent-Stienon-Xu:Non-Abelian_Gerbes}).
\end{proof}

\begin{pro}\label{pro:equivalence_vs_exchanger}
Homomorphisms (resp. equivalences) of crossed extensions are semi-exchangers (resp. exchangers).
\end{pro}

\begin{proof}
Let ${\bf A}=(\frak G_1,\cal M,\frak G_2)$ and ${\bf B}=(\frak G_3,\cal N,\frak G_4)$ be objects in $\fXExt$ with unit spaces $X$ and $Y$, respectively. If $\xymatrix{{\bf A} \ar[r]^-{\Phi} & {\bf B}}$ is a homomorphism of crossed extensions, where $\Phi=(\chi_1,\Phi,\chi_2)$ as in Definition~\ref{def:SCM}, we let $P_{\Phi}=X\times_{\Phi,t}\cal N$ be the generalized morphism of topological groupoids from $\cal M$ to $\cal N$ 
\[
\xymatrix{
X & P_{\Phi} \ar[l]_{pr_1} \ar[r]^-{pr_2\circ s} & Y
}
\]
with the left $\cal M$--action and right $\cal N$--action $m\cdot (s(m),n)=(t(m),\Phi(m)n)$ and $(x,n)\cdot n'=(x, nn')$, respectively. Then, the induced left $\cal H_1$--action and right $\cal H_3$--action of $P_{\Phi}$ are 
\[
h_1\cdot (s(h_1),n)=(s(h_1),\Phi(\a_1(h_1))n)=(s(h_1), \d_1(\chi_1(h_1))n) \quad {\rm and\ } \quad (x,n)\cdot h_3=(x,n\d_1(h_3)).
\]
It follows from the isomorphism~\eqref{SCM1} that these two actions are free and have the same orbit space $$\frac{P_{\Phi}}{\cal H_1}=\frac{P_{\Phi}}{\cal H_3}=X\times_{\chi_2,t}\cal G_4.$$ Similarly, we deduce from isomorphism~\eqref{SCM2} that the induced left $\cal H_2$--action and right $\cal H_4$--action on $P_{\Phi}$ below
\[
h_2\cdot (s(h_2),n)=(s(h_2),\Phi(\b_1(h_2))n)=(s(h_2),\t_1(\chi_2(h_2))n) \quad {\rm and\ } \quad (x,n)\cdot h_4=(x, n\t_1(h_4))
\]
are free and have the same orbit space $$\frac{P_{\Phi}}{\cal H_2}=\frac{P_{\Phi}}{\cal H_4}=X\times_{\chi_1,t}\cal G_3.$$
Therefore, $P_\Phi$ is a semi-echanger from ${\bf A}$ to ${\bf B}$. Obviously, if $\Phi:{\bf A}\rTo {\bf B}$ is an equivalence, the generalised morphism $P_{\Phi}$ is a Morita equivalence, hence an exchanger ${\bf A}\overset{P_\Phi}{\xc}{\bf B}$. 
\end{proof}

\noindent In particular the identity  homomorphism $\id_{\bf A}$ induces the exchanger ${\bf A} \stackrel{\bf I_{\cal M}}{\xc} {\bf A}$.

Semi-exchangers can be composed as it can be seen in the straightforward proposition below. 

\begin{pro}\label{pro:exchangers_compo}
Let ${\bf A} = (\frak G_1, \cal M,\frak G_2), {\bf B}=(\frak G_3,\cal N,\frak G_4), {\bf C}=(\frak G_5,\cal R,\frak G_6)$ be crossed extensions with unit spaces $X,Y$, and $Z$, respectively.  Assume ${\bf A}\stackrel{P}{\sxc} {\bf B} \stackrel{Q}{\sxc} {\bf C}$ are semi-echangers. Then, the generalized morphism $\cal M\overset{P\underset{\cal N}{\bullet}Q}{\rTo} \cal R$ is a semi-exchanger from ${\bf A}$ to ${\bf C}$. The semi-exchanger ${\bf A} \overset{P\underset{\cal N}{\bullet}Q}{\sxc} {\bf B}$ will be called the {\em vertical composition} of $P$ with $Q$.
\end{pro}

Notice that the vertical composition of two exchangers is an exchanger.  

\begin{df}
Given ${\bf A}, {\bf B}\in \fXExt$, the collection of all semi-exchangers from ${\bf A}$ to ${\bf B}$ is denoted by $\fXExt({\bf A}, {\bf B})$.
\end{df}

The result below draws a link between exchangers and equivalences of crossed extensions. 

\begin{thm}
Let ${\bf A}, {\bf B}\in \fXExt$. Then there is an exchanger ${\bf A}\stackrel{P}{\xc}{\bf B}$ if and only if there are a third crossed extension ${\bf P}$ and two equivalences of crossed extensions 
\begin{eqnarray}\label{eq:decomposition_exchanger}
{\bf A} \lTo {\bf P} \rTo {\bf B}.
\end{eqnarray}
\end{thm}

\begin{proof} ($\RTo$). Assume ${\bf A}=(\frak G_1,\cal M,\frak G_2)$ and ${\bf B}=(\frak G_3,\cal N,\frak G_4)$ with unit spaces $X$ and $Y$, and crossings $(\cal M,\a_1,\a_2,\b_1,\b_2)$ and $(\cal N,\d_1,\d_2,\t_1,\t_2)$, respectively. Suppose 
\[
X\stackrel{\vp}{\lTo} P\stackrel{\psi}{\rTo} Y
\]
is an exchanger from ${\bf A}$ to ${\bf B}$. Form the projective product groupoid $\cal M\underset{P}{\ast}\cal N\rrTo P$ over the Morita equivalence $P$. Define the topological space $\cal M\ou{\ast}{\cal H_2,\cal H_4}{P}\cal N$ as the quotient of $\cal M\underset{P}{\ast}\cal N$ out of the equivalence relation given by $(m,p_1,p_2,n)\sim (m',p_1,p_2, n')$ if and only if $m'=m\b_1(h_2)$ and $n'=\t_1(h_4)n$ for some $h_2\in \cal H_2^{\vp(p_2)}, h_4\in \cal H_4^{\psi(p_1)}$. Then it is easy to check that the groupoid structure of $\cal M\underset{P}{\ast}\cal N$ descends to $\cal M\ou{\ast}{\cal H_2,\cal H_4}{P}\cal N$ turning it into a topological groupoid with unit space $P$. Moreover, we get the crossed module of topological groupoids 
\[
{\frak G}_1^P = \left(\cal H_1\underset{P}{\ast}\cal H_3 \stackrel{\p_1^P}{\xTo} \cal M\ou{\ast}{\cal H_2,\cal H_4}{P}\cal N\right)
\]
by setting 
\begin{itemize}
\item $\p_1^P(h_1,p,h_3)\colonequals [\a_1(h_1),p,p,\t_1(h_3)]$, for $(h_1,p,h_3)\in \cal H_1\underset{P}{\ast}\cal H_3$;
\item $(h_1,p_1,h_3)^{[m,p_1,p_2,n]}\colonequals (h_1^{\a_2(m)},p_2,h_3^{\d_2(n)})$. 
\end{itemize}
With analogous reasoning, we form the topological groupoid $\cal M\ou{\ast}{\cal H_1,\cal H_3}{P}\cal N\rrTo P$ as the quotient of $\cal M\underset{P}{\ast}\cal N$ by the equivalence relation $(m,p_1,p_2,n)\sim (m',p_1,p_2,n')$ if and only if $m'=m\a_1(h_1)$ and $n'=\d_1(h_3)n$ for some $h_1\in \cal H_1^{\vp(p_2)}, h_3\in \cal H_3^{\psi(p_1)}$, hence we get the groupoid crossed module
\[
\frak G_2^P = \left(\cal H_2\underset{P}{\ast}\cal H_4\stackrel{\p_2^P}{\xTo}\cal M\ou{\ast}{\cal H_1,\cal H_3}{P}\cal N\right)
\]
with $\p_2^P(h_2,p,h_4)\colonequals [\b_1(h_2),p,p,\t_1(h_4)]$. 
Furthermore, we obtain the crossed extension ${\bf P}=(\frak G_1^P,\cal M\underset{P}{\ast}\cal N,\frak G_2^P)$ as
\[
\xymatrix{
\cal H_1\underset{P}{\ast}\cal H_3\ar[rr]^{\p_1^P} \ar[rd]_{\a_1\ast\d_1} && \cal M\ou{\ast}{\cal H_2,\cal H_4}{P}\cal N \\ 
& \cal M\underset{P}{\ast}\cal N\ar[ru]_{\pi}\ar[rd]^{\pi'} & \\
\cal H_2\underset{P}{\ast}\cal H_4 \ar[rr]^{\p_2^P} \ar[ru]^{\b_1\ast\t_1} && \cal M\ou{\ast}{\cal H_1,\cal H_3}{P}\cal N
}
\]
where $\pi$ and $\pi'$ are the quotient maps, and 
\[
\a_1\ast\d_1(h_1,p,h_3)\colonequals (\a_1(h_1),p,p,\d_1(h_3)), \ \ \b_1\ast\t_1(h_2,p,h_4)\colonequals (\b_1(h_2),p,p,\t_1(h_4)).
\] 

Moreover, since $P$ is a Morita equivalence, the maps  
\[
\cal M\underset{P}{\ast}\cal N\ni (m,p_1,p_2,n)\mto (p_1n,m,p_2) \in P\times_{\vp,t}\cal M\times_{s,\vp}P,
\] 
and
\[
 \cal M\underset{P}{\ast}\cal N\ni (m,p_1,p_2,n)\mto (p_1,n,mp_2)\in P\times_{\psi,t}\cal N\times_{s,\psi}P
\]
are homeomorphisms, and the maps 
\[
P\times_{\vp,t}\cal M\ni (p,m)\mto s(m)\in X, \ {\rm and\ } \ P\times_{\psi,t}\cal N\ni (p,n)\mto s(n)\in Y
\]
are continuous surjections; therefore the canonical projections 
\[
\cal M\stackrel{pr_1}{\lTo} \cal M\underset{P}{\ast}\cal N\stackrel{pr_4}{\rTo}\cal N
\]

are equivalences of topological groupoids. Finally, we construct the homomorphisms of crossed extensions below

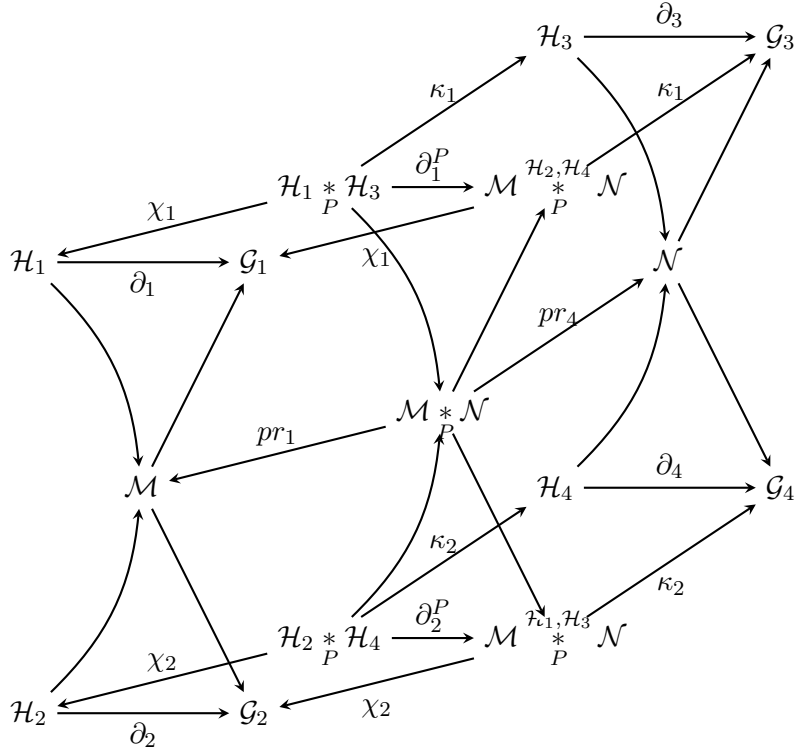
\[
\begin{tikzpicture}
\node (A) at (-2.5,1) {$\cal H_1$};
\node (B) at (0.5,1) {$\cal G_1$};
\node (M) at (-1,-2) {$\cal M$};
\node (C) at (-2.5,-5) {$\cal H_2$};
\node (D) at (0.5,-5) {$\cal G_2$};

\node (A1) at (1.5,2) {$\cal H_1\underset{P}{\ast}\cal H_3$};
\node (B1) at (4.5,2) {$\cal M\ou{\ast}{\cal H_2,\cal H_4}{P}\cal N$};
\node (M1) at (3,-1) {$\cal M\underset{P}{\ast}\cal N$};
\node (C1) at (1.5,-4) {$\cal H_2\underset{P}{\ast}\cal H_4$};
\node (D1) at (4.5,-4) {$\cal M\ou{\ast}{\cal H_1,\cal H_3}{P}\cal N$};

\node (A2) at (4.5,4) {$\cal H_3$};
\node (B2) at (7.5,4) {$\cal G_3$};
\node (M2) at (6,1) {$\cal N$};
\node (C2) at (4.5,-2) {$\cal H_4$};
\node (D2) at (7.5,-2) {$\cal G_4$};

\draw[1ar] (A) -- (B) node[midway, below] {$\p_1$};
\draw[1ar] (A) to[bend left=20] (M);
\draw[1ar] (M) -- (D);
\draw[1ar] (M) -- (B);
\draw[1ar] (C) to[bend right=20] (M); 
\draw[1ar] (C) -- (D) node[midway, below] {$\p_2$};

\draw[1ar] (A1) -- (B1) node[midway, above,swap] {$\p_1^P$};
\draw[1ar] (A1) to[bend left=20] (M1);
\draw[1ar] (M1) edge[-,line width=6pt, draw=white] (D1) edge (D1);
\draw[1ar] (M1) -- (B1);
\draw[1ar] (C1) to[bend right=20] (M1); 
\draw[1ar] (C1) -- (D1) node[midway,above, swap] {$\p_2^P$};

\draw[1ar] (A2) -- (B2) node[midway, above] {$\p_3$};
\draw[1ar] (A2) to [bend left=20] (M2);
\draw[1ar] (M2) -- (D2);
\draw[1ar] (M2) -- (B2);
\draw[1ar] (C2) to[bend right=20] (M2);
\draw[1ar] (C2) -- (D2) node[midway,above] {$\p_4$}; 
\draw[1ar] (A1) -- (A) node[midway,above] {$\chi_1$};
\draw[1ar] (B1) -- (B) node[midway,below] {$\chi_1$}; 
\draw[1ar] (M1) -- (M) node[midway,above] {$pr_1$};
\draw[1ar] (C1) -- (C) node[midway,above] {$\chi_2$};
\draw[1ar] (D1)  -- (D) node[midway,below] {$\chi_2$}; 

\draw[1ar] (A1) -- (A2) node[midway,above] {$\k_1$};
\draw[1ar] (B1) -- (B2) node[midway,above] {$\k_1$};
\draw[1ar] (M1) -- (M2) node[midway,above] {$pr_4$};
\draw[1ar] (C1) -- (C2) node[midway,above] {$\k_2$};
\draw[1ar] (D1) -- (D2) node[midway,below] {$\k_2$};
\end{tikzpicture}
\]
where the strict morphisms of top and bottom groupoid crossed modules $\chi_i, \k_i, i=1,2$, are defined by
\[
\begin{array}{c}
\chi_1([m,p_1,p_2,n])\colonequals \a_2(m), \chi_1(h_1,p, h_3)\colonequals h_1, \chi_2([m,p_1,p_2,n])\colonequals \b_2(m), \chi_2(h_2,p,h_4)\colonequals h_2, \\
\k_1([m,p_1,p_2,n])\colonequals \d_2(n), \k_2(h_1,p,h_3)\colonequals h_3, \k_2([m,p_1,p_2,n])\colonequals \t_2(n), \ {\rm and\ } \ \k_2(h_2,p,h_4)\colonequals h_4, 
\end{array}
\]
which, using the properties of $\a_2,\b_2,\d_2$, and $\t_2$,  guaranteed by axioms~\eqref{CR1}~--~\ref{CR4}, can be seen to be hypercovers; this gives the desired decomposition~\ref{eq:decomposition_exchanger}\\

 ($\Longleftarrow$). Conversely, suppose we have a a chain of equivalences of crossed extensions ${\bf A}\stackrel{\Phi}{\lTo} {\bf P} \stackrel{\Psi}{\rTo}{\bf B}$ as in~\eqref{eq:decomposition_exchanger}, where ${\bf P}=(\frak G_5,\cal R,\frak G_6)$. Then, by Proposition~\ref{pro:equivalence_vs_exchanger}, we have two exchangers 
 \[
 {\bf A} \stackrel{P_{\Phi}}{\xc} {\bf P}\stackrel{P_{\Psi}}{\xc}{\bf B}
 \]
Therefore, we get the exchanger ${\bf A}\stackrel{P_\Phi\underset{\cal R}{\bullet}P_\Psi}{\xc}{\bf B}$, which achieves the proof. 

\end{proof}

\subsection{Morphisms of exchangers}

\begin{df}
Let ${\bf A}$ and ${\bf B}$ be crossed extensions with unit spaces $X$ and $Y$, respectively, and ${\bf A}\stackrel{P,Q}{\sxc} {\bf B}$ two semi-exchangers. A {\em morphism} from $P$ to $Q$ is just a morphism of generalized morphisms from $P$ to $Q$; that is, a continuous map $\eta:P\rTo Q$ such that 
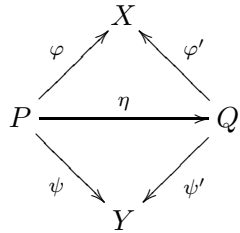
\[
\xymatrix{
& X & \\
P \ar[ru]^{\vp} \ar[rd]_{\psi} \ar[rr]^{\eta} && Q \ar[lu]_{\vp^\prime} \ar[ld]^{\psi^\prime} \\
& Y&
} 
\]
is commutative, and such that $\eta$ commutes with the groupoid actions of $\cal M$ and $\cal N$ on $P$ and $Q$. Semi-exchanger morphisms will be symbolized as 
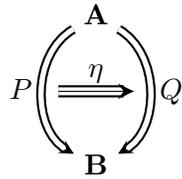
\begin{center}
\begin{tikzpicture}
\node (A) at (0,1) {${\bf A}$};
\node (B) at (0,-1) {${\bf B}$};
\node (P) at (-1,0) {$P$};
\node (Q) at (1,0) {$Q$};
\node (e) at (0,0.3) {$\eta$};

\draw[2ar] (A) to[bend right=60] (B);
\draw[2ar] (A) to[bend left=60] (B);
\draw[3ar] (-0.5,0) -- (0.5,0);
\draw (-0.5,0) -- (0.5,0);
\end{tikzpicture}
\end{center}

or just $P\stackrel{\eta}{\exch} Q$ when the crossed extensions ${\bf A}$ and ${\bf B}$ are understood. A {\em semi-exchanger isomorphism} is a semi-exchanger morphism that is a homeomorphism such that the inverse is also a semi-exchanger morphism.
\end{df}

It is easily seen that given two semi-exchanger morphisms 
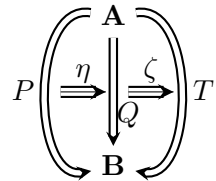
\begin{center}
\begin{tikzpicture}
\node (A) at (0,1) {${\bf A}$};
\node (B) at (0,-1) {${\bf B}$};
\node (P) at (-1.2,0) {$P$};
\node (Q) at (0.2,-0.3) {$Q$};
\node (T) at (1.2,0) {$T$};
\node (e) at (-0.4,0.3) {$\eta$};
\node (z) at (0.5,0.3) {$\zeta$};

\draw[2ar] (A) to[bend right=100] (B);
\draw[2ar] (A) -- (B);
\draw[2ar] (A) to[bend left=100] (B);
\draw[3ar] (-0.7,0) -- (-0.1,0);
\draw (-0.7,0) -- (-0.1,0);
\draw[3ar] (0.2,0) -- (0.8,0);
\draw (0.2,0) -- (0.8,0);
\end{tikzpicture}
\end{center}
the composite $\zeta\circ \eta\colon P\rTo T$ of the continuous maps $\eta$ and $\zeta$ is a semi-exchanger morphism from $P$ to $T$. We call the resulting morphism the {\em horizontal composition} of $\eta$ with $\zeta$, and we denote it by $\eta\star_h\zeta$. 

On the other hand, given two morphisms of semi-exchangers 
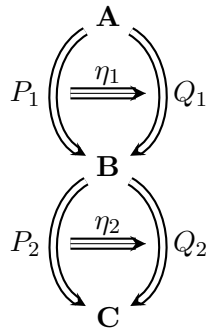
\[
\begin{tikzpicture}
\node (A) at (0,1) {${\bf A}$};
\node (B) at (0,-1) {${\bf B}$};
\node (C) at (0,-3) {${\bf C}$};
\node (P1) at (-1.1,0) {$P_1$};
\node (Q1) at (1.1,0) {$Q_1$};
\node (P2) at (-1.1,-2) {${P_2}$};
\node (Q2) at (1.1,-2) {$Q_2$};
\node (e1) at (0,0.3) {$\eta_1$};
\node (e1) at (0,-1.7) {$\eta_2$};

\draw[2ar] (A) to[bend right=60] (B);
\draw[2ar] (A) to[bend left=60] (B);
\draw[3ar] (-0.5,0) -- (0.5,0);
\draw (-0.5,0) -- (0.5,0);

\draw[2ar] (B) to[bend right =60] (C);
\draw[2ar] (B) to[bend left=60] (C);
\draw[3ar] (-0.5,-2) -- (0.5, -2);
\draw (-0.5,-2) -- (0.5,-2);
\end{tikzpicture}
\]

the continuous map 
\[
\eta_1\star_v\eta_2\colon P_1\underset{\cal N}{\bullet}P_2 \ni [p_1,p_2]\mto [\eta_1(p_1),\eta_2(p_2)]\in Q_1\underset{\cal N}{\bullet}Q_2
\]
defines the morphism
\[
\begin{tikzpicture}
\node (A) at (0,1) {${\bf A}$};
\node (C) at (0,-1) {${\bf C}$};
\node (P) at (-1.8,0) {$P_1\underset{\cal N}{\bullet}P_2$};
\node (Q) at (1.8,0) {$Q_1\underset{\cal N}{\bullet}Q_2$};
\node (e) at (0,0.3) {$\eta_1\star_v\eta_2$};

\draw[2ar] (A) to[bend right=90] (C);
\draw[2ar] (A) to[bend left=90] (C);
\draw[3ar] (-0.7,0) -- (0.7,0);
\draw (-0.7,0) -- (0.7,0);
\end{tikzpicture}
\]

Now, let ${\bf A}$ and ${\bf B}$ be crossed extensions as usual and ${\bf A} \stackrel{P}{\xc} {\bf B}$ an exchanger. We denote by $\bar{P}$ the topological space $P$ equipped with the left ${\cal N}$--action $n\cdot \bar{p}\colonequals pn^{-1}$ and right $\cal M$--action $\bar{p}\cdot m\colonequals m^{-1}p$ when the left hand sides make sense, where $\bar{p}$ is $p$ viewed as an element in $\bar{P}$. Then,
\[
Y\stackrel{\psi}{\lTo} \bar{P} \stackrel{\vp}{\rTo} X
\]
together with these groupoid actions, is a generalized morphism from $\cal N$ to $\cal M$. Moreover,

\begin{lem}\label{lem:exchanger_product}
Let ${\bf A} \stackrel{P}{\xc} {\bf B}$ be an exchanger. Then ${\bf B} \stackrel{\bar{P}}{\RTo} {\bf A}$ is an exchanger. Furthermore, there are exchanger isomorphisms $P\underset{\cal N}{\bullet}\bar{P} \exch {\bf I}_{\cal M}$ and $\bar{P}\underset{\cal M}{\bullet}P\exch {\bf I}_{\cal N}$.
\end{lem}

\begin{proof}
The first statement is immediate. Now, note that since $P\rTo X$ is a principal $\cal N$--bundle and $P\rTo Y$ is a principal $\cal M$--bundle, the maps $P\times_{X}\cal M\ni (p,m)\mto (p,\bar{mp})\in P\times_Y\bar{P}$ and $\cal N\times_YP \ni (n,p)\mto (\bar{pn}, p)\in \bar{P}\times_XP$ are homeomorphisms which descend to the quotients and give two homeomorphisms $P\underset{\cal N}{\bullet}\bar{P}\overset{\cong}{\rTo} \cal M$ and $\bar{P}\underset{\cal M}{\bullet}P\overset{\cong}{\rTo}\cal N$. 
\end{proof}

\begin{thm}\label{thm:2-cat}
The collection $\fXExt$ of all crossed extensions of groupoid crossed modules is a weak $2$--category whose $1$--arrows are semi-exchangers and $2$--arrows are semi-exchanger morphisms.
\end{thm}

\begin{proof}
Define the source and target maps for $1$--arrows in $\fXExt$ in the obvious way: 
\[
s({\bf A}\stackrel{P}{\sxc}{\bf B}) \colonequals {\bf A}, \ {\rm and\ } \ t({\bf A}\stackrel{P}{\sxc}{\bf B}) \colonequals {\bf B}, 
\]
and similarly one defines the source and target maps for the $2$--arrows.
Furthermore, by definition of the horizontal composition of exchangers, associated to any triple of semi-exchangers ${\bf A} \stackrel{P}{\sxc} {\bf B} \stackrel{Q}{\sxc}{\bf C} \stackrel{T}{\RTo} {\bf D}$, where ${\bf B}=(\frak G_3,\cal N,\frak G_4), {\bf C}=(\frak G_5,\cal R,\frak G_6)$, there is an {\em associator} 

\[
\begin{tikzpicture}
\node (A) at (0,1) {${\bf A}$};
\node (D) at (0,-1) {${\bf D}$};
\node (P) at (-2.2,0) {$P\underset{\cal N}{\bullet}(Q\underset{\cal R}{\bullet}T)$};
\node (Q) at (2.2,0) {$(P\underset{\cal N}{\bullet}Q)\underset{\cal R}{\bullet}T$};
\node (e) at (0,0.3) {$a_{P,Q,T}$};

\draw[2ar] (A) to[bend right=90] (D);
\draw[2ar] (A) to[bend left=90] (D);
\draw[3ar] (-0.7,0) -- (0.7,0);
\draw (-0.7,0) -- (0.7,0);
\end{tikzpicture}
\]
given by the map $[p,[q,t]]\mto [[p,q],t]$, which is clearly a semi-exchanger isomorphisms. 
Moreover, for all semi-exchanger ${\bf A} \overset{P}{\sxc} {\bf B}$, we get two semi-exchanger isomorphisms 
\[
\begin{tikzpicture}
\node (A) at (0,1) {${\bf A}$};
\node (B) at (0,-1) {${\bf B}$};
\node (P) at (-1.5,0) {${\bf I}_{\cal M}\underset{\cal M}{\bullet}P$};
\node (Q) at (1,0) {$P$};
\node (r) at (0,0.3) {$r_P$};

\draw[2ar] (A) to[bend right=60] (B);
\draw[2ar] (A) to[bend left=60] (B);
\draw[3ar] (-0.5,0) -- (0.5,0);
\draw (-0.5,0) -- (0.5,0);

\node (and) at (2,0) {and};

\node (A1) at (5,1) {${\bf A}$};
\node (B1) at (5,-1) {${\bf B}$};
\node (P1) at (3.5,0) {$P\underset{\cal N}{\bullet}{\bf I}_{\cal N}$};
\node (Q1) at (6,0) {$P$};
\node (l) at (5,0.3) {$l_P$};

\draw[2ar] (A1) to[bend right=60] (B1);
\draw[2ar] (A1) to[bend left=60] (B1);
\draw[3ar] (4.5,0) -- (5.5,0);
\draw (4.5,0) -- (5.5,0);
\end{tikzpicture}
\]
through the homeomorphisms $$r_P\colon {\bf I}_{\cal M}\underset{\cal M}{\bullet}P\ni [m,p]\mto mp\in P \quad {\rm and} \quad l_P\colon P\underset{\cal N}{\bullet}{\bf I}_{\cal N}\ni [p,n]\mto pn\in P.$$
Furthermore, we leave to the reader to check that horizontal and vertical compositions of semi-exchanger morphisms satisfy the following coherence law 
 
\[
\begin{tikzpicture}
\node (A1) at (0,2.5) {${\bf A}$};
\node (B1) at (0,0) {${\bf B}$};
\node (P1) at (-1.2,1) {$P_1$};
\node (Q1) at (0.29,0.7) {$Q_1$};
\node (T1) at (1.2,1) {$T_1$};
\node (e1) at (-0.4,1.5) {$\eta_1$};
\node (z1) at (0.5,1.5) {$\zeta_1$};

\node (C1) at (0,-2.5) {${\bf C}$};
\node (P2) at (-1.2,-1) {$P_2$};
\node (Q2) at (0.29, -1.7) {$Q_2$};
\node (T2) at (1.2,-1) {$T_2$};
\node (e2) at (-0.4,-1) {$\eta_2$};
\node (z2) at (0.5,-1) {$\zeta_2$};

\draw[2ar] (-0.3,2.3) to[bend right=100] (-0.3,0.2);
\draw[2ar] (A1) -- (B1);
\draw[2ar] (0.3,2.3) to[bend left=100] (0.3,0.2);
\draw[3ar] (-0.8,1.2) -- (-0.1,1.2);
\draw (-0.8,1.2) -- (-0.1,1.2);
\draw[3ar] (0.2,1.2) -- (0.8,1.2);
\draw (0.2,1.2) -- (0.8,1.2);

\draw[2ar] (-0.3,-0.2) to[bend right=100] (-0.3,-2.3);
\draw[2ar] (B1) -- (C1);
\draw[2ar] (0.3,-0.2) to[bend left=100] (0.3,-2.3);
\draw[3ar] (-0.8,-1.3) -- (-0.1,-1.3);
\draw (-0.8,-1.3) -- (-0.1,-1.3);
\draw[3ar] (0.2,-1.3) -- (0.8,-1.3);
\draw (0.2,-1.3) -- (0.8,-1.3);

\node (mto) at (1.8,0) {$\mto$};

\node (A2) at (5,1.5) {${\bf A}$};
\node (C2) at (5,-1.5) {${\bf C}$};
\node (P) at (3,0) {$P_1\underset{\cal N}{\bullet}P_2$};
\node (Q) at (7.2,0) {$T_1\underset{\cal N}{\bullet}T_2$};
\node (e3) at (5,0.3) {$\eta$};

\draw[2ar] (A2) to[bend right=100] (C2);
\draw[2ar] (A2) to[bend left=100] (C2);
\draw[3ar] (4,0) -- (6,0);
\draw (4,0) -- (6,0);
\end{tikzpicture}
\]
where $\eta = (\eta_1\star_h\zeta_1)\star_v(\eta_2\star_h\zeta_2) = (\eta_1\star_v\eta_2)\star_h(\zeta_1\star_v\zeta_2)$.
\end{proof}

\begin{cor}
Let $\frak G_1$ and $\frak G_2$ be groupoid crossed modules. The collection $\XExt(\frak G_1,\frak G_2)$ forms a weak $2$--groupoid whose objects are crossed extensions $\frak G_1\underset{\cal M}{\xext}\frak G_2$, $1$--arrows are exchangers, and $2$--arrows are exchanger isomorphisms. 
\end{cor}

\begin{proof}
For all fixed $\frak G_1$ and $\frak G_2$, it is a consequence of the decomposition theorem of crosssings (Theorem~\ref{thm:crossing_decomposition}) that $\XExt(\frak G_1,\frak G_2)$ is indeed a set. Now, for two crossed extensions $\frak G_1\underset{\cal M}{\xext}\frak G_2$ and $\frak G_1\underset{\cal N}{\xext}\frak G_2$, we define 
\[
\XExt(\frak G_1,\frak G_2)(\cal M,\cal N)=\fXExt((\frak G_1,\cal M,\frak G_2), (\frak G_1,\cal N,\frak G_2))
\] 
which, thanks to the preceding theorem and Lemma~\ref{lem:exchanger_product}, is a groupoid; for, every exchanger ${\bf A}\stackrel{P}{\xc}{\bf B}$ has an inverse $\bar{P}$ up to an exchanger isomorphism. Moreover, the composition "functor" 
\[
\XExt(\frak G_1,\frak G_2)(\cal M,\cal N)\times \XExt(\frak G_1,\frak G_2)(\cal N,\cal R)\rTo \XExt(\frak G_1,\frak G_2)(\cal M,\cal R)
\]
is given by vertical composition of exchangers and the two compositions of exchanger morphisms defined previously.
\end{proof}

\subsection{Principal $(\cal A\rTo \Aut(\cal A))$--bundles over $2$--groupoids}\label{subs:principal}

In this section we generalize the notion of principal $(G\rTo \Aut(G))$--bundles over Lie groupoids defined in~\cite{Ginot-Stienon:G-gerbes} to topological $2$--groupoids and crossed modules of topological groupoids and show their connection with crossings.  

\begin{df}
Let $\cal A\rTo X$ be a bundle of topological groups and let ${\bf G}$ be a topological $2$--groupoids. A {\em principal} $(\cal A\rTo \Aut(\cal A))$--{\em bundle} over ${\bf G}$ is a weak homomorphism from $\bf G$ to the topological $2$--groupoid $\cal A\rtimes\Aut(\cal A)\rrTo \Aut(\cal A)\rrTo X$ associated to the groupoid crossed module $\cal A\xTo \Aut(\cal A)$.
\end{df}

In view of Theorem~\ref{thm:crossing_decomposition}, a principal $(\cal A\rTo \Aut(\cal A))$--bundle over $\bf G$ is equivalent to a crossing 
\[
{\frak G}_{\bf G}\underset{\cal M}{\cross}(\cal A\stackrel{Ad}{\xTo}\Aut(\cal A)).
\]

This suggests the following equivalent definition in terms of groupoid crossed modules and crossings. 

\begin{df}
Let $\frak G$ be a groupoid crossed module. A {\em principal} $(\cal A\rTo \Aut(\cal A))$--{\em bundle} over ${\frak G}$ is a crossing from $\frak G$ to the groupoid crossed module $\cal A\stackrel{Ad}{\xTo}\Aut(\cal A)$.
\end{df} 

The result below generalizes~\cite[Theorem 3.4]{Ginot-Stienon:G-gerbes}. 

\begin{thm}\label{thm:crossings_vs_extensions}
Let $\cal G\rrTo X$ be a topological groupoid and $\cal A\rTo X$ a groupoid $\cal G$--module. There is a bijection between the set ${\bf ext}(\cal G,\cal A)$ of Morita equivalence classes of groupoid $\cal A$--extensions and the set of crossings $(\cal G^0\xTo \cal G)\cross (\cal A\xTo \Aut(\cal A))$ up to exchangers. 
\end{thm}

\begin{proof}
We outline the correspondence between groupoid extensions of $\cal G$ and crosssings. That this correspondence respects Morita equivalences on one hand and exchangers on the other will be clear once we will have defined horizontal compositions of exchangers in the next section. Let $\cal A\stackrel{p}{\rTo} \cal G^0$ be a $\cal G$--module and $\cal A\stackrel{\iota}{\rTo}\tilde{\cal G}\stackrel{\pi}{\rTo}\cal G$ be a groupoid $\cal A$--extension. Then by identifying $\cal A$ with its image in $\tilde{\cal G}$ via $\iota$, Example~\ref{ex:extension_crossed} shows that $\cal A\stackrel{\iota}{\rTo}\tilde{\cal G}$ is a groupoid crossed module with respect to the action $Ad_{\tilde{g}}(a)\colonequals \tilde{g}^{-1}a\tilde{g}$, for $a\in \cal A_{t(\tilde{g})}$. Consider the two strict morphisms of crossed modules 
\[
\xymatrix{
\cal G^0\ar[r]^{\iota} & \cal G \\
\cal A \ar[u]^{p} \ar[r]^{\iota} \ar@{=}[d] & \tilde{\cal G} \ar[u]_{\pi} \ar[d]^{Ad} \\ 
\cal A\ar[r]^{Ad} & \Aut(\cal A)
}
\]
with the upper strict morphism being naturally a hypercover. Therefore, by Proposition~\ref{pro2:crossing} and Theorem~\ref{thm:Morita_xext}, we have the crossed extension and crossing below 
\[
(\cal G^0\xTo \cal G)\underset{\cal G^0\rtimes \tilde{\cal G}=\tilde{\cal G}}{\xext} (\cal A\xTo \tilde{\cal G})\underset{\cal A\rtimes \tilde{\cal G}}{\cross}(\cal A\xTo \Aut(\cal A)), 
\] 
hence the crossing $\tilde{\cal G}\ou{\Dprod}{\cal A}{\tilde{\cal G}}(\cal A\rtimes \tilde{\cal G})$ from $(\cal G^0\xTo \cal G)$ to $(\cal A\xTo \Aut(\cal G))$.

Conversely, as it was shown in Example~\ref{ex1:groupoid_crossing}, any crossing $(\cal G^0\xTo \cal G)\underset{\cal M}{\cross} (\cal A\xTo \Aut(\cal A))$ is a groupoid $\cal A$--extension of $\cal G$, thanks to~\ref{CR3}. 
\end{proof}

\section{The $3$--category $\XMod$}\label{sec:3-cat}

We are going to show the collection of all crossed modules of topological groupoids has the structure of a weak $3$--category. Precisely, we think of a crossing  $\frak G_1\underset{\cal M}{\cross}\frak G_2$ as an arrow from $\frak G_1$ to $\frak G_2$, and given another crossing $\frak G_1\underset{\cal N}{\xext}\frak G_2$, a semi-exchanger $P$ from $(\frak G_1,\cal M,\frak G_2)$ to $(\frak G_1,\cal N,\frak G_2)$ will be regarded as an arrow $\cal M\stackrel{P}{\sxc} \cal N$, and will be represented by a bigon 
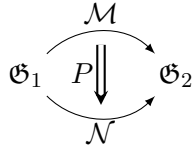
\begin{center}
\begin{tikzpicture}
\node (G1) at (-1,0) {${\frak G_1}$};
\node (G2) at (1,0) {${\frak G_2}$};
\node (M) at (0,0.8) {${\cal M}$};
\node (N) at (0,-0.8) {${\cal N}$};
\node (P) at (-0.25,0) {$P$};

\draw[2ar] (0,0.4) -- (0,-0.4);
\draw[->,>=latex] (G1) to[bend left=45] (G2);
\draw[->,>=latex] (G1) to[bend right=45] (G2);
\end{tikzpicture}
\end{center}

As for (semi-)exchanger morphisms, we will represent them in a way to include the crossed modules, the crossings, and the (semi-)exchangers, altogether, as the following figure 
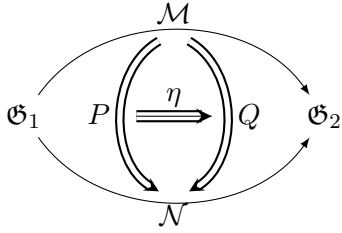
\begin{center}
\begin{tikzpicture}
\node (G1) at (-2,0) {${\frak G_1}$};
\node (G2) at (2,0) {${\frak G_2}$};
\node (M) at (0,1.35) {${\cal M}$};
\node (N) at (0,-1.35) {${\cal N}$};
\node (P) at (-1,0) {$P$};
\node (Q) at (1,0) {$Q$};
\node (e) at (0,0.3) {$\eta$};

\draw[->, >=latex] (G1) to[bend left=55] (G2);
\draw[->,>=latex] (G1) to[bend right=55] (G2);
\draw[2ar] (-0.2,1) to[bend right=60] (-0.2,-1);
\draw[2ar] (0.2,1) to[bend left=60] (0.2,-1);
\draw[3ar] (-0.5,0) -- (0.5,0);
\draw (-0.5,0) -- (0.5,0);
\end{tikzpicture}
\end{center} 

\noindent  We now want to show the following.

\begin{thm}\label{thm:3-cat}
There is a weak \ $3$--category \ $\XMod$ whose objects are crossed modules of topological groupoids, and in which an $1$--arrow from an object $\frak G_1$ to an object $\frak G_1$ is a crossing $\frak G_1\cross\frak G_2$, $2$--arrows are semi-exchangers, and $3$--arrows are semi-exchanger morphisms. Furthermore, weak identity $1$--arrows are given by the trivial crossed extensions $\cal O_{\frak G}$. 
\end{thm}

\noindent In order to prove this statement, we need a few more constructions.

\subsection{Horizontal composition of exchangers}

We have already seen how to concatenate "vertically" exchangers. We are now going to define a new composition of exchangers that we will call {\em horizontal composition}. Specifically, Suppose we have four crossed extensions (or just crossings)
\[
\frak G_1\underset{\cal M_1}{\xext}\frak G_2\underset{\cal M_2}{\xext}\frak G_3, \ {\rm and\ } \ \frak G_4\underset{\cal N_1}{\xext}\frak G_5\underset{\cal N_2}{\xext}\frak G_6
\]  
together with two (semi-)exchengers ${\bf A}_1\stackrel{P_1}{\sxc}{\bf B}_1$ and ${\bf A}_2\stackrel{P_2}{\sxc}{\bf B}_2$, where $${\bf A}_1=(\frak G_1,\cal M_1,\frak G_2), {\bf B}_1=(\frak G_2,\cal M_2,\frak G_3), {\bf A}_2=(\frak G_4,\cal N_1,\frak G_5), {\rm and\ } \ {\bf B}_2=(\frak G_5,\cal N_2,\frak G_5).$$ 

Then, if we denote by ${\bf A}_1\Dprod_{\frak G_2}{\bf B}_1$ and ${\bf A}_2\Dprod_{\frak G_5}{\bf B}_2$ the crossed extensions $$(\frak G_1,\cal M_1\Dprod_{\frak G_2}\cal M_2,\frak G_3), \ {\rm and\ } \ (\frak G_4,\cal N_1\Dprod_{\frak G_5}\cal N_2,\frak G_6),$$ we aim at constructing a new (semi-)exchanger 
\[
{\bf A}_1\Dprod_{\frak G_2}{\bf B}_1\stackrel{P_1\Dprod P_2}{\sxc} {\bf A}_2\Dprod_{\frak G_5}{\bf B}_2
\]
which we will call the {\em horizontal composition} of the (semi-)exchangers $P_1$ and $P_2$. Moreover, we wish this construction to be compatible with vertical composition in a higher categorical sense. \\

For this end, we may assume that ${\bf A}_1, {\bf B}_1$ have common unit space $X$, and ${\bf A}_2, {\bf B}_2$ have common unit space $Y$; otherwise, as usual, it suffices to work on pullbacks. We then have the (semi-)exchangers and commutative diagrams as below 
\[
\xymatrix{
\cal H_1\ar[rr]^{\p_1}\ar[rd]_{\a_1} && \cal G_1 && \cal H_4\ar[rr]^{\p_4} \ar[rd]_{\lambda_1} && \cal G_4 \\
& \cal M_1\ar[ru]_{\a_2} \ar[rd]^{\b_2} & & \stackrel{P_1}{\sxc} & & \cal N_1\ar[ru]_{\lambda_2}\ar[rd]^{\mu_2} & \\ 
\cal H_2 \ar[rr]^{\p_2}\ar[ru]^{\b_1} \ar[rd]_{\d_1} && \cal G_2 && \cal H_5 \ar[ru]^{\mu_1} \ar[rr]^{\p_5} \ar[rd]_{\nu_1} && \cal G_5 \\
& \cal M_2\ar[ru]_{\d_2} \ar[rd]^{\t_2} & & \stackrel{P_2}{\sxc} & & \cal N_2\ar[ru]_{\nu_2} \ar[rd]^{\rho_2} & \\
\cal H_3 \ar[ru]^{\t_1} \ar[rr]^{\p_3} && \cal G_3 && \cal H_6\ar[ru]^{\rho_1} \ar[rr]^{\p_6} && \cal G_6
}
\]

Suppose that the (semi-)exchangers $P_1$ and $P_2$ are defined through the continuous maps 
\[
X\stackrel{\vp_1}{\lTo}P_1\stackrel{\psi_1}{\rTo}Y, \ {\rm and\ } \ X\stackrel{\vp_2}{\lTo}P_2\stackrel{\psi_2}{\rTo}Y.
\]

Then the pair groupoid $\cal H_1\times \cal H_5\rrTo X\times Y$ acts continuously on the space 
\[
(P_1\times_XP_2)\cap (P_1\times_YP_2)
\]
(equipped with the induced topology of $P_1\times P_2$), through the momentum map $$(p_1,p_2)\mto (\vp_2(p_2),\psi_1(p_1))$$ and the formula 
\begin{eqnarray}\label{eq:product_exchangers}
(h_2,h_5)\cdot (p_1,p_2)\colonequals (\b_1(h_2)p_1\mu_1(h_5)^{-1},\d_1(h_2)p_2\nu_1(h_5)^{-1}), 
\end{eqnarray}
for $(h_2,h_5)\in \cal H_2^{\vp_2(p_2)}\times \cal H_5^{\psi_1(p_1)}$. 

\begin{df}
Let ${\bf A}_1\stackrel{P_1}{\sxc}{\bf B}_1$ and ${\bf A}_1\stackrel{P_2}{\sxc}{\bf B}_2$ be as above. We define the following space 
\[
P_1\Dprod P_2 \colonequals \frac{(P_1\times_XP_2)\cap (P_1\times_YP_2)}{\cal H_2\times \cal H_5}
\]
where the quotient is taken out of the $\cal H_2\times\cal H_5$--action~\eqref{eq:product_exchangers}. This space (when non-empty) will be called the {\em diamond product} of $P_1$ and $P_2$ over the groupoid crossed modules $\frak G_2$ and ${\frak G}_5$.
\end{df}

 As usual, equivalence classes in this quotient will be symbolized with brackets $[\cdot,\cdot]$. Next, equip $P_1\Dprod P_2$ with the quotient topology and consider the canonical projections 
\begin{eqnarray}\label{eq2:product_exchangers}
\xymatrix{
X && P_1\Dprod P_2 \ar[ll]_{\vp_2\circ pr_2}   \ar[rr]^{\psi_1\circ pr_1} && Y.
}
\end{eqnarray}

Now, notice we have left and right groupoid actions of $\cal M_1\Dprod_{\frak G_2}\cal M_2\rrTo X$ and $\cal N_1\Dprod_{\frak G_5}\cal N_2\rrTo Y$, respectively given by 
\begin{eqnarray}\label{eq3:product_exchangers}
[m_1,m_2]\cdot [p_1,p_2]\colonequals [m_1p_1,m_2p_2], \ {\rm and\ \ } [p_1,p_2]\cdot [n_1,n_2]\colonequals [p_1n_1,p_2n_2],
\end{eqnarray}
for $[m_1,m_2]\in \cal M_1\Dprod_{\frak G_2}\cal M_2$ such that $s(m_1)=s(m_2)=\vp_1(p_1)=\vp_2(p_2)$,  $[n_1,n_2]\in \cal N_1\Dprod_{\frak G_5}\cal N_2$ such that $t(n_1)=t(n_2)=\psi_1(p_1)=\psi_2(p_2)$. Furthermore, we have the following

\begin{pro}
Let ${\bf A}_1\stackrel{P_1}{\sxc}{\bf B}_1$ and ${\bf A}_1\stackrel{P_2}{\sxc}{\bf B}_2$ be semi-exchangers (resp. exchangers) as above. Then the diamond product space $P_1\Dprod P_2$, together with the maps~\eqref{eq2:product_exchangers} and the actions~\eqref{eq3:product_exchangers}, is a generalized morphism (resp. Morita equivalence) from $\cal M_1\Dprod_{\frak G_2}\cal M_2\rrTo X$ to $\cal N_1\Dprod_{\frak G_5}\cal N_2\rrTo Y$.
\end{pro}

In fact, this construction gives us a (semi-)exchanger. Precisely,

\begin{pro}
Let ${\bf A}_1\stackrel{P_1}{\sxc}{\bf B}_1$ and ${\bf A}_2\stackrel{P_2}{\sxc}{\bf B}_2$ be semi-exchangers (resp. exchangers) as previously. Then $P_1\Dprod P_2$ is a semi-exchanger 
\[
(\frak G_1,\cal M_1\Dprod_{\frak G_2}\cal M_2,\frak G_3)\sxc (\frak G_4,\cal N_1\Dprod_{\frak G_5}\cal N_2,\frak G_6)
\]
which is an exchanger if $P_2$ and $P_2$ are. This (semi-)exchanger will be labeled as the {\em horizontal composition} of ${\bf A}_i\stackrel{P_i}{\sxc}{\bf B}_i, i=1,2$.
\end{pro}

\begin{proof}
By construction of the morphisms involved in the diamond product of crossings (see Lemma~\ref{lem:diamond_crossing}), we see that the actions of $\cal H_1$, $\cal H_3$, $\cal H_4$ and $\cal H_6$ on $P_1\Dprod P_2$, are respectively given by 

\begin{align*}
h_1\cdot [p_1,p_2]=[\a_1(h_1)p_1, p_2], \ h_3\cdot[p_1,p_2]=[p_1,\t_1(h_3)p_2] \\ 
[p_1,p_2]\cdot h_4=[p_1\lambda_1(h_4),p_2], \ [p_1,p_2]\cdot h_6=[p_1,p_2\rho_1(h_6)] 
\end{align*}
where the actions involved are well defined. Therefore, that these actions  on $P_1\Dprod P_2$ satisfy~\ref{E1} and~\ref{E2} follows from $P_1$ and $P_2$ satisfying these axioms. 
\end{proof}

We shall note that the use of the label "horizontal" comes from our viewing the exchanger $P_1\Dprod P_2$ as the composition of two arrows horizontally represented as below
\[
\xymatrix{
\frak G_1 & \underset{\cal M_1}{\xext} \ar@{=>}[d]^{P_1} & \frak G_2 & \underset{\cal M_2}{\xext} \ar@{=>}[d]^{P_2} & \frak G_3 \\ 
\frak G_4 & \underset{\cal N_1}{\xext} & \frak G_5 & \underset{\cal N_2}{\xext} & \frak G_6 
}
\]

The following proposition shows how horizontal composition is compatible (in the higher categorical sense) with the vertical composition discussed in the previous sections. 

\begin{pro}[Coherence law]\label{pro:coherence_exchangers}
Assume given the data of four (semi-)exchangers as below 
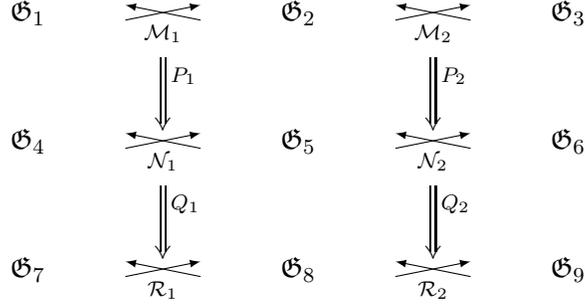
\[
\xymatrix{
\frak G_1 & \underset{\cal M_1}{\xext} \ar@{=>}[d]^{P_1} & \frak G_2 & \underset{\cal M_2}{\xext} \ar@{=>}[d]^{P_2} & \frak G_3 \\ 
\frak G_4 & \underset{\cal N_1}{\xext} \ar@{=>}[d]^{Q_1}& \frak G_5 & \underset{\cal N_2}{\xext}\ar@{=>}[d]^{Q_2} & \frak G_6 \\
\frak G_7 & \underset{\cal R_1}{\xext} & \frak G_8 & \underset{\cal R_2}{\xext} & \frak G_9
}
\] 
Then, there is an isomorphism of (semi-)exchangers 
\begin{eqnarray}
(P_1\Dprod P_2)\underset{\cal N_1\Dprod_{\frak G_5}\cal N_2}{\bullet} (Q_1\Dprod Q_2) \stackrel{\eta_{P_1,P_2,Q_1,Q_2}}{\exch} (P_1\underset{\cal N_1}{\bullet}Q_1)\Dprod (P_2\underset{\cal N_2}{\bullet}Q_2).
\end{eqnarray}
\end{pro}

\begin{proof}
The canonical map that permutes the components gives the desired isomorphism. 
\end{proof}


\subsection{Spatial composition of morphisms of exchangers}

We have already defined two compositions of morphisms of (semi-)exchangers: a horizontal composition and a vertical one. We are now defining a third one that we will call {\em spatial composition}. Specifically, suppose we have crossings, semi-exchangers, and morphisms of exchangers as below 

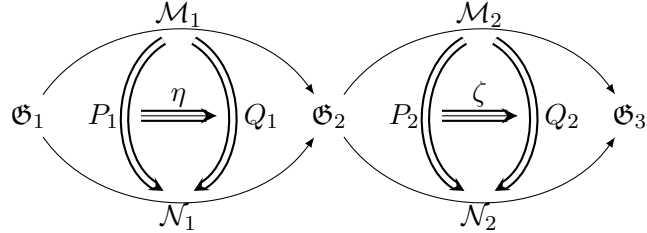
\begin{center}
\begin{tikzpicture}
\node (G1) at (-2,0) {${\frak G_1}$};
\node (G2) at (2,0) {${\frak G_2}$};
\node (G3) at (6,0) {${\frak G_3}$};
\node (M1) at (0,1.35) {${\cal M}_1$};
\node (N1) at (0,-1.35) {${\cal N}_1$};
\node (M2) at (4,1.35) {${\cal M}_2$};
\node (N2) at (4,-1.35) {${\cal N}_2$};
\node (P1) at (-1,0) {$P_1$};
\node (Q1) at (1.1,0) {$Q_1$};
\node (e) at (0,0.3) {$\eta$};
\node (P2) at (3,0) {$P_2$};
\node (Q2) at (5.1,0) {$Q_2$};
\node (z) at (4,0.3) {$\zeta$};

\draw[->, >=latex] (G1) to[bend left=55] (G2);
\draw[->,>=latex] (G1) to[bend right=55] (G2);
\draw[2ar] (-0.2,1) to[bend right=60] (-0.2,-1);
\draw[2ar] (0.2,1) to[bend left=60] (0.2,-1);
\draw[3ar] (-0.5,0) -- (0.5,0);
\draw (-0.5,0) -- (0.5,0);

\draw[->, >=latex] (G2) to[bend left=55] (G3);
\draw[->,>=latex] (G2) to[bend right=55] (G3);
\draw[2ar] (3.8,1) to[bend right=60] (3.8,-1);
\draw[2ar] (4.2,1) to[bend left=60] (4.2,-1);
\draw[3ar] (3.5,0) -- (4.5,0);
\draw (3.5,0) -- (4.5,0);
\end{tikzpicture}
\end{center} 

Then, we get the morphism of semi-exchangers 

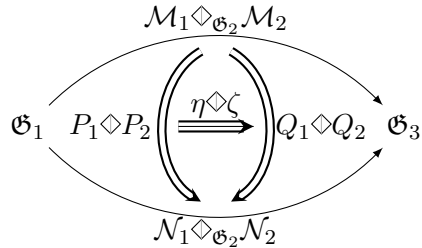
\begin{center}
\begin{tikzpicture}
\node (G1) at (-2.5,0) {${\frak G_1}$};
\node (G3) at (2.5,0) {${\frak G_3}$};
\node (M) at (0,1.44) {${\cal M}_1\Dprod_{\frak G_2}{\cal M}_2$};
\node (N) at (0,-1.4) {${\cal N}_1\Dprod_{\frak G_2}{\cal N}_2$};
\node (P) at (-1.4,0) {$P_1\Dprod P_2$};
\node (Q) at (1.4,0) {$Q_1\Dprod Q_2$};
\node (e) at (0,0.3) {$\eta\Dprod \zeta$};

\draw[->, >=latex] (G1) to[bend left=45] (G3);
\draw[->,>=latex] (G1) to[bend right=45] (G3);
\draw[2ar] (-0.2,1) to[bend right=60] (-0.2,-1);
\draw[2ar] (0.2,1) to[bend left=60] (0.2,-1);
\draw[3ar] (-0.5,0) -- (0.5,0);
\draw (-0.5,0) -- (0.5,0);
\end{tikzpicture}
\end{center} 
by setting $(\eta\Dprod \zeta)([p_1,p_2])\colonequals [\eta(p_1),\zeta(p_2)]$. Indeed, this is well defined, for the maps $\eta\colon P_1\rTo Q_1$ and $\zeta\colon Q_1\rTo Q_2$ commutes with the groupoid actions involved. It is now a matter of simple algebraic verifications to check the coherence laws below; so we omit the proofs.  

\begin{pro}\label{pro2:coherence_exchangers}
Suppose we have four morphisms of semi-exchangers as below 
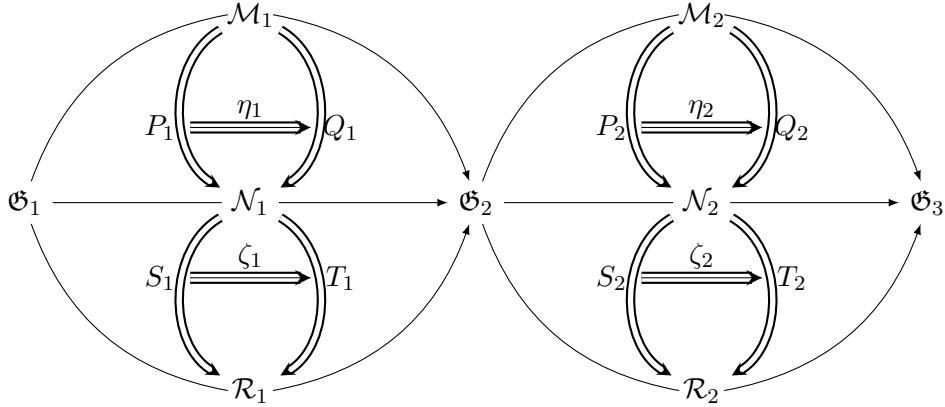
\begin{center}
\begin{tikzpicture}
\node (G1) at (-6,0) {${\frak G}_1$};
\node (G2) at (0,0) {${\frak G}_2$};
\node (G3) at (6,0) {${\frak G}_3$};
\node (M1) at (-3,2.5) {${\cal M}_1$};
\node (N1) at (-3,0) {${\cal N}_1$};
\node (R1) at (-3,-2.5) {${\cal R}_1$};
\node (M2) at (3,2.5) {${\cal M}_2$};
\node (N2) at (3,0) {${\cal N}_2$};
\node (R2) at (3,-2.5) {${\cal R}_2$};
\node (P1) at (-4.2,1) {$P_1$};
\node (Q1) at (-1.8,1) {$Q_1$};
\node (S1) at (-4.2,-1) {$S_1$};
\node (T1) at (-1.8,-1) {$T_1$};
\node (P2) at (1.8,1) {$P_2$};
\node (Q2) at (4.2,1) {$Q_2$};
\node (S2) at (1.8,-1) {$S_2$};
\node (T2) at (4.2,-1) {$T_2$};
\node (e1) at (-3,1.3) {$\eta_1$};
\node (z1) at (-3,-0.7) {$\zeta_1$};
\node (e2) at (3,1.3) {$\eta_2$};
\node (z2) at (3,-0.7) {$\zeta_2$}; 

\draw[-] (G1) to[bend left=30] (-3.3,2.5);
\draw[->,>=latex] (-2.7,2.5) to[bend left=30] (G2);
\draw[->,>=latex] (G1) -- (N1) -- (G2); 
\draw[-] (G1) to[bend right=30] (-3.3,-2.5);
\draw[->,>=latex] (-2.7,-2.5) to[bend right=30] (G2);
\draw[-] (G2) to[bend left=30] (2.7,2.5);
\draw[->,>=latex] (3.3,2.5) to[bend left=30] (G3);
\draw[->,>=latex] (G2) -- (N2) --(G3);
\draw[-] (G2) to[bend right=30] (2.7,-2.5);
\draw[->,>=latex] (3.3,-2.5) to[bend right=30] (G3);

\draw[2ar] (-3.4,2.3) to[bend right=60] (-3.4,0.2);
\draw[2ar] (-2.6,2.3) to[bend left=60] (-2.6,0.2);

\draw[2ar] (-3.4,-0.2) to[bend right=60] (-3.4,-2.3);
\draw[2ar] (-2.6,-0.2) to[bend left=60] (-2.6,-2.3);

\draw[2ar] (2.6,2.3) to[bend right=60] (2.6,0.2);
\draw[2ar] (3.4,2.3) to[bend left=60] (3.4,0.2);

\draw[2ar] (2.6, -0.2) to[bend right=60] (2.6,-2.3);
\draw[2ar] (3.4,-0.2) to[bend left=60] (3.4,-2.3);

\draw[3ar] (-3.8,1) -- (-2.2,1);
\draw (-3.8,1) -- (-2.2,1);
\draw[3ar] (-3.8,-1) --(-2.2,-1);
\draw (-3.8,-1) --(-2.2,-1);
\draw[3ar] (2.2,1) --(3.8,1);
\draw (2.2,1) -- (3.8,1);
\draw[3ar] (2.2,-1) -- (3.8,-1);
\draw (2.2,-1) -- (3.8,-1);
\end{tikzpicture}
\end{center}

Then the following diagram of morphisms of semi-exchangers commutes 
\[
\xymatrix{\relax
(P_1\Dprod P_2)\underset{\cal N_1\Dprod_{\frak G_2}\cal N_2}{\bullet}(S_1\Dprod S_2) \ar@3{->}[rrr]^{(\eta_1\Dprod \eta_2)\star_v(\zeta_1\Dprod \zeta_2)} \ar@3{<->}[d]_-{\eta_{P_1,P_2,S_1,S_2}} &&& (Q_1\Dprod Q_2)\underset{\cal N_1\Dprod_{\frak G_2}\cal N_2}{\bullet}(T_1\Dprod T_2) \ar@3{<->}[d]^-{\eta_{Q_1,Q_2,T_1,T_2}} \\ 
(P_1\underset{\cal N_1}{\bullet}S_1)\Dprod(P_2\underset{\cal N_2}{\bullet}S_2) \ar@3{->}[rrr]^{(\eta_1\star_v\zeta_1)\Dprod (\eta_2\star_v\zeta_2) } &&& (Q_1\underset{\cal N_1}{\bullet}T_1)\Dprod (Q_2\underset{\cal N_2}{\bullet}T_2)
}
\] 
\end{pro}

A second coherence relation upon compositions of morphisms of semi-exchangers is given by the following proposition. 

\begin{pro}\label{pro3:coherence_exchangers}
Suppose given four morphisms of semi-exchangers as below 
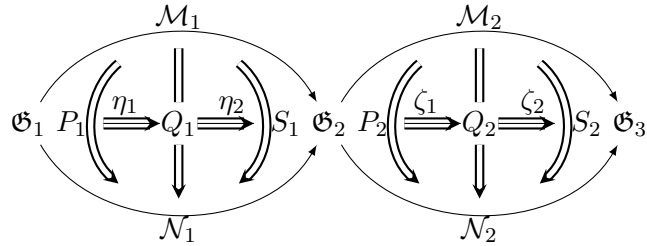
\begin{center}
\begin{tikzpicture}
\node (G1) at (-2,0) {${\frak G_1}$};
\node (G2) at (2,0) {${\frak G_2}$};
\node (G3) at (6,0) {${\frak G_3}$};
\node (M1) at (0,1.42) {${\cal M}_1$};
\node (N1) at (0,-1.42) {${\cal N}_1$};
\node (M2) at (4,1.42) {${\cal M}_2$};
\node (N2) at (4,-1.42) {${\cal N}_2$};
\node (P1) at (-1.42,0) {$P_1$};
\node (Q1) at (0,0) {$Q_1$};
\node (S1) at (1.42,0) {$S_1$};
\node (e1) at (-0.71,0.3) {$\eta_1$};
\node (e2) at (0.71,0.3) {$\eta_2$};
\node (P2) at (2.58,0) {$P_2$};
\node (Q2) at (4,0) {$Q_2$};
\node (S2) at (5.42,0) {$S_2$};
\node (z1) at (3.29,0.3) {$\zeta_1$};
\node (z2) at (4.71,0.3) {$\zeta_2$};

\draw[->, >=latex] (G1) to[bend left=60] (G2);
\draw[->,>=latex] (G1) to[bend right=60] (G2);
\draw[2ar] (0,1) -- (Q1) -- (0,-1);
\draw[2ar] (-0.8,0.8) to[bend right=50] (-0.8,-0.8);
\draw[2ar] (0.8,0.8) to[bend left=50] (0.8,-0.8);
\draw[3ar] (-1,0) -- (-0.25,0);
\draw (-1,0) -- (-0.25,0);
\draw[3ar] (0.25,0) -- (1,0);
\draw (0.25,0) -- (1,0);

\draw[->, >=latex] (G2) to[bend left=60] (G3);
\draw[->,>=latex] (G2) to[bend right=60] (G3);
\draw[2ar] (4,1) -- (Q2) -- (4,-1);
\draw[2ar] (3.2,0.8) to[bend right=50] (3.2,-0.8);
\draw[2ar] (4.8,0.8) to[bend left=50] (4.8,-0.8);
\draw[3ar] (3,0) -- (3.75,0);
\draw (3,0) -- (3.75,0);
\draw[3ar] (4.25,0) -- (5,0);
\draw (4.25,0) -- (5,0);
\end{tikzpicture}
\end{center}
Then 
\[
(\eta_1\Dprod \zeta_1)\star_h(\eta_2\Dprod\zeta_2)=(\eta_1\star_h\eta_2)\Dprod(\zeta_1\star_h\zeta_2).
\]
\end{pro}


\subsection{Proof of Theorem~\ref{thm:3-cat}}

Given two groupoid crossed modules $\frak G_1$ and $\frak G_2$, we define $\XMod(\frak G_1,\frak G_2)$ be the collection of all crossings $\frak G_1\underset{\cal M}{\cross}\frak G_2$; in other words, it is the sub-collection of $\fXExt$ consisting  of all objects of the form $(\frak G_1,\cal M, \frak G_2)$. It follows from Theorem~\ref{thm:2-cat} that $\XMod(\frak G_1,\frak G_2)$ is a weak $2$--category whose $1$--arrows are semi-exchangers $\cal M\sxc \cal N$, and $2$--arrows are morphisms of these. The {\em weak unit} in $\XMod(\frak G_1,\frak G_2)$ associated to an arrow $\cal M$ is the trivial exchanger ${\bf I}_{\cal M}$ (cf. Example~\ref{ex:unit_exchanger}). Furthermore, as a consequence of the coherence laws established in Propositions~\ref{pro:coherence_exchangers},~\ref{pro2:coherence_exchangers}, and~\ref{pro3:coherence_exchangers}, we have the lemma below. 

\begin{lem}
Let $\frak G_i, i=1,2,3$, be groupoid crossed modules. Then the assignment 
\[
\Dprod\colon \XMod(\frak G_1,\frak G_2)\times \XMod(\frak G_2,\frak G_3)\rTo \XMod(\frak G_1,\frak G_3)
\]
sending a pair of crossings $(\cal M,\cal N)$ to the crossing $\cal M\Dprod_{\frak G_2}\cal N$ together with the correspondence mappings on semi-exchangers and their morphisms, is a weak $2$--functor.
\end{lem}

Now, to end the proof, we show that the trivial crossed extensions are indeed weak identity $1$--morphisms. More precisely, we need to show that

\begin{pro}
Let $\frak G_1$ and $\frak G_2$ be groupoid crossed modules. For every crossing 
\[
\xymatrix{
\cal H_1 \ar[rr]^{\p_1} \ar[rd]_{\a_1} && \cal G_1 \\
& \cal M \ar[ru]_{\a_2}  \ar[rd]^{\b_2} & \\
\cal H_2 \ar[ru]^{\b_1} \ar[rr]^{\p_2} && \cal G_2
}
\]
there are semi-exchangers 
\[
\cal M\Dprod_{\frak G_2}\cal O_{\frak G_2} \stackrel{R_{\frak G_2}}{\sxc} \cal M\stackrel{\bar{R}_{\frak G_2}}{\sxc}\cal M\Dprod_{\frak G_2}\cal O_{\frak G_2}
\]
and morphisms 
\[
R_{\frak G_2}\underset{\cal M}{\bullet}\bar{R}_{\frak G_2}\exch {\bf I}_{\cal M\Dprod_{\frak G_2}\cal O_{\frak G_2}} \ {\rm and \ } \ \bar{R}_{\frak G_2}\underset{\cal M\Dprod_{\frak G_2}\cal O_{\frak G_2}}{\bullet}R_{\frak G_2}\exch {\bf I}_{\cal M}.
\]
Analogously, there are semi-exchangers
\[
\cal O_{\frak G_1}\Dprod_{\frak G_1}\cal M \stackrel{L_{\frak G_1}}{\sxc} \cal M\stackrel{\bar{L}_{\frak G_1}}{\sxc}\cal O_{\frak G_1}\Dprod_{\frak G_1}\cal M
\]
and morphisms  
\[
L_{\frak G_1}\underset{\cal M}{\bullet}\bar{L}_{\frak G_1}\exch {\bf I}_{\cal O_{\frak G_1}\Dprod_{\frak G_1}\cal M} \ {\rm and \ } \ \bar{L}_{\frak G_1}\underset{\cal O_{\frak G_1}\Dprod_{\frak G_1}\cal M}{\bullet}L_{\frak G_1}\exch {\bf I}_{\cal M}.
\]
\end{pro}

\begin{proof}
As usual, we assume, for the sake of simplicity, that both crossed modules $\frak G_1$ and $\frak G_2$ have the common unit space $X=\cal M^0$.  We get the semi-exchanger $R_{\frak G_2}$ by setting $R_{\frak G_2}\colonequals \cal M$ and considering the groupoid left action of $\cal M\Dprod_{\frak G_2}\cal O_{\frak G_2}\rrTo X$ and the right groupoid action of $\cal M\rrTo X$ on the space $\cal M$ defined through the moment maps 
\[
X \stackrel{t}{\lTo} \cal M \stackrel{s}{\rTo} X
\]
and the formulas $[m,(h_2,g_2)]\cdot m'\colonequals \b_1(h_2)mm'$ and $m'\cdot m''\colonequals mm''$, respectively. Similarly, by taking $\bar{R}_{\frak G_2}$ to be the topological space $\cal M\Dprod_{\frak G_2}\cal O_{\frak G_2}$ acted upon on the left by $\cal M\rrTo X$ via the formula 
\[
m\cdot [m',(h_2,g_2)]\cdot [mm', (h_2^{\b_2(m)^{-1}},\b_2(m)g_2)], \ {\rm for \ } s(m)=t(m')=t(g_2)=s(h_2), 
\] 
and on the right by the groupoid $\cal M\Dprod_{\frak G_2}\cal O_{\frak G_2}\rrTo X$ by groupoid multiplication, we get a semi-exchanger. Now, the canonical projections 
\[
\cal M\underset{\cal M}{\bullet}(\cal M\Dprod_{\frak G_2}\cal O_{\frak G_2})\rTo \cal M\Dprod_{\frak G_2}\cal O_{\frak G_2}
\]
and 
\[
(\cal M\Dprod_{\frak G_2}\cal O_{\frak G_2})\underset{\cal M\Dprod_{\frak G_2}\cal O_{\frak G_2}}{\bullet}\cal M\rTo \cal M
\]
gives the desired morphisms of exchangers. Similar methods can easily be used to get $L_{\frak G_1}$ and its "weak inverse".  
\end{proof}



\end{document}